\documentclass[12pt]{article}
\usepackage[utf8]{inputenc}

\usepackage{amsmath}
\usepackage{amsfonts}
\usepackage{amssymb}
\usepackage{anysize}
\usepackage{cite}
\usepackage{comment}
\usepackage[shortlabels]{enumitem}
\usepackage{mathrsfs}
\usepackage{amsthm}
\usepackage{tikz}
    \usetikzlibrary{matrix}
    \usetikzlibrary{decorations.markings}
\usepackage{hyperref}
\usepackage[T1]{fontenc}
\usepackage{oldgerm}
\usepackage{scalerel,stackengine}
\usepackage{stmaryrd}
\usepackage{xcolor}
\usepackage{xfrac}
\usepackage{authblk}

\usepackage{framed}

\stackMath
\newcommand\reallywidehat[1]{%
\savestack{\tmpbox}{\stretchto{%
  \scaleto{%
    \scalerel*[\widthof{\ensuremath{#1}}]{\kern.1pt\mathchar"0362\kern.1pt}%
    {\rule{0ex}{\textheight}}%WIDTH-LIMITED CIRCUMFLEX
  }{\textheight}% 
}{2.1ex}}%
\stackon[-6.9pt]{#1}{\tmpbox}%
}
\newcommand\varfrak[1]{\mathord{\text{\textgoth{#1}}}}
\newcommand{\ad}{\operatorname{ad}}
\newcommand{\Ad}{\operatorname{Ad}}

\newtheorem{theorem}{Theorem}[section]
\newtheorem{definition}{Definition}[section]

\newtheorem{proposition}{Proposition}[section]
\newtheorem{lemma}{Lemma}[section]
\newtheorem{corollary}{Corollary}[section]
\newtheorem{remark}{Remark}[section]

\begin{document}

\title{Gauge reduction in covariant field theory}
%Covariant reduction in gauge theories

\author[1]{Marco Castrillón López\thanks{mcastri@mat.ucm.es}}

\author[2,3]{Álvaro Rodríguez Abella\thanks{alvrod06@ucm.es}}

\affil[1]{Facultad de Ciencias Matemáticas, Universidad Complutense de Madrid, Plaza de las Ciencias, 3, Madrid, 28040, Madrid, Spain}

\affil[2]{Department of Mathematics and Computer Science, Saint Louis University (Madrid Campus), Avenida del Valle, 34, Madrid, 28003, Madrid, Spain}

\affil[3]{Instituto de Ciencias Matem\'aticas (CSIC--UAM--UC3M--UCM), Calle Nicol\'as Cabrera, 13--15, Madrid, 28049, Spain}

\date{}

\maketitle

\small

\noindent\emph{2020 Mathematics Subject Classification:} 70S05, 70S10, 70S15 (Primary), 58E15, 53Z05  (Secondary)

\noindent \emph{Key words:} covariant reduction, Euler--Lagrange equations, gauge symmetry, generalized principal bundle, Lagrangian field theory, Noether theorem

\begin{abstract}
In this work, we develop a Lagrangian reduction theory for covariant field theories with gauge symmetries. These symmetries are modeled by a Lie group fiber bundle acting fiberwisely on a configuration bundle. In order to reduce the variational principle, we utilize generalized principal connections, a type of Ehresmann connections that are equivariant by the fiberwise action. After obtaining the reduced equations, we give the reconstruction condition and we relate the vertical reduced equation with the Noether theorem. Lastly, we illustrate the theory with several examples, including the classical case (Lagrange--Poincaré reduction), Electromagnetism, symmetry-breaking and non-Abelian gauge theories.
\end{abstract}

\normalsize

%\tableofcontents

\section{Introduction}

Reduction by symmetries has played a major role in Geometric Mechanics since its was first introduced with the its modern approach \cite{Ar1966,Sm1970,Me1973,MaWe1974}. The key idea is to use the symmetry group of a (Lagrangian or Hamiltonian) system to obtain a reduced set of equations on a space of lower dimension. In the Lagrangian setting, this reduced space is the quotient of the configuration manifold by the symmetry group, and the reduced equations come from a reduced variational principle. More precisely,  (see for example, \cite{MaSh1993,CeMaRa2001,CeMaRa2001b} and the references there in) reduction is performed in the tangent bundle $TQ$ of a configuration manifold $Q$ where a Lie group $G$ acts on. The key point is the determination of the geometry of $(TQ)/G$ as well as the nature of the projected variations to that quotient. 

This philosophy has been extended successfully to the realm of classical field theories \cite{CaMa2008,GaRa2014} and, in particular, to covariant field theories \cite{CaGaRa2001,CaRa2003,ElGaHoRa2011,CaGaRo2013}. Here, the configuration space is a fiber bundle $Y\to X$, the objects under study are sections of this bundle, and the phase space is the corresponding jet bundle $J^1Y\to Y$. The base $X$ can be space-time or, for non-dynamical problems, any manifold. 
 In this context, reduction is performed when there is a Lie group $G$ acting vertically on $Y$. This situation is sometimes known as global symmetries or global gauge symmetries. 
 
 Nevertheless, there is a wide variety of problems in Field Theories where symmetries are not global but \emph{local}, that is, the action of the symmetry group depends on the point $x\in X$ where it is evaluated. This is also known as local gauge (or simply gauge) symmetry and its geometric model is written in terms of a Lie group fiber bundle $\mathcal{G}\to X$ acting fiberwisely on the configuration bundle $Y\to X$. The main instances of this framework are pure gauge theories as Electromagnetism or Yang--Mills,  where gauge transformations are sections of $T^*X$ or the jet space of an adjoint bundle, respectively \cite{MaMa1992,Fo2012}. It is impossible to grant the appropriate importance to this kind of gauge theories both in Physics (as one of the foundations of the models for fields) and in Differential Geometry (as the basic tool for the constructions of a growing number of geometric and topological invariants).
 Interestingly, these gauge theories posses a particular kind of reduction result in the celebrated theorem of Utiyama \cite{Ut1956,Ga1977}, also generalized to the setting of interaction  \cite{Be1989}. The theorem proves that gauge invariant Lagrangian densities only depend on the curvature (force), instead of depending on the principal connection itself (potential). However, it is frustrating to see that the standard field theory reduction scheme (Lagrange--Poincaré reduction) stemmed from Geometric Mechanics does not cover gauge symmetries, a fact that has prevented the gauge case from taking advantage of the rich geometric interpretations and constructions that the powerful reduction program has provided so far (for example, geometric integrators just to mention one of them).

This article addresses the construction of a general Lagrangian reduction procedure for (covariant) field theories with local symmetries. The reduction of the variational principle relies on the use of the generalized principal connections (see \cite{CaRo2023} or \cite{Fischer2022}), that is, Ehresmann connections that are equivariant by the fiberwise action. Our theory of local symmetries extends the classical case investigated in \cite{ElGaHoRa2011} for global symmetries. In some sense, we can say that this is the culmination of the reduction program in its generality. 

We organise the paper as follows. In Section \ref{sec:preliminaries}, we recall the geometric tools needed to develop our theory: fibered actions of Lie group fiber bundles, generalized principal connections, and the natural connections arising on the corresponding quotient bundles. Section \ref{sec:reducedconfigurationspace} is devoted to studying the geometry of the quotient of the first jet of the configuration bundle by the fibered action. We have to note that, given a Lie group bundle $\mathcal{G}\to X$ acting on a bundle $Y\to X$, the interesting first order gauge Lagrangian densities are not invariant by the full jet bundle $J^1\mathcal{G}$ but only by an affine subbundle $H\subset J^1 \mathcal{G}$ projecting surjectively onto $\mathcal{G}$. This is the case of all the important examples one can find in the literature. This bundle $H$ can be thought as a non-holonomic constraint in the nature of the derivatives of the symmetries, a subbundle that is usually affine. The existence of this constraint is in the the core of the complexity of the theory we develop here, since invariance under the full jet bundle $J^1\mathcal{G}$ is in fact rather trivial (see examples in \S \ref{sec:examples}.2 below). Next, in Section \ref{sec:reducedvariationalprinciple} we compute the reduced variations and we apply it to the reduced Lagrangian, thus obtaining the main result of this work: the reduced field equations. After that, in Section \ref{sec:reconstruction} we investigate the reconstruction condition, that is, the additional equations that a solution of the reduced equations must satisfy to come from a solution of the original problem. This necessary condition is a characteristic trait of reduction field theories that is absent in Classical Mechanics and it is written as the flatness of a generalized connection. Furthermore, we relate part of the reduced equations (the so-called vertical equation) with Noether theorem in Section \ref{sec:noethertheorem} by showing that the conservation laws derived from the gauge symmetries become part of the reduced equations themselves. Finally, in Section \ref{sec:examples} we apply our theory in several contexts. We begin with classical reduction in field theories (global symmetries) obtained as a particular case of the main result. We follow with by reduction under the full jet symmetry. We continue with Electromagnetism in vacuum as well as $k$-form Electromagnetism. EM is also analyzed in the case of symmetry breaking of gauge theories when there is a product of groups. The last application works with gauge invariance in the Yang--Mills setting. We describe the geometry of the reduced space and, in particular, we get a new proof of Utiyama theorem from a reduction point of view. In addition we get the reduced equations that, together with the reconstruction condition, are the celebrated Yang--Mills equations for the appropriated Lagrangian density. In fact, the reduction technique of this article sheds light to an uncomfortable situation that Utiyama reduction does not fully explain: even though the reduced Lagrangian depends on curvature only, the equations still depend on both curvature and connection. Surprisingly, we prove that this is connected with the fact that the non-Abelian gauge symmetry is not a free action and the reduction is, in principle, singular. There is not, even in Mechanics, a well defined theory of Lagrangian singular reduction. The way to overcome this difficulty is to enlarge the phase bundle (from the bundle of connections $C(P)$ of a principal bundle $P\to X$ to its 1-jet bundle $J^1P$) before reduction. It turns out that the reduced phase bundle comprises both connections and curvature and hence the equations as well. But the elimination of that augmented geometry after reduction (in a sort of unreduction-reduction process) gives that the Lagrangian does not depend on the connection, without eliminating that dependence in the equations.

This article opens new interesting questions for future work. First, even though the main gauge examples we know have symmetries defined by an affine subbundle, we think that the comprehension of reduction by any subbundle $H\subset J^1\mathcal{G}$ is still a valuable topic. The generalization to higher order cases would also be of much interest, a questions that is connected with higher order gauge theories (for example, see \cite{CuMePo07}). Furthermore, it would be interesting to connect the reduction approach in this work with some constructions of curved gauge theories (\cite{Fischer2022, KoSt15}). The singular actions in these cases could be tackled again with the unreduction-reduction idea used for Yang--Mills in \S \ref{sec:examples}.6. Finally, the construction of variational integrators (see, for example, \cite{BlCoJi2019, BoLoSu1998, MaPeSh1999, MaPeSh2000} for Mechanics or  \cite{Va2007, ChFeGaRo19} for Field Theories) in gauge theories with reduction is a big goal that will require a careful analysis of the discrete analogs of the objects in the article.

In the following, every manifold or map is assumed to be smooth, meaning $C^\infty$, unless otherwise stated. In addition, every fiber bundle $\pi_{Y,X}: Y\rightarrow X$ is assumed to be locally trivial and is denoted by $\pi_{Y,X}$. Given $x\in X$, $Y_x=\pi_{Y,X}^{-1}(\{x\})$ denotes the fiber over $x$. The space of (smooth) global sections of $\pi_{Y,X}$ is denoted by $\Gamma(\pi_{Y,X})$. In particular, vector fields on a manifold $X$ are denoted by $\mathfrak X(X)=\Gamma(\pi_{TX,X})$, where $TX$ is the tangent bundle of $X$. Likewise, the space of local sections on an open set $\mathcal U\subset X$ is denoted by $\Gamma(\mathcal U,\pi_{Y,X})$. The tangent map of a map $f\in C^\infty(X,X')$ between the manifolds $X$ and $X'$ is denoted by $(df)_x:  T_xX\rightarrow T_{f(x)}X'$ for each $x\in X$. In the same vein, the pull-back of $\alpha\in\Omega^k(X')$ is denoted by $f^*\alpha\in\Omega^k(X)$ and its exterior derivative is denoted by ${\rm d}\alpha\in\Omega^{k+1}(X')$. When working in local coordinates, we assume the Einstein summation convention for repeated indices. A compact interval will be denoted by $I=[a,b]$.

%%%%%%%%%%%%%%%%%%%%
\section{Preliminaries}\label{sec:preliminaries}

This section is devoted to introducing the main geometric tools used in the forthcoming development. In particular, the theory of generalized principal bundles and connections is an essential point. We refer the reader to \cite{CaRo2023} (see also \cite{Fischer2022}) for a complete exposition of this topic.

%%%%%%%%%%
\subsection{Actions of Lie group bundles}

A \emph{Lie group fiber bundle} with typical fiber a Lie group $G$ is a fiber bundle $\pi_{\mathcal G,X}:\mathcal{G} \to X$ such that for any point $x\in X$ the fiber $\mathcal G_x$ is equipped with a Lie group structure and there is a neighborhood $\mathcal U\subset X$ and a diffeomorphism $\mathcal U\times G  \rightarrow \pi_{\mathcal G,X}^{-1}(\mathcal U)$ preserving the Lie group structure fiberwisely.

Note that the map $1: X\rightarrow\mathcal G$ that assigns the identity element $1_x\in\mathcal G_x$ to each $x\in X$ is a global section (called the \emph{unit section}) of $\pi_{\mathcal G,X}$. Any Lie group bundle defines a Lie algebra bundle $\pi_{\varfrak g,X}:\varfrak g\to X$ as the vector bundle whose fiber $\varfrak g_x$ at each $x\in X$ is the Lie algebra of $\mathcal G_x$. That is, $\varfrak g=1^*(V\mathcal G)$, where $V\mathcal G \subset T\mathcal G$ is the vertical bundle of $\pi_{\mathcal G,X}$, i.e. the kernel of $(\pi_{\mathcal G,X})_*$. We consider subgroups of Lie group bundles in the following sense.

\begin{definition}
A \emph{Lie group subbundle} of a Lie group bundle $\pi_{\mathcal{G},X}:\mathcal{G}\to X$ is a Lie group bundle $\pi_{\mathcal H,X}:\mathcal H\to X$ such that $\mathcal H$ is a submanifold of $\mathcal{G}$ and $\mathcal H_x$ is a Lie subgroup of $\mathcal{G}_x$ for each $x\in X$. It is said to be \emph{closed} if $\mathcal H_x$ is a closed Lie subgroup of $\mathcal G_x$ for every $x\in X$.
\end{definition}

Let $\pi_{Y,X}$ be a fiber bundle and $\pi_{\mathcal G,X}$ be a Lie group fiber bundle. We denote by $Y\times_X\mathcal G$ the corresponding fibered product, which is also a fiber bundle over $X$.

\begin{definition}
A \emph{right fibered action} of $\pi_{\mathcal G,X}$ on $\pi_{Y,X}$ is a bundle morphism $\Phi: Y\times_X\mathcal G\rightarrow Y$ covering the identity $\operatorname{id}_X$ such that $\Phi(y,hg)=\Phi(\Phi(y,h),g)$ and $\Phi(y,1_x)=y$, for all $(y,g),(y,h)\in Y\times_X\mathcal G$, $\pi_{\mathcal G,X}(y)=x$.
\end{definition}

\begin{remark}
The notion of a left action $\Phi :\mathcal{G}\times _X Y\to Y$ is as above but with the condition $\Phi (hg,y)=\Phi (h,\Phi (g,y))$. However, if it is not explicitly indicated, all action considered in this article will be right actions. The results for left action would be completely analogous.
\end{remark}
For the sake of simplicity, we will denote $\Phi(y,g)=y\cdot g$ and we will say that $\pi_{\mathcal G,X}$ acts fiberwisely on the right on $\pi_{Y,X}$. Note that it induces a right action on each fiber, $\Phi_x=\Phi|_{Y_x\times\mathcal G_x}: Y_x\times \mathcal G_x\rightarrow Y_x$. The fibered action is said to be \emph{free} if $y\cdot g=y$ for some $(y,g)\in Y\times_X\mathcal G$ implies that $g=1_x$, $x=\pi_{Y,X}(y)$. In the same way, it is said to be \emph{proper} if the bundle morphism $Y\times_X \mathcal G\ni (y,g)\mapsto(y,y\cdot g)\in  Y\times_X Y$ is proper. If $\Phi$ is free and proper, so is each action $\Phi_x$, since the fibers of a bundle are closed. 

As the fibered action is vertical (i.e. it covers the identity $\operatorname{id}_X$), we may regard the quotient space $Y/\mathcal G$ as the disjoint union of the quotients of the fibers by the induced actions, that is,
\begin{equation*}
Y/\mathcal G=\bigsqcup_{x\in X}Y_x/\mathcal G_x=\left\{[y]_{\mathcal G}=(x,[y]_{\mathcal G_x}): x\in X,y\in Y_x\right\},
\end{equation*}
The following diagram is commutative: 

\begin{equation*} \label{eq:diagramafibrados}
\begin{array}{cc}
\begin{tikzpicture}
\matrix (m) [matrix of math nodes,row sep=3em,column sep=3em,minimum width=2em]
{	Y & & X \\
	& Y/\mathcal G & \\};
\path[-stealth]
(m-1-1) edge [] node [above] {$\pi_{Y,X}$} (m-1-3)
(m-1-1) edge [] node [left] {$\pi_{Y,Y/\mathcal G}\,$} (m-2-2)
(m-2-2) edge [] node [right] {$\;\,\pi_{Y/\mathcal G,X}$} (m-1-3);
\end{tikzpicture}
& 
\begin{tikzpicture}
\matrix (m) [matrix of math nodes,row sep=3em,column sep=3em,minimum width=2em]
{	y & & x \\
	& \left[y\right]_{\mathcal G} & \\};
\path[-stealth]
(m-1-1) edge [|->,decoration={markings,mark=at position 1 with {\arrow[scale=1.7]{>}}},
    postaction={decorate},shorten >=0.4pt] node [above] {} (m-1-3)
(m-1-1) edge [|->,decoration={markings,mark=at position 1 with {\arrow[scale=1.7]{>}}},
    postaction={decorate},shorten >=0.4pt] node [left] {} (m-2-2)
(m-2-2) edge [|->,decoration={markings,mark=at position 1 with {\arrow[scale=1.7]{>}}},
    postaction={decorate},shorten >=0.4pt] node [right] {} (m-1-3);
\end{tikzpicture}
\end{array}
\end{equation*}

\begin{proposition}\label{prop:Y-Y/Gfibrado}
If $\pi_{\mathcal G,X}$ acts on $\pi_{Y,X}$ freely and properly, then $Y/\mathcal G$ admits a unique smooth structure such that 
\begin{enumerate}[(i)]
    \item $\pi_{Y,Y/\mathcal G}$ is a fiber bundle with typical fiber $G$.
    \item $\pi_{Y/\mathcal G,X}$ is a fibered manifold, i.e. a surjective submersion.
\end{enumerate}
\end{proposition}

If we fix $x\in X$, $y_0\in Y_x$ and $g_0\in \mathcal G_x$, we can consider the maps
\begin{align*}
\Phi_{y_0}:\mathcal G_x\to Y_x,\quad g\mapsto y_0\cdot g;\qquad\Phi_{g_0}:Y_x\to Y_x,\quad y\mapsto y\cdot g_0.
\end{align*}
In the same way, denote by $L_{g_0}: \mathcal G_x\rightarrow \mathcal G_x$ and $R_{g_0}: \mathcal G_x\rightarrow \mathcal G_x$ the left and right multiplication by $g_0\in \mathcal G_x$, respectively. \emph{Infinitesimal generators} (or \emph{fundamental fields}) are defined in the same fashion as in classical actions of Lie groups. Namely, for each $\xi$ belonging to the Lie algebra $\varfrak g_x$ of $\mathcal G _x$, then $\xi^*\in\mathfrak X(Y_x)$ is defined as
\begin{equation}\label{eq:definfinitesimalgenerator}
\xi^*_y=\left.\frac{d}{dt}\right|_{t=0} y\cdot\exp(t\xi)=(d\Phi _y)_{1_x}(\xi),\qquad y\in Y_x.
\end{equation}
Fundamental vector fields are $\pi_{Y,Y/\mathcal G}$-vertical, i.e. $\xi^*_y\in V_y Y$ for each $y\in Y_x$, where $VY=\ker{(\pi_{Y,Y/\mathcal G})_*}$ is the vertical bundle of $\pi_{Y,Y/\mathcal G}$. Of course, they are also $\pi_{Y,X}$-vertical. 

\begin{lemma}\label{lemma:isomorphismverticalbundle}
Let $\pi_{\varfrak g,X}$ be the Lie algebra bundle of $\pi_{\mathcal G,X}$. The following map is a vertical isomorphism of vector bundles over $Y$: 
\begin{align}\label{eq:isog}
Y\times_X\varfrak g\to VY,\quad(y,\xi)\mapsto\xi^*_y.
\end{align}
In addition, for any $(g,\xi)\in \mathcal G\times_X\varfrak g$, we have $(\Phi _g)_*(\xi^*)=\Ad_{g^{-1}}(\xi)^*$.
\end{lemma}

%%%%%%%%%%

\subsection{Lie group bundle connections}

Recall that an Ehresmann connection (for example see \cite{michor1993}) on a fiber bundle $\pi_{Z,X}: Z\to X$ is a fiber map $TZ \to VZ=\ker(\pi_{Z,X})_*$ such that its restriction to $VZ$ is the identity. Similarly, we can regard an Ehresmann connection as a distribution $HZ\subset TZ$ complementary to $VZ$. Finally, an Ehresmann connection is also a section of the jet bundle $\pi_{J^1Z,Z}: J^1 Z \to Z$. In the case where $\pi_{\mathcal G,X}$ is a Lie group bundle, an Ehresmann connection $\nu: T\mathcal G\to V\mathcal G$ can be also regarded as a linear bundle map (denoted by the same letter for the sake of simplicity)
$$\nu: T\mathcal G \to \varfrak g,\qquad U_g\mapsto \left(dR_{g^{-1}}\right)_g(\nu(U_g)). $$ 

\begin{definition}\label{def:liegroupconnection}
A \emph{Lie group bundle connection} on $\pi_{\mathcal G,X}$ is an Ehresmann connection $\nu: T\mathcal G\rightarrow\varfrak g$ satisfying
\begin{enumerate}[(i)]
    \item $\ker\nu_{1_x}=(d1)_x(T_x X)$ for each $x\in X$.
    \item For every $(g,h)\in \mathcal G\times_X\mathcal G$ and $(U_g,U_h)\in T_g \mathcal G\times_{T_x X}T_h \mathcal G$, $x=\pi_{\mathcal G,X}(g)$, then:
    \begin{equation*}
    \nu\left((dM)_{(g,h)}(U_g,U_h)\right)=\nu(U_g)+\Ad_g\left(\nu(U_h)\right),
    \end{equation*}
where $M:\mathcal G\times_X \mathcal G\rightarrow\mathcal G$ is the fiber multiplication map.
\end{enumerate}
\end{definition}

The geometric interpretation of Lie group bundle connections is provided by the following results. We denote by $^\nu\big|\big|$ the parallel transport associated to $\nu$ and by $Hor_g^\nu: T_x X\rightarrow T_g \mathcal G$ its horizontal lift at any $g\in\mathcal G$, $x=\pi_{\mathcal G,X}(g)$.

\begin{proposition}\label{prop:liegroupconnection}
Let $\nu$ be an Ehresmann connection on $\pi_{\mathcal G,X}$ such that $\ker\nu_{1_x}=(d1)_x(T_x X)$ for each $x\in X$. Then $\nu$ is a Lie group connection if and only if for any curve $x: I\rightarrow X$ we have
\begin{equation*}\label{eq:compatibilityofnu}
^\nu\big|\big|^{x(b)}_{x(a)}(gh)=\left( ^\nu\big|\big|^{x(b)}_{x(a)}g\right)\left( ^\nu\big|\big|^{x(b)}_{x(a)}h\right),\qquad g,h\in \mathcal G_{x(a)}.
\end{equation*}
Consequently,
\begin{equation*}\label{eq:inversatransportenu}
^\nu\big|\big|^{x(b)}_{x(a)}\,g^{-1}=\left(^\nu\big|\big|^{x(b)}_{x(a)}\,g\right)^{-1},\qquad^\nu\big|\big|^{x(b)}_{x(a)}\,1_{x(a)}=1_{x(b)}.
\end{equation*}
\end{proposition}

\begin{proposition}\label{prop:liegroupconnectionjet}
Let $\nu$ be an Ehresmann connection on $\pi_{\mathcal G,X}$ and consider the corresponding jet section $\hat{\nu}\in\Gamma(\pi_{J^1 \mathcal G,\mathcal G})$. Then $\nu$ is a Lie group bundle connection if and only if
\begin{enumerate}[(i)]
    \item $\hat\nu\circ 1=j^1 1=d1$
    \item $\hat{\nu}:\mathcal{G}\to J^1 \mathcal{G}$ is a Lie group bundle morphism with respect to the natural Lie group bundle structure of $J^1\mathcal{G}$, that is, $\hat{\nu}(gh)=\hat{\nu}(g)\,\hat{\nu}(h)$ for each $(g,h)\in \mathcal G\times_X\mathcal G$.
\end{enumerate}
\end{proposition}

Lie group connections induce linear connections on the corresponding Lie algebra (vector) bundle.

\begin{proposition}\label{prop:nablaginducida}
Let $x: I\rightarrow X$ be a smooth curve. Then the map
${^{\varfrak g}\big|\big|}_{x(a)}^{x(b)}:\varfrak g_{x(a)}\rightarrow\varfrak g_{x(b)}$ defined as
\begin{equation*}
{^{\varfrak g}\big|\big|}_{x(a)}^{x(b)}\xi=\left.\frac{d}{d\epsilon}\right|_{\epsilon=0}{^\nu\big|\big|}_{x(a)}^{x(b)}\exp(\epsilon\,\xi),\quad \xi\in\varfrak g_{x(a)}
\end{equation*}
is a linear parallel transport on $\pi_{\varfrak g,X}$.
\end{proposition}

Denoting by $\nabla^{\varfrak g}$ the linear connection corresponding to this parallel transport ${^{\varfrak g}\big|\big|}$, it can be checked that
\begin{equation}\label{eq:nablavarfrakg}
\nabla ^{\varfrak g}\xi =\left.\frac{d}{dt}\right|_{t=0}\nu\circ d\,\mathrm{exp}(t\xi),\qquad\xi\in\Gamma(\pi_{\varfrak g,X}).    
\end{equation}

%%%%%%%%%%
\subsection{Generalized principal connections}

Let $\pi _{Y,X}:Y\to X$ a fiber bundle on which a Lie group bundle $\pi_{\mathcal{G},X}:\mathcal{G}\to X$ acts freely and properly on the right. We denote by $\Phi :Y\times _X\mathcal{G}\to Y$ the fibered action.

\begin{definition}\label{def:connectionform}
Let $\nu$ be a Lie group bundle connection on the Lie group bundle on $\pi : \mathcal{G}\to X$. A \emph{generalized principal connection} on $\pi_{Y,Y/\mathcal G}$ associated to $\nu$ is a form with values in the Lie algebra bundle\footnote{
In fact, $\omega$ takes values on the vector bundle $\pi_{Y\times_X\varfrak g,Y}$, which is the pull-back of $\pi_{\varfrak g,X}$ by $\pi_{Y,X}$. Abusing the notation, we denote the pull-back bundle by the same symbol.
} $\omega\in\Omega^1(Y,\varfrak g)$ satisfying:

\begin{enumerate}[(i)]
    \item (Complementarity) $\omega_y(\xi^*_y)=\xi$ for every $(y,\xi)\in Y\times_X\varfrak g$.
    \item ($\Ad$-equivariance) For each $(y,g)\in Y\times_X \mathcal G$ and $(U_y,U_g)\in T_yY\times_{T_x X} T_g \mathcal G$, $x=\pi_{Y,X}(y)$, then:
\begin{equation*}
\omega_{y\cdot g}\left((d\Phi)_{(y,g)}(U_y,U_g)\right)=\Ad_{g^{-1}}\left(\omega_y(U_y)+\nu(U_g)\right).
\end{equation*}
\end{enumerate}
\end{definition}

We denote by $Hor^\omega_{y}: T_{[y]_{\mathcal G}}(Y/\mathcal G)\rightarrow T_y Y$ the \emph{horizontal lifting} given by $\omega$ at $y\in Y$. The next result gives a geometric interpretation of the above definition in terms of the parallel transports ${^\nu\big|}\big|$ and ${^\omega\big|}\big|$, in the same vein as Proposition \ref{prop:liegroupconnection} above.

\begin{proposition}\label{prop:compatibilityofomega}
Let $\nu$ be a Lie group bundle connection on $\pi_{\mathcal G,X}$ and $\omega\in\Omega^1(Y,\varfrak g)$ be an Ehresmann connection on $\pi_{Y,Y/\mathcal G}$. Then $\omega$ is a generalized principal connection associated to $\nu$ if and only if for any curve $\gamma: I\rightarrow Y/\mathcal G$, the corresponding parallel transports satisfy
\begin{equation*}\label{eq:compatibilityofomega}
{^\omega\big|}\big|^{\gamma(t)}_{\gamma(a)}(y\cdot g)=\left({^\omega\big|}\big|^{\gamma(t)}_{\gamma(a)}y\right)\cdot\left( {^\nu\big|}\big|^{x(t)}_{x(a)}g\right),\qquad g\in \mathcal G_{x(a)},~y\in Y_{\gamma(a)},~t\in I,
\end{equation*}
where $x=\pi_{Y/\mathcal G,X}\circ\gamma$.
\end{proposition}

The \emph{cuvature} of $\omega$ (see, for example \cite[\S 9.4]{michor1993}) is the 2-form $\Omega\in\Omega^2(Y,\varfrak g)$ defined as:
\begin{equation*}
\Omega(U_1,U_2)=-\omega\left(\left[U_1-\omega(U_1)^*,U_2-\omega(U_2)^*\right]\right),\qquad U_1,U_2\in\mathfrak X(Y).
\end{equation*}
The linear connection $\nabla^{\varfrak g}$ on $\pi_{\varfrak g,X}$ enables us to express the curvature as follows.

\begin{proposition}\label{prop:curvature}
Let ${\rm d}^{\varfrak g}$ be the exterior covariant derivative\footnote{
The \emph{exterior covariant derivative} of a linear connection $\nabla^E$ on a vector bundle $\pi_{E,X}$ is an operator in the family of $E$-valued forms on $X$, ${\rm d}^E:\Omega^k(X,E)\to \Omega^{k+1}(X,E)$. For a 1-form $\alpha\in\Omega^1(X,E)$ it is given by
\begin{equation*}
{\rm d}^E \alpha (U_1,U_2)=\nabla_{U_1}^E(\alpha(U_2))-\nabla_{U_2}^E(\alpha(U_1))-\alpha([U_1,U_2]),\qquad U_1,U_2\in\mathfrak X(X).
\end{equation*}
}
associated to $\nabla^{\varfrak g}$. Then
\begin{equation*}
\Omega\left(U_1,U_2\right)={\rm d}^{\varfrak g}\,\omega\left(U_1^h,U_2^h\right),\qquad U_1,U_2\in\mathfrak X(Y).
\end{equation*}
\end{proposition} 

The \emph{(generalized) adjoint bundle} of the action of $\pi_{Y,X}$ is defined to be the quotient $\tilde{\varfrak{g}} = (Y\times_X\varfrak{g})/\mathcal{G}$ by the (right) fibered action
\begin{equation}\label{eq:defadjunto}
\left(Y\times_X\varfrak g\right)\times_X\mathcal G\to Y\times_X\varfrak g,\quad\left((y,\xi),g\right)\mapsto\left(y\cdot g,\Ad_{g^{-1}}(\xi)\right),
\end{equation}
where $\Ad_g\in\operatorname{End}(\varfrak g_x)$ denotes the adjoint representation of $\mathcal G_x$, where $x=\pi_{\mathcal G,X}(g)$. It is a vector bundle over $Y/\mathcal{G}$ equipped with a Lie algebra bundle structure. As in the case of (standard) principal connections, it is possible to regard the curvature as a 2-form on the base space $Y/\mathcal G$ with values in $\tilde{\varfrak g}$. 

\begin{definition}\label{def:reducedcurvature}
The \emph{reduced curvature} of $\omega$ is the 2-form $\tilde\Omega\in\Omega^2\left(Y/\mathcal G,\tilde{\varfrak g}\right)$ given by
\begin{equation*}
\tilde\Omega_{[y]_{\mathcal G}}\left(U_1,U_2\right)=\left[y,\Omega_y\left(Hor_y^\omega(U_1),Hor_y^\omega(U_2)\right)\right]_{\mathcal G}
\end{equation*}
for each $[y]_{\mathcal G}\in Y/\mathcal G$ and $U_1,U_2\in T_{[y]_{\mathcal G}}(Y/\mathcal G)$, where $y\in Y$ is such that $\pi_{Y,Y/\mathcal G}(y)=[y]_{\mathcal G}$.
\end{definition}

The reduced curvature is well-defined, i.e. it does not depend on the choice of $y\in Y$. Indeed, let $g\in\mathcal G_x$, where $x=\pi_{Y,X}(y)$, $u_i=(d\pi_{Y/\mathcal G,X})_{[y]_{\mathcal G}}(U_i)\in T_x X$ for $i=1,2$ and $\gamma\in\Gamma(\pi_{\mathcal G,X})$ be such that $\gamma(x)=g$ and $\nu_{\gamma(x)}\circ(d\gamma)_x=0$. Using \cite[Proposition 3.10]{CaRo2023} we obtain
\begin{equation*}
Hor_{y\cdot g}^\omega(U_i)=(d\Phi)_{(y,g)}\left(Hor_y^\omega(U_i),(d\gamma)_x(u_i)\right),\qquad i=1,2.
\end{equation*}
Hence, we have
\begin{equation}\label{eq:reducedcurvature}
\begin{array}{l}
\left[y\cdot g,\Omega_{y\cdot g}\left(Hor_{y\cdot g}^\omega(U_1),Hor_{y\cdot g}^\omega(U_2)\right)\right]_{\mathcal G}\quad=\quad\left[y\cdot g,-\omega_{y\cdot g}\left(\left[Hor_{y\cdot g}^\omega(U_1),Hor_{y\cdot g}^\omega(U_2)\right]\right)\right]_{\mathcal G}\vspace{0.1cm}\\
\quad\begin{array}{cl}
= & \left[y\cdot g,-\omega_{y\cdot g}\left(\left[(d\Phi)_{(y,g)}\left(Hor_y^\omega(U_1),(d\gamma)_x(u_1)\right),(d\Phi)_{(y,g)}\left(Hor_y^\omega(U_2),(d\gamma)_x(u_2)\right)\right]\right)\right]_{\mathcal G}\vspace{0.1cm}\\
\overset{(\star)}{=} & \left[y\cdot g,-\omega_{y\cdot g}\left((d\Phi)_{(y,g)}\left(\left[Hor_y^\omega(U_1),Hor_y^\omega(U_2)\right],[(d\gamma)_x(u_1),(d\gamma)_x(u_2)]\right)\right)\right]_{\mathcal G}\vspace{0.1cm}\\
= & \left[y\cdot g,-\omega_{y\cdot g}\left((d\Phi)_{(y,g)}\left(\left[Hor_y^\omega(U_1),Hor_y^\omega(U_2)\right],(d\gamma)_x\left([u_1,u_2]\right)\right)\right)\right]_{\mathcal G}\vspace{0.1cm}\\
= & \left[y\cdot g,-\Ad_{g^{-1}}\left(\omega_y\left(\left[Hor_y^\omega(U_1),Hor_y^\omega(U_2)\right]\right)\right)\right]_{\mathcal G}\vspace{0.1cm}\\
= & \left[y,-\omega_y\left(\left[Hor_y^\omega(U_1),Hor_y^\omega(U_2)\right]\right)\right]_{\mathcal G}\vspace{0.1cm}\\
= & \left[y,\Omega_y\left(Hor_y^\omega(U_1),Hor_y^\omega(U_2)\right)\right]_{\mathcal G},
\end{array}
\end{array}
\end{equation}
where we have used that $\left[(d\gamma)_x(u_1),(d\gamma)_x(u_2)\right]=(d\gamma)_x\left([u_1,u_2]\right)$.

%%%%%%%%%%%%%%%%%%%%
\section{Geometry of the reduced configuration space}\label{sec:reducedconfigurationspace}

A Lie group bundle connection, regarded as a section $\hat{\nu} :\mathcal{G}\to J^1\mathcal{G}$, provides an identification of the affine bundle $J^1\mathcal{G} \to \mathcal{G}$ with its model vector bundle $T^*X\otimes V\mathcal{G} \simeq \mathcal{G}\times _X (T^*X\otimes \varfrak{g})$. If the Lie group bundle connection is regarded as a 1-form $\nu$ taking values in $\varfrak{g}$, the identification is given explicitly as
\begin{align}\label{eq:Thetanu}
\Theta_\nu:J^1\mathcal G\to\mathcal G\times_X(T^*X\otimes\varfrak g),\quad j^1_x\eta\mapsto\left(\eta(x),\nu\circ(d\eta)_x\right).
\end{align}
We can transfer the Lie group bundle structure of $J^1 \mathcal G\to X$ (for example, cf. \cite[\S3, Theorem 1]{Fo2012}) via the above identification, as
\begin{equation*}
(g,\xi_x)\,(h,\eta_x)=(gh,\xi_x+\Ad_g\circ\eta_x),\qquad (g,\xi_x),(h,\eta_x)\in\mathcal G_x\times(T_x^*X\otimes\varfrak g_x),~x\in X.
\end{equation*}
Note that the identity element is $(1_x,0_x)\in\mathcal G\times_X(T^*X\otimes\varfrak g)$. 

\noindent {\bf Assumption}. Henceforth, let $\pi_{Y,X}$ be a fiber bundle on which $\pi_{\mathcal G,X}$ acts fiberwisely, freely and properly, and $\pi_{H,X}$ be a closed Lie group subbundle $H\subset J^1\mathcal G$ of $\pi_{J^1 \mathcal G,X}$ such that $\pi_{J^1 \mathcal G,\mathcal G}(H)=\mathcal G$. Besides, we suppose that $\pi_{H,\mathcal G}$ is an affine subbundle of $\pi_{J^1\mathcal G,\mathcal G}$. Identification \eqref{eq:Thetanu} enables us to regard 
\[
\Theta_\nu(H)\subset \mathcal G\times_X(T^*X\otimes \varfrak g),
\]
We denote $\Theta_\nu(H)$ by the same symbol, $H$, since the context will clarify the distinction between the two objects. Then $H\subset\mathcal G\times_X(T^*X\otimes\varfrak g)$ is both a closed Lie group subbundle over $X$ and an affine subbundle over $\mathcal G$.

\begin{proposition}\label{ggg}
Let $H\subset J^1\mathcal G$ be a closed Lie group subbundle of $\pi_{J^1\mathcal G,X}$ and an affine subbundle of $\pi_{J^1\mathcal G,\mathcal G}$. Let $\varphi:\mathcal G\to T^*X\otimes\varfrak g$ be a bundle morphism such that $(g,\varphi(g))\in H$ for each $g\in\mathcal G$ and\footnote{
Observe that $(1_x,0_x)\in H_{1_x}$, since $H\subset\mathcal G\times_X(T^*X\otimes\varfrak g)$ is a Lie group subbundle (over $X$).} $\varphi(1_x)=0_x$ for each $x\in X$. Taking into account the identification \eqref{eq:Thetanu}  there exists a vector subbundle $\mathfrak{H}\subset T^*X\otimes\varfrak g$ such that
\begin{equation*}
H=\left\{(g,\hat\eta_x)\in\mathcal G\times_X(T^*X\otimes\varfrak g)\mid\hat\eta_x=\varphi(g)+\eta_x,~\eta_x\in\mathfrak{H}_x\right\}.
\end{equation*}
Furthermore, $\mathfrak{H}$ is $\Ad$-invariant, i.e., $\Ad_g(\mathfrak{H}_x)\subset\mathfrak{H}_x$ for each $x\in X$ and $g\in\mathcal G_x$.
\end{proposition}

\begin{proof}
Let $g\in\mathcal G$ and $x=\pi_{\mathcal G,X}(g)\in X$. Since $H_g\subset\{g\}\times(T_x^*X\otimes\varfrak g_x)$ is an affine subspace, we may write $H_g=\{g\}\times(\varphi(g)+\mathfrak{H}_g)$ for certain vector subspace $\mathfrak{H}_g\subset T_x^*X\otimes\varfrak g_x$. In addition, $H_x\subset\mathcal G_x\times(T_x^*X\otimes\varfrak g_x)$ is a Lie subgroup, whence
\begin{equation*}
(1_x,\eta_x)\cdot(g,\varphi(g))=(g,\varphi(g)+\eta_x)\in H_g,\qquad\eta_x\in\mathfrak{H}_{1_x}.
\end{equation*}
Hence, $\mathfrak{H}_g=\mathfrak{H}_{1_x}$ and we denote $\mathfrak{H}_x\equiv\mathfrak{H}_{1_x}$. We thus define $\mathfrak{H}=\bigsqcup_{x\in X}\mathfrak{H}_x\to X$.

For the second part, observe that for each $\eta_x\in\mathfrak{H}_x$ and $g\in\mathcal G_x$ we have
\begin{equation*}
(g,\varphi(g))(1_x,\eta_x)=(g,\varphi(g)+\Ad_g\circ\eta_x)\in H_g.
\end{equation*}
\end{proof}

Note that, as a vector subbundle, we may consider the corresponding quotient, $(T^*X\otimes\varfrak g)/\mathfrak{H}$, which is again a vector bundle over $X$. We denote the corresponding elements by $[\xi_x]_{\mathfrak{H}}$.

\begin{lemma}\label{lemma:classvarphi}
The map $\varphi:\mathcal G\to T^*X\otimes\varfrak g$ satisfies
\begin{equation*}
\left[\varphi(gh)\right]_{\mathfrak{H}}=\left[\varphi(g)+\Ad_g\circ\varphi(h)\right]_{\mathfrak{H}},\qquad g,h\in\mathcal G_x,~x\in X.
\end{equation*}
\end{lemma}

\begin{proof}
We have that
\begin{equation*}
(g,\varphi(g))\cdot(h,\varphi(h))=(gh,\varphi(g)+\Ad_g\circ\varphi(h))\in H_{gh}.
\end{equation*}
Hence, $\varphi(g)+\Ad_d\circ\varphi(h)-\varphi(gh)\in\mathfrak{H}_x$ and we conclude.
\end{proof}

Now we define the following map
\begin{equation}\label{eq:affineaction}
\begin{array}{rccc}
\Phi_{\varphi}: & \left(Y\times_X(T^*X\otimes\varfrak g)/\mathfrak{H}\right)\times_X\mathcal G & \to & Y\times_X(T^*X\otimes\varfrak g)/\mathfrak{H}\\
& \left((y,[\xi_x]_{\mathfrak{H}}),g\right) & \mapsto & \left(y\cdot g,[\Ad_{g^{-1}}\circ(\xi_x+\varphi(g))]_{\mathfrak{H}}\right)
\end{array}
\end{equation}

\begin{remark} If the action of $\mathcal{G}$ on $Y$ is left, then the map \eqref{eq:affineaction} should be
\begin{equation}\label{eq:affineaction2}
\begin{array}{rccc}
\Phi_{\varphi}: & \mathcal G \times _X\left(Y\times_X(T^*X\otimes\varfrak g)/\mathfrak{H}\right) & \to & Y\times_X(T^*X\otimes\varfrak g)/\mathfrak{H}\\
& \left(g,(y,[\xi_x]_{\mathfrak{H}})\right) & \mapsto & \left(g\cdot y,[\Ad_{g}\circ\xi_x+\varphi(g)]_{\mathfrak{H}}\right)
\end{array}
\end{equation}
\end{remark}

Thanks to the  Ad-invariance of $\mathfrak{H}$, the map \eqref{eq:affineaction} is well-defined. Moreover, a straightforward computation shows that it is a (free and proper) right fibered action. As a result, we may consider the corresponding quotient, which is a fiber bundle over $Y/\mathcal G$:
\begin{equation}\label{eq:cocienteafin}
(T^*X\otimes\varfrak{g})_{\varphi,\mathfrak{H}}=\frac{Y\times_X(T^*X\otimes\varfrak g)/\mathfrak{H}}{\Phi_{\varphi}}.
\end{equation} 
The elements of this quotient are denoted by $\llbracket y,\xi_x\rrbracket_{\mathfrak{H},\varphi}$, where $y\in Y_x$, $\xi_x\in (T^*X\otimes \varfrak g)_x$ and $x\in X$.Analogously, we have the fibered action given by \eqref{eq:affineaction} with $\varphi=0$, which yields the quotient 
\begin{equation*}\label{eq:cocientelineal}
(T^*X\otimes\varfrak{g})_{0,\mathfrak{H}}=\frac{Y\times_X(T^*X\otimes\varfrak g)/\mathfrak{H}}{\Phi_0}
\end{equation*} 
\begin{proposition}\label{prop:vectormodelaffine}
In the above conditions, the following statements are true:
\begin{enumerate}[(i)]
    \item $(T^*X\otimes\varfrak{g})_{0,\mathfrak{H}}\to Y/\mathcal G$ is a vector bundle.
    \item $(T^*X\otimes\varfrak{g})_{\varphi,\mathfrak{H}}\to Y/\mathcal G$ is an affine bundle modelled on $(T^*X\otimes\varfrak{g})_{0,\mathfrak{H}}\to Y/\mathcal G$.
\end{enumerate}
\end{proposition}

\begin{proof}
For $(i)$, it is easy to check that the following operations are well-defined
\begin{equation*}
\llbracket y,\xi_x\rrbracket_{\mathfrak{H},0}+\llbracket y,\zeta_x\rrbracket_{\mathfrak{H},0}=\llbracket y,\xi_x+\zeta_x\rrbracket_{\mathfrak{H},0},\quad\lambda\llbracket y,\xi_x\rrbracket_{\mathfrak{H},0}=\llbracket y,\lambda\xi_x\rrbracket_{\mathfrak{H},0},
\end{equation*}
for each $\llbracket y,\xi_x\rrbracket_{\mathfrak{H},0},\llbracket y,\zeta_x\rrbracket_{\mathfrak{H},0}\in (T^*X\otimes\varfrak{g})_{0,\mathfrak{H}}$ and $\lambda\in\mathbb R$. Likewise, for $(ii)$, the following map is well-defined
\begin{equation*}
\llbracket y,\xi_x\rrbracket_{\mathfrak{H},\varphi}+\llbracket y,\zeta_x\rrbracket_{\mathfrak{H},0}=\llbracket y,\xi_x+\zeta_x\rrbracket_{\mathfrak{H},\varphi},
\end{equation*}
for each $\llbracket y,\xi_x\rrbracket_{\mathfrak{H},\varphi}\in (T^*X\otimes\varfrak{g})_{\varphi,\mathfrak{H}}$ and $\llbracket y,\zeta_x\rrbracket_{\mathfrak{H},0}\in (T^*X\otimes\varfrak{g})_{0,\mathfrak{H}}$.
\end{proof}

\begin{remark}
Observe that, given a Lie group subbundle $H\subset J^1\mathcal G$, the map $\varphi:\mathcal G\to T^*X\otimes\varfrak g$ is not unique. Furthermore, $H\subset\mathcal G\times_X(T^*X\otimes\varfrak g)$ is a vector subbundle over $\mathcal G$ (that is, $(T^*X\otimes\varfrak{g})_{\varphi,\mathfrak{H}}=(T^*X\otimes\varfrak{g})_{0,\mathfrak{H}}$) if and only we can choose the map $\varphi=0$. This is equivalent to the condition $\hat\nu(\mathcal G)\subset H$ when we regard $H$ as a Lie group subbundle of $J^1\mathcal{G}$.

Event though the case $\varphi =0$ is very relevant in examples, in the following we will assume that $\varphi$ does not necessarily vanish.
\end{remark}

Analogous to the adjoint bundle defined in \eqref{eq:defadjunto}, we can consider the tensor product $T^*X\otimes \tilde{\varfrak{g}}$ as the quotient 
\begin{equation*}
	T^*X\otimes\tilde{\varfrak g}=\frac{Y\times_X(T^*X\otimes\varfrak g)}{\mathcal G}\to Y/\mathcal G
\end{equation*}
by letting $\mathcal{G}$ act trivially on $T^*X$. As the vector subbundle $\mathfrak{H}\subset T^*X\otimes\varfrak g$ obtained in Proposition \ref{ggg} is $\Ad$-equivariant, the action of $\mathcal{G}$ on $Y\times_X(T^*X\otimes\varfrak g)$ restricts to $Y\times _X\mathfrak{H}$ and we may consider the corresponding quotient  
\begin{equation}
	\label{h tilde}
	\tilde{\mathfrak{H}}=\frac{Y\times_X\mathfrak{H}}{\mathcal G}\subset T^*X\otimes\tilde{\varfrak g}\to Y/\mathcal G
\end{equation}
In addition, $\pi_{\tilde{\mathfrak{H}},Y/\mathcal G}$ is a vector subbundle of $\pi_{T^*X\otimes\tilde{\varfrak g},Y/\mathcal G}$, so it makes sense to consider the corresponding quotient vector bundle, $\left.\left(T^*X\otimes \tilde{\varfrak g}\right)\right/\tilde{\mathfrak{H}}$, whose elements are denoted by $\llbracket y,\xi_x\rrbracket_{\tilde{\mathfrak{H}}}$, where $y\in Y_x$, $\xi_x\in (T^*X\otimes \varfrak g)_x$ and $x\in X$. Moreover, it is easily seen to be isomorphic to $(T^*X\otimes\varfrak{g})_{0,\mathfrak{H}}$.

\begin{lemma}\label{lemma:igualdad}
In the above conditions, we have the following isomorphism of vector bundles over $Y/\mathcal G$,
\begin{equation*}
	\frac{T^*X\otimes \tilde{\varfrak g}}{\tilde{\mathfrak{H}}}\simeq\frac{Y\times_X(T^*X\otimes\varfrak g)/\mathfrak{H}}{\Phi_0}=(T^*X\otimes\varfrak{g})_{0,\mathfrak{H}}.
\end{equation*} 
\end{lemma}

On the other hand, the first jet extension of the fibered action $\Phi$ is a right fibered action of $\pi_{J^1 \mathcal G,X}$ on $\pi_{J^1 Y,X}$
\begin{equation}\label{eq:accionjet}
\Phi^{(1)}:J^1 Y\times_X J^1 \mathcal G\to J^1 Y,\quad\left(j^1_x s,j^1_x\gamma\right)\mapsto j^1_x(\Phi(s,\gamma)).
\end{equation}
where, for (local) sections $s\in\Gamma(\pi_{Y,X})$ and $\gamma\in\Gamma(\pi_{\mathcal G,X})$, we have the (local) section  $\Phi(s,\gamma)(x)=\Phi(s(x),\gamma(x))$, $x\in X$. If we regard 1-jets as tangent maps of sections at a point, that is, $j_x^1s\equiv (ds)_x$ and $j_x^1\gamma\equiv(d\gamma)_x$, we can regard $\Phi ^1$ as
\begin{equation*}
\Phi^{(1)}(j^1_x s,j^1_x\gamma)=(d\Phi)_{\left(s(x),\gamma(x)\right)}\circ\left(ds_x,d\gamma_x\right): T_x X\to T_{s(x)\cdot\gamma(x)} Y.
\end{equation*}

The following result studies the geometry of the quotient of $\pi_{J^1 Y,X}$ by the fibered action \eqref{eq:accionjet} when restricted to the Lie group subbundle $H\subset J^1\mathcal G$.

\begin{theorem}\label{theorem:descomposicioncociente}
Let $\omega\in\Omega^1(Y,\varfrak g)$ be a generalized principal connection on $\pi_{Y,Y/\mathcal G}$ associated to a Lie group connection $\nu$ on $\pi_{\mathcal G,X}$ and consider a Lie group subbundle $H\subset J^1\mathcal G$ as in Proposition \ref{ggg}. Then the following map is a bundle isomorphism over $Y/\mathcal G$:
\begin{equation}
\begin{array}{cccc}
\label{ISO}
J^1 Y/H & \to & J^1(Y/\mathcal G)\times_{Y/\mathcal G}(T^*X\otimes\varfrak{g})_{\varphi,\mathfrak{H}} & \vspace{0.1cm}\\
\left[ j^1_x s\right]_H & \mapsto & \left(j^1_x\sigma_s,\left\llbracket s(x),\omega_{s(x)}\circ (ds)_x\right\rrbracket_{\mathfrak{H},\varphi}\right),
\end{array}
\end{equation}
where $\sigma_s=[s]_\mathcal G=\pi_{Y,Y/\mathcal G}\circ s\in\Gamma(\pi_{Y/\mathcal G,X})$.
\end{theorem}

\begin{proof}
For $j^1_x s\in J^1 Y$ and $j^1_x\eta\in H$, let $s$ and $\eta$ be (local) sections defining those jet elements. On one hand, it is clear that $\sigma_{\Phi(s,\eta)}=\sigma_s$. On the other, by \eqref{eq:Thetanu} we have $j_x^1\eta\simeq\left(\eta(x),\nu\circ(d\eta)_x\right)\in H$, whence $\nu\circ(d\eta)_x=\varphi(\eta(x))+\eta_x$ for some $\eta_x\in\mathfrak{H}_x$, by Proposition \ref{ggg}. Therefore, 
\begin{align*}
& \left\llbracket\Phi(s,\eta)(x),\omega_{\Phi(s,\eta)(x)}\circ(d\Phi(s,\eta))_x\right\rrbracket_{\mathfrak{H},\varphi}\vspace{0.1cm}\\
& \hspace{15mm} =\left\llbracket s(x)\cdot\eta(x),\omega_{s(x)\cdot\eta(x)}\circ(d\Phi)_{(s(x),\eta(x))}\circ\left((ds)_x,(d\eta)_x\right)\right\rrbracket_{\mathfrak{H},\varphi}\vspace{0.1cm}\\
& \hspace{15mm} =\left\llbracket s(x)\cdot\eta(x),\Ad_{\eta(x)^{-1}}\left(\omega_{s(x)}\circ(ds)_x+\nu\circ (d\eta)_x\right)\right\rrbracket_{\mathfrak{H},\varphi}\vspace{0.1cm}\\
& \hspace{15mm} =\left\llbracket s(x),\omega_{s(x)}\circ(ds)_x+\eta_x\right\rrbracket_{\mathfrak{H},\varphi}\vspace{0.1cm}\\
& \hspace{15mm} =\left\llbracket s(x),\omega_{s(x)}\circ(ds)_x\right\rrbracket_{\mathfrak{H},\varphi}.
\end{align*}
Subsequently, the map is well-defined. A straightforward computation shows that the inverse map is
\begin{equation*}
J^1(Y/\mathcal G)\times_{Y/\mathcal G}(T^*X\otimes\varfrak{g})_{\varphi,\mathfrak{H}}\to J^1 Y/H,\quad\left(j^1_x\sigma,\left\llbracket y,\xi_x\right\rrbracket_{\mathfrak{H},\varphi}\right)\mapsto\left[Hor^\omega_y\circ(d\sigma)_x+(\xi_x)^*_y\right]_H,
\end{equation*}
where $Hor^\omega_y: T_{[y]_{\mathcal G}}(Y/\mathcal G)\rightarrow T_y Y$ is the horizontal lifting given by $\omega$ at $y\in Y$. Observe that it is well-defined thanks to \cite[Lemma 3.1, Proposition 3.10]{CaRo2023}.

\end{proof}

%%%%%%%%%%
\subsection{Connections on the reduced spaces}\label{sec:inducedconnection}

Let $\omega$ be a generalized principal connection on $\pi_{Y,Y/\mathcal G}$ associated to a Lie group connection $\nu$ on $\pi_{\mathcal G,X}$ and $\nabla^X$ be a linear connection on $\pi_{TX,X}$. These connections induce a linear connection on the vector bundle $\pi_{\tilde{\varfrak g},Y/\mathcal G}$, and an affine connection on the affine bundle $\pi_{(T^*X\otimes\varfrak{g})_{\varphi,\mathfrak{H}},Y/\mathcal G}$ as follows. Consider the linear connection $\nabla^{\varfrak g}$ on $\pi_{\varfrak g,X}$ induced by $\nu$ (cf. Proposition \ref{prop:nablaginducida}). As above, denote by $^\nu\big|\big|$, $^{\varfrak g}\big|\big|$ and $^\omega\big|\big|$ the corresponding parallel transports. It is easy to check from the definition that the connections $\nu$ and $\nabla^{\varfrak g}$ satisfy the following compatibility relation: 
\begin{equation}\label{eq:compatibilityofnablag}
{^{\varfrak g}\big|\big|}_{x(a)}^{x(b)}\Ad_g(\xi)= \Ad_{{^\nu||}_{x(a)}^{x(b)}g}\left({^{\varfrak g}\big|\big|}_{x(a)}^{x(b)}\xi\right),\qquad g\in\mathcal G_{x(a)},~\xi\in\varfrak g_{x(a)}
\end{equation}
for every curve $x: I\rightarrow X$. This and Proposition \ref{prop:compatibilityofomega} ensure that the following parallel transport is well-defined.

\begin{proposition}\label{prop:nablatildevarfarkg}
The assignment sending any curve $\gamma: I\to Y/\mathcal G$ to the map
\begin{equation*}
{^{\tilde{\varfrak g}}\big|\big|}_{\gamma(a)}^{\gamma(b)}:\tilde{\varfrak g}_{\gamma(a)}\to\tilde{\varfrak g}_{\gamma(b)},\quad
\left[y,\xi\right]_{\mathcal G}\mapsto\left[{^\omega\big|\big|}_{\gamma(a)}^{\gamma(b)}y,{^{\varfrak g}\big|\big|}_{x(a)}^{x(b)}\xi\right]_{\mathcal G},
\end{equation*}
is a linear parallel transport on $\pi_{\tilde{\varfrak g},Y/\mathcal G}$, where $x=\pi_{Y/\mathcal G,X}\circ\gamma$.
\end{proposition} 

We denote by $\nabla^{\tilde{\varfrak g}}$ the corresponding linear connection on $\pi_{\tilde{\varfrak g},Y/\mathcal G}$ and by $\nabla^{\tilde{\varfrak g}}/dt$ the corresponding covariant derivative.

\begin{lemma}
Let $\nabla^{\mathfrak{H}}$ be a linear connection on $\pi_{\mathfrak{H},X}$. The assignment sending a curve $x: I\rightarrow X$ to the map
\begin{equation*}
{^H\big|\big|}_{x(a)}^{x(b)}: H_{x(a)}\rightarrow H_{x(b)},\quad{^{ H}\big|\big|}_{x(a)}^{x(b)}(g,\varphi(g)+\eta_{x(a)})=\left({^\nu\big|\big|}_{x(a)}^{x(b)}g,\varphi\left({^\nu\big|\big|}_{x(a)}^{x(b)}g\right)+{^{\mathfrak{H}}\big|\big|}_{x(a)}^{x(b)}\eta_{x(a)}\right)
\end{equation*}
for each $(g,\varphi(g)+\eta_{x(a)})\in H_{x(a)}$ is a parallel transport on $\pi_{H,X}$.
\end{lemma}

In the same vein, a parallel transport may be defined on $\overline H=\mathcal G\times_X\mathfrak{H}\to X$ as follows,
\begin{equation*}
{^{\overline H}\big|\big|}_{x(a)}^{x(b)}:\overline H_{x(a)}\rightarrow\overline H_{x(b)},\quad{^{\overline H}\big|\big|}_{x(a)}^{x(b)}(g,\eta_{x(a)})=\left({^\nu\big|\big|}_{x(a)}^{x(b)}g,{^{\mathfrak{H}}\big|\big|}_{x(a)}^{x(b)}\eta_{x(a)}\right).
\end{equation*}

Now we extend $\nabla^{\mathfrak{H}}$ to a linear connection $\nabla^\otimes$ on $T^*X\otimes\varfrak g$ that is \emph{compatible with $\nu$}, i.e., the relation \eqref{eq:compatibilityofnablag} holds for $^\otimes\big|\big|$ instead of $^{\varfrak g}\big|\big|$. Similarly, $\nabla^\otimes$ is said to be \emph{compatible with $\varphi$} if it satisfies
    \begin{equation}\label{eq:compatibilityofvarphi}
{^\otimes\big|\big|}^{x(b)}_{x(a)}\,\varphi(g)=\varphi\left({^\nu\big|\big|}^{x(b)}_{x(a)}\,g\right),\qquad g\in\mathcal G_{x(a)},
\end{equation}
for each curve $x:I\to X$. Henceforth, we assume that $\nabla^\otimes$ is compatible with $\nu$ and $\varphi$. We are ready to define a linear connection on the reduced bundle.

\begin{proposition}\label{prop:inducedconnection}
The assignment that each curve $\gamma: I\rightarrow Y/\mathcal G$ corresponds to the map
\begin{equation*}
\big|\big|^{\gamma(b)}_{\gamma(a)}:\left((T^*X\otimes\varfrak{g})_{\varphi,\mathfrak{H}}\right)_{\gamma(a)}\to\left((T^*X\otimes\varfrak{g})_{\varphi,\mathfrak{H}}\right)_{\gamma(b)},\quad
\llbracket y,\xi_x\rrbracket_{\mathfrak{H},\varphi}\mapsto\left\llbracket^\omega\big|\big|^{\gamma(b)}_{\gamma(a)}\,y,{^\otimes\big|\big|}^{x(b)}_{x(a)}\,\xi_x\right\rrbracket_{\mathfrak{H},\varphi},
\end{equation*}
where $x=\pi_{Y/\mathcal G,X}\circ\gamma$, is an affine parallel transport on $\pi_{(T^*X\otimes\varfrak{g})_{\varphi,\mathfrak{H}},Y/\mathcal G}$
\end{proposition}

Proposition \ref{prop:compatibilityofomega}, Equation \eqref{eq:compatibilityofnablag}, together with the $\Ad$-invariance of $\mathfrak{H}$ and \eqref{eq:compatibilityofvarphi}, guarantee that it is well-defined. Besides, it is affine thanks to the linearity of $\nabla^\otimes$. We denote by $\nabla$ the affine connection on $\pi_{(T^*X\otimes\varfrak{g})_{\varphi,\mathfrak{H}},Y/\mathcal G}$ associated to $\big|\big|$ and by $\nabla/dt$ the corresponding covariant derivative.

%%%%%%%%%%%%%%%%%%%%
\section{Reduction of the variational principle}\label{sec:reducedvariationalprinciple}

%%%%%%%%%%
\subsection{Calculus of variations and reduced Lagrangian}

Let $\pi_{Y,X}: Y\to X$ be a fiber bundle. A \emph{(first order) Lagrangian density} on $\pi_{Y,X}$ is a bundle morphism 
\begin{equation*}
\textstyle\mathfrak L: J^1 Y\to \bigwedge^n T^*X
\end{equation*}
covering the identity on $X$, where $n=\dim X$. Assuming that $X$ is orientable and $v\in\Omega^n(X)$ is a volume form, we can write $\mathfrak L=Lv$ for certain $L: J^1 Y\rightarrow\mathbb R$ called \emph{Lagrangian}. 

Henceforth, we suppose that $X$ is compact for simplicity. The \emph{action functional} defined by $\mathfrak L$ is
\begin{equation*}
\mathbb S(s)=\int_X\,\mathfrak L\left(j^1s\right),\qquad s\in\Gamma\left(\pi_{Y,X}\right)
\end{equation*}

A \emph{variation} of a section $s\in\Gamma\left(\pi_{Y,X}\right)$ is a smooth 1-parameter family 
\begin{equation*}
\{s_t\}=\left\{s_t\in\Gamma\left(\pi_{Y,X}\right): t\in(-\epsilon,\epsilon)\right\}
\end{equation*}
such that $s_0=s$. The vector field $\delta s=\left.ds_t/dt\right|_{t=0}\in\Gamma\left(\pi_{s^*TY,X}\right)$ along $s$ is called the \emph{infinitesimal variation}. In what follows, we consider only $\pi_{Y,X}$-vertical variations, that is, $d\pi_{Y,X}\circ\delta s=0$.

\begin{definition}
A section $s\in\Gamma\left(\pi_{Y,X}\right)$ is \emph{critical for the variational problem defined by $\mathfrak L$} if the variation of the corresponding action functional\footnote{
The variation of $\mathbb S$ only depends on the infinitesimal variation. This means that if $\{s_t\}$ and $\{s'_t\}$ are two variations of $s$ such that $\delta s=\delta s'$, then $\left.d\mathbb S(s_t)/dt\right|_{t=0}=\left.d\mathbb S(s'_t)/dt\right|_{t=0}$
} vanishes for every vertical variation of $s$, that is,
\begin{equation*}
\delta\mathbb S(s)=\left.\frac{d}{dt}\right|_{t=0}\mathbb S(s_t)=\left.\frac{d}{dt}\right|_{t=0}\int_X\mathfrak L\left(j^1s_t\right)=0,\qquad\forall\delta s\in\Gamma\left(\pi_{s^*TY,X}\right).
\end{equation*}
\end{definition}
The section $s\in\Gamma\left(\pi_{Y,X}\right)$ is critical for $\mathbb S$ if and only if it satisfies the Euler--Lagrange equations for the Lagrangian $L$, i.e. $\mathcal{EL}(L)\left(j^2 s\right)=0$, where $\mathcal{EL}(L): J^2 Y\rightarrow V^*Y=(\ker{d\pi_{Y,X}})^*$ is the \emph{Euler--Lagrange operator} (see \cite[\S 2.4]{CaRa2003}).
We now pick a closed, affine, Lie group subbundle $H\subset J^1\mathcal G$ as in Proposition \ref{ggg}, and we assume that $L$ is $H$-invariant, i.e., we have 
\begin{equation*}
L\left(\Phi^{(1)}\left(j_x^1 s,j_x^1\eta\right)\right)=L\left(j_x^1 s\right),\qquad\forall \left(j_x^1 s,j_x^1\eta\right)\in J^1 Y\times_X H.
\end{equation*}
This enables us to define the \emph{dropped} or \emph{reduced Lagrangian} as $$l: J^1 Y/H\to\mathbb R,\quad[j_x^1 s]_H\mapsto l\left(\left[j_x^1 s\right]_H\right)=L\left(j_x^1 s\right)$$ 

Let $\omega\in\Omega^1(Y,\varfrak g)$ be a generalized principal connection on $\pi_{Y,Y/\mathcal G}$ associated to a Lie group connection $\nu$ on $\pi_{\mathcal G,X}$, and choose a map $\varphi:\mathcal G\to T^*X\otimes\varfrak g$ for the Lie group subbundle $H$ as in Proposition \ref{ggg}. Thanks to Theorem \ref{theorem:descomposicioncociente}, we may regard the reduced Lagrangian as defined on $J^1(Y/\mathcal G)\times_{Y/\mathcal G}(T^*X\otimes\varfrak{g})_{\varphi,\mathfrak{H}}$. Given $s\in\Gamma\left(\pi_{Y,X}\right)$, the corresponding \emph{reduced section} is
\begin{equation*}
\overline s=\left\llbracket s,s^*\omega\right\rrbracket_{\mathfrak{H},\varphi}\in\Gamma\left(\pi_{(T^*X\otimes\varfrak{g})_{\varphi,\mathfrak{H}},X}\right).
\end{equation*}
Observe that the projection $\pi_{(T^*X\otimes\varfrak{g})_{\varphi,\mathfrak{H}},Y/\mathcal G}\circ \overline s$ is nothing but the quotient section $\sigma _s = [s]_\mathcal{G}= \pi _{Y,Y/\mathcal{G}} \circ s\in\Gamma\left(\pi_{Y/\mathcal G,X}\right)$. 

A variation $\{s_t\}$ of $s$ induces a variation $\left\{\overline s_t=\llbracket s_t,s_t^*\omega\rrbracket_{\mathfrak{H},\varphi}\right\}$ of the reduced section $\overline s=\llbracket s,s^*\omega\rrbracket_{\mathfrak{H},\varphi}$. By construction, $L\left(j^1 s_t\right)=l\left(j^1(\sigma_s)_t,\overline s_t\right)$ for every $t\in(-\epsilon,\epsilon)$. Therefore:
\begin{equation}\label{eq:equivalenciavariacionesparticular}
\left.\frac{d}{dt}\right|_{t=0}\int_X\, L\left(j^1 s_t\right)v=\left.\frac{d}{dt}\right|_{t=0}\int_X\,l\left(j^1(\sigma_s)_t,\overline s_t\right)v
\end{equation}

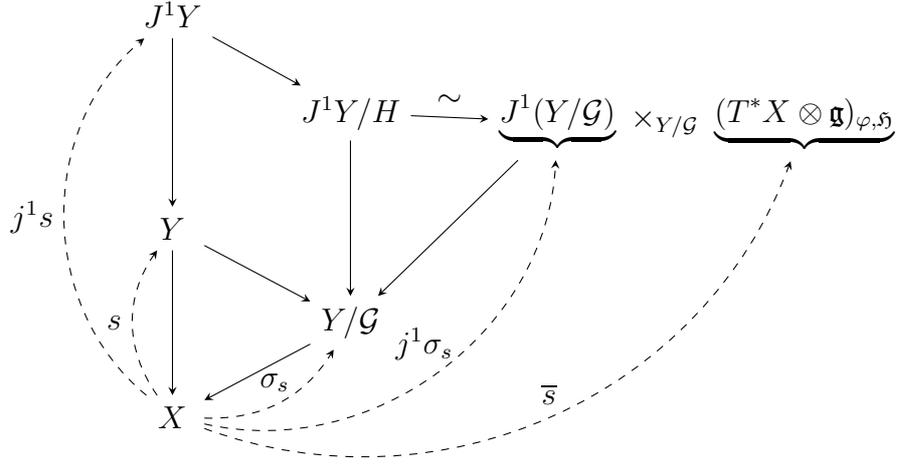
\begin{equation*}
\begin{tikzpicture}
\matrix (m) [matrix of math nodes,row sep=1.5em,column sep=2.5em,minimum width=2em]
	{J^1 Y & & &\\
	 & J^1 Y/H & \underbrace{J^1(Y/\mathcal G)} & \underbrace{(T^*X\otimes\varfrak{g})_{\varphi,\mathfrak{H}}}\\
	 Y & & &\\
	 & Y/\mathcal G & &\\
	 X & & &\\};
	\path[-stealth]
	(m-1-1) edge [] node [] {} (m-3-1)
	(m-3-1) edge [] node [] {} (m-5-1)
	(m-1-1) edge [] node [] {} (m-2-2)
	(m-2-2) edge [] node [] {} (m-4-2)
	(m-3-1) edge [] node [] {} (m-4-2)
	(m-4-2) edge [] node [] {} (m-5-1)
	(m-2-3) edge [] node [] {} (m-4-2)
	(m-2-2) edge [] node [above] {$\sim$} (m-2-3)
	(m-2-3) edge [draw=none] node [] {$\times_{Y/\mathcal G}$} (m-2-4)
	(m-5-1) edge [bend left=50,dashed] node [left] {$j^1 s$} (m-1-1)
	(m-5-1) edge [bend left=35,dashed] node [left] {$s$} (m-3-1)
	(m-5-1) edge [bend right,dashed] node [above] {$\sigma_s\;$} (m-4-2)
	(m-5-1) edge [bend right=50,dashed] node [above] {$j^1\sigma_s\quad$} (m-2-3)
	(m-5-1) edge [bend right=45,dashed] node [above] {$\overline s\;$} (m-2-4);
\end{tikzpicture}
\end{equation*}

\begin{remark}
The calculus of variations described above is straightforwardly extended to a non-compact base manifold $X$ by considering compactly supported variations. In other words, given a section $s\in\Gamma(\pi_{Y,X})$, the only variations $\delta s$ of $s$ allowed are those satisfying 
\begin{equation*}
\{x\in X: \delta s(x)\neq 0\}\subset\mathcal U
\end{equation*}
for some open subset $\mathcal U\subset X$ with compact closure $\overline{\mathcal U}$. Observe that, in particular $\delta s=0$ on the boundary $\partial\mathcal U$.
\end{remark}

%%%%%%%%%%
\subsection{Reduced variations on \texorpdfstring{$(T^*X\otimes\varfrak{g})_{\varphi,\mathfrak{H}}$}{the reduced space}}

In this section we compute the variation $\delta \overline s$ of the reduced section $\overline s\in \Gamma (\pi_{(T^*X\otimes\varfrak{g})_{\varphi,\mathfrak{H}},X})$ induced by a variation $\delta s$ of an unreduced section $s\in\Gamma(\pi_{Y,X})$. More particularly, we are first interested in the vertical part 
\begin{equation*}
\delta^\nabla\overline s(x)=\delta\overline s(x)^v=\left.\frac{\nabla\overline s_t(x)}{dt}\right|_{t=0},\qquad x\in X,
\end{equation*}
of that reduced variation with respect to the connection $\nabla$ on $\pi_{(T^*X\otimes\varfrak{g})_{\varphi,\mathfrak{H}},Y/\mathcal G}$ built in Proposition \ref{prop:inducedconnection}. For that, we analyze below the corresponding expression of $\delta^\nabla\overline s$ when the variation $\delta s = ds_t/dt|_{t=0}$ is vertical or horizontal with respect to $\omega$. By linearity, the expression of $\delta ^\nabla \overline s$ will be the combination of both terms. In addition, since $(T^*X\otimes\varfrak{g})_{\varphi,\mathfrak{H}}\to Y/\mathcal G$ is an affine bundle, we will identify the tangent space of each fiber with the associated vector space, i.e.,
\begin{equation*}
V_{\overline y}\left((T^*X\otimes\varfrak{g})_{\varphi,\mathfrak{H}}\right)=T_{\overline y}\left(\left((T^*X\otimes\varfrak{g})_{\varphi,\mathfrak{H}}\right)_\sigma\right)=\left((T^*X\otimes\varfrak{g})_{0,\mathfrak{H}}\right)_\sigma,
\end{equation*}
for each $\overline y\in (T^*X\otimes\varfrak{g})_{\varphi,\mathfrak{H}}$, where $\sigma=\pi_{(T^*X\otimes\varfrak{g})_{\varphi,\mathfrak{H}},Y/\mathcal G}(\overline y)\in Y/\mathcal G$. Analogously, the restriction of the tangent map of $\varphi:\mathcal G\to T^*X\otimes\varfrak g$ to $\varfrak g=1^*(V\mathcal G)$ yields the following morphism of vector bundles,
\begin{equation*}
\varphi_*=d\varphi|_{\varfrak g}:\varfrak g\to T^*X\otimes\varfrak g.
\end{equation*}

\begin{lemma}
If $\{s_t\}$ is a $\pi_{Y,Y/\mathcal G}$-vertical variation, that is, $\delta s=\xi_{s}^*%=(d\Phi_{s(x)})_{1_x}(\xi(x))
$ for some $\xi\in\Gamma(\pi_{\varfrak g,X})$, then we can suppose that it is of the form $s_t=s\cdot\exp(t\,\xi)$. In that
case, $$\left.\frac{d}{dt}\right|_{t=0}\Ad_{\exp(t\,\xi)}\circ\left(s_t^*\omega+\varphi(\exp(-t\,\xi))\right)=\nabla^{\varfrak g}\xi-\varphi_*\xi.$$
\end{lemma}

\begin{proof}
The first statement is because the variation of the functional only depends on the infinitesimal variation. For the second part, we have
\begin{align*}
\left.\frac{d}{dt}\right|_{t=0}\Ad_{\exp(t\,\xi)}\circ s_t^*\omega & =\left.\frac{d}{dt}\right|_{t=0} s_t^*\omega+\left.\frac{d}{dt}\right|_{t=0}\Ad_{\exp(t\,\xi)}\circ s^*\omega\vspace{0.15cm}\\
& =\left.\frac{d}{dt}\right|_{t=0}\omega_{\Phi(s,\exp(t\,\xi))}\left((d\Phi)_{(s,\exp(t\,\xi))}\left(ds,d\exp(t\,\xi)\right)\right)+\operatorname{ad}(\xi)(s^*\omega)\vspace{0.15cm}\\
& =\left.\frac{d}{dt}\right|_{t=0}\Ad_{\exp(-t\,\xi)}\left(s^*\omega+\nu\circ d\exp(t\,\xi)\right)+[\xi,s^*\omega]\vspace{0.15cm}\\
& =\left.\frac{d}{dt}\right|_{t=0}\nu\circ d\exp(t\,\xi)+\left.\frac{d}{dt}\right|_{t=0}\Ad_{\exp(-t\,\xi)}\circ\nu\circ d1\vspace{0.15cm}\\
& \overset{(\star)}{=}\left.\frac{d}{dt}\right|_{t=0}\nu\circ d\exp(t\,\xi)\vspace{0.15cm}\\
%& =\left.\frac{d}{dt}\right|_{t=0}\nu_{\exp(t\,\xi)}\circ d\exp(t\,\xi)\vspace{0.15cm}\\
%& = left.\frac{d}{dt}\right|_{t=0}\left(d\exp(t\,\xi)-\hat\nu(\exp(t\,\xi)\left((d\pi_{\mathcal G,X})_{\exp(t\,\xi)}\circ d\exp(t\,\xi)\right)\right)\vspace{0.15cm}\\
%& =d\xi-\left.\frac{d}{dt}\right|_{t=0}\hat\nu(\exp(t\,\xi)\left(d(\pi_{\mathcal G,X}\circ\exp(t\,\xi))\right)\vspace{0.15cm}\\
%& =d\xi-\left.\frac{d}{dt}\right|_{t=0}\hat\nu(\exp(t\,\xi))\\
& =\nabla^{\varfrak g}\xi.
\end{align*}
Equality $(\star)$ comes from property (i) in the definition of Lie group connection. Similarly, 
\begin{align*}
\left.\frac{d}{dt}\right|_{t=0}\Ad_{\exp(t\,\xi)}\circ\varphi(\exp(-t\,\xi)) & =\left.\frac{d}{dt}\right|_{t=0}\varphi(\exp(-t\,\xi))+\left.\frac{d}{dt}\right|_{t=0}\Ad_{\exp(t\,\xi)}\circ\varphi(\exp(0))\vspace{0.15cm}\\
& =-\varphi_* \left((d\exp)_0(\xi)\right)+\operatorname{ad}(\xi)(\varphi(1))\vspace{0.15cm}\\
& =-\varphi_*\xi ,
\end{align*}
where we have used that $\varphi(1)=0$ and $(d\exp)_0=\operatorname{id}_{\varfrak g}$. 
\end{proof}

\begin{proposition}[$\delta ^\nabla \overline s$ for vertical variations]\label{prop:variacionesverticales}
If $\{s_t\}$ is $\pi_{Y,Y/\mathcal G}$-vertical, that is, $\delta s=\xi_{s}^*%=(d\Phi_{s(x)})_{1_x}(\xi(x))
$ for some $\xi\in\Gamma(\pi_{\varfrak g,X})$, then:
\begin{equation*}
\delta^\nabla\overline s(x)=\left\llbracket s(x),\left(\nabla^{\varfrak g}\xi\right)(x)-(\varphi_*\xi)(x)\right\rrbracket_{\mathfrak{H},0},\qquad x\in X.
\end{equation*}
\end{proposition}

\begin{proof}
The reduced variations are
\begin{equation*}
\overline s_t(x)=\left\llbracket s(x)\cdot\exp(t\,\xi(x)),\left(s_t^*\omega\right)_x\right\rrbracket_{\mathfrak{H},\varphi}=\left\llbracket s(x),\Ad_{\exp(t\,\xi(x))}\circ\left(\left(s_t^*\omega\right)_x+\varphi(\exp(-t\,\xi(x)))\right)\right\rrbracket_{\mathfrak{H},\varphi}.
\end{equation*}
Since $\pi_{(T^*X\otimes\varfrak{g})_{\varphi,\mathfrak{H}},Y/\mathcal G}\circ\overline s_t =\sigma_s$ for all $t\in(-\epsilon,\epsilon)$, we have that $\{\overline s_t\}$ is a $\pi_{(T^*X\otimes\varfrak{g})_{\varphi,\mathfrak{H}},Y/\mathcal G}$-vertical variation and $\left.\nabla \overline s_t(x)/dt\right|_{t=0}=\left.d\overline s_t(x)/dt\right|_{t=0}$. Subsequently, the previous Lemma results in
\begin{align*}
\delta^\nabla\overline s(x) & =\left\llbracket s(x),\left.\frac{d}{dt}\right|_{t=0}\Ad_{\exp(t\,\xi(x))}\circ\left((s_t^*\omega)_x+\varphi(\exp(-t\,\xi_x))\right)\right\rrbracket_{\mathfrak{H},0}\vspace{0.15cm}\\
& =\left\llbracket s(x),(\nabla^{\varfrak g}\xi)(x)-(\varphi_*\xi)(x)\right\rrbracket_{\mathfrak{H},0}.
\end{align*}
\end{proof}

\begin{lemma}
If $\{s_t\}$ is $\omega$-horizontal, i.e. $\omega(\delta s)=0$, then the horizontal component of $\delta \overline s$ with respect to $\nabla$ is $$\delta\overline s(x)^h=\left.\frac{d}{dt}\right|_{t=0}\left\llbracket s_t(x),(s^*\omega)_x\right\rrbracket_{\mathfrak{H},\varphi},\qquad x\in X.$$
\end{lemma}

\begin{proof}
Let $Hor_{\overline s(x)}^\nabla: T_{\sigma_s(x)}(Y/\mathcal G)\rightarrow T_{\overline s(x)}((T^*X\otimes\varfrak{g})_{\varphi,\mathfrak{H}})$ be the horizontal lift given by $\nabla$ at $\overline s(x)$. Then $\delta\overline s(x)^h=Hor_{\overline s(x)}^{\nabla}\left(\delta\sigma_s(x)\right)$, since $(d\pi_{(T^*X\otimes\varfrak{g})_{\varphi,\mathfrak{H}},Y/\mathcal G})_{\overline s(x)}\circ\delta\overline s(x)=\delta\sigma_s(x)$. By construction of $\nabla$, we have\footnote{
Recall that $(d\pi_{Y/\mathcal G,X})_{\sigma_s(x)}(\delta\sigma_s(x))=0$, since the variation $\{s_t\}$ is $\pi_{Y,X}$-vertical.
}:
\begin{equation*}
Hor_{\overline s(x)}^\nabla\left(\delta\sigma_s(x)\right)=\left(dq\right)_{\left(s(x),(s^*\omega)_x\right)}\left(Hor_{s(x)}^\omega\left(\delta\sigma_s(x)\right),Hor_{(s^*\omega)_x}^{\nabla^\otimes}(0_x)\right),
\end{equation*}
where $q: Y\times_X(T^*X\otimes\varfrak g)\to (T^*X\otimes\varfrak{g})_{\varphi,\mathfrak{H}}$ is the quotient projection. The horizontality of $\delta s(x)$ yields $Hor_{s(x)}^\omega\left(\delta\sigma_s(x)\right)=\delta s(x)$, and it is clear that $Hor_{(s^*\omega)_x}^{\nabla^\otimes}(0_x)=0_{(s^*\omega)_x}$. Hence,
\begin{equation*}
\delta\overline s(x)^h=Hor_{\overline s(x)}^\nabla(\delta\sigma_s(x))=\left.\frac{d}{dt}\right|_{t=0}(q\circ\gamma)(t)=\left.\frac{d}{dt}\right|_{t=0}\left\llbracket s_t(x),(s^*\omega)_x\right\rrbracket_{\mathfrak{H},\varphi},
\end{equation*}
where for the computation of the tangent map $dq$, we consider the curve $\gamma:(-\epsilon,\epsilon)\rightarrow Y\times_X(T^*X\otimes\varfrak g)$, $\gamma(t)=\left(s_t(x),(s^*\omega)_x\right)$, since $\gamma (0)=(s(x),(s^*\omega)_x)$ and $\gamma '(0)=(\delta s(x),0_x)$.
\end{proof}

\begin{proposition}[$\delta ^\nabla \overline s$ for horizontal variations]
If $\{s_t\}$ is $\omega$-horizontal, i.e. $\omega(\delta s)=0$, then
\begin{equation*}
\delta^\nabla\overline s(x)=\left\llbracket s(x),\Omega_{s(x)}\left(\delta s(x),(ds)_x\right)\right\rrbracket_{\mathfrak{H},0},\qquad x\in X.
\end{equation*}
\end{proposition}

\begin{proof}
Thanks to the previous Lemma:
\begin{align*}
\delta^\nabla\overline s(x) & =\delta\overline s(x)-\delta\overline s(x)^h\\
& =\left.\frac{d}{dt}\right|_{t=0}\left\llbracket s_t(x),(s_t^*\omega)_x\right\rrbracket_{\mathfrak{H},\varphi}-\left.\frac{d}{dt}\right|_{t=0}\left\llbracket s_t(x),(s^*\omega)_x\right\rrbracket_{\mathfrak{H},\varphi}\\
& =\left.\frac{d}{dt}\right|_{t=0}\left\llbracket s_t(x),(s_t^*\omega)_x-(s^*\omega)_x\right\rrbracket_{\mathfrak{H},0}\\
& =\left\llbracket s(x),\left.\frac{d}{dt}\right|_{t=0}(s_t^*\omega)_x\right\rrbracket_{\mathfrak{H},0}.
\end{align*}

Since the formula that we are proving is local, we can suppose that our bundles are trivial, that is, $\mathcal G=X\times G$, $\varfrak g=X\times\mathfrak g$ and $Y=Y/\mathcal G\times G$, where $\mathfrak g$ is the Lie algebra of $G$.
We can thus regard $\omega$ as a 1-form on $Y$ with values in $\mathfrak g$. Then
\[
\left.\frac{d}{dt}\right|_{t=0}(s_t^*\omega)_x = s^*\pounds_{\delta s}\omega = s^*(i_{\delta s}d\omega+d(i_{\delta s}\omega))=d\omega _{s(x)}(\delta s, ds),
\]
where $\pounds$ is the Lie derivative and $\omega (\delta s)=0$. Then
\begin{align*}
\left.\frac{d}{dt}\right|_{t=0}(s_t^*\omega)_x & =d\omega _{s(x)}(\delta s, (ds)^h)+ d\omega _{s(x)}(\delta s, (ds)^v)\\
& =d\omega _{s(x)}(\delta s, (ds)^h)\\
& ={\rm d}^{\varfrak g} \omega (\delta s, ds)=\Omega _{s(x)}(\delta s, ds),
\end{align*}
since $d\omega_{s(x)}(\delta s, (ds)^v)=\delta s(\omega((ds)^v))-\omega_{s(x)}([\delta s, (ds)^v])=0$. Indeed, we are working on a trivialization, so we can write $(ds)^v=\hat\xi^*$ for some $\hat\xi\in\mathfrak g$. Thus, $\omega((ds)^v)=\omega\left(\hat\xi^*\right)=\hat\xi$ and $\delta s(\hat\xi)=0$. In addition, we have
\begin{align*}
\left[\delta s,(ds)^v\right] & =\left(\pounds_{(ds)^v}\delta s\right)\vspace{0.1cm}\\
& =\left.\frac{d}{dt}\right|_{t=0}(d\Phi)_{(s,g(t))}\left(\delta s,0_{g(t)}\right)\vspace{0.1cm}\\
& =\left.\frac{d}{dt}\right|_{t=0} U_t=U,
\end{align*}
where $g:(-\epsilon,\epsilon)\rightarrow G$ is defined as $g(t)=\exp\left(t\hat\xi\right)$. Note that $U\in T_{s}Y$ is $\pi_{Y,Y/\mathcal G}$-horizontal since so is each $Z_t\in T_{s\cdot g(t)}Y$:
\begin{equation*}
\omega_{s\cdot g(t)}\left((d\Phi)_{s,g(t)}\left(\delta s,0_{g(t)}\right)\right)=\Ad_{g(t)^{-1}}\left(\omega_{s}(\delta s(x))+\nu(0_{g(t)})\right)=0.
\end{equation*}
\end{proof}

Since every arbitrary variation $\delta s$ can be split into its $\pi_{Y,Y/\mathcal G}$-vertical and horizontal parts, we obtain the following result. We also make use of Lemma \ref{lemma:igualdad}.

\begin{corollary}\label{corollary:variacionesreducidas}
Let $\delta s$ be a variation of a section $s\in\Gamma(\pi_{Y,X})$ and consider the induced variation $\delta^\nabla\overline s$ of the reduced section $\overline s=\llbracket s,s^*\omega\rrbracket_{\mathfrak{H},\varphi}\in\Gamma(\pi_{(T^*X\otimes\varfrak{g})_{\varphi,\mathfrak{H}},X})$. Then:
\begin{equation*}
\delta^\nabla\overline s=\left\llbracket s,\nabla^{\varfrak g}\xi-\varphi_*\xi +\Omega\left(\delta s,ds\right)\right\rrbracket_{\tilde{\mathfrak{H}}},\qquad\xi=\omega(\delta s).
\end{equation*}
\end{corollary}

At last, we express the reduced variation in terms of the reduced section. To that end, we define the operator $\overline\nabla^{\tilde{\varfrak g}}:\Gamma(\pi_{\tilde{\varfrak g},X})\to\Gamma(\pi_{T^*X\otimes\tilde{\varfrak g},X})$ as
\begin{equation}\label{eq:overlinenablag}
\left(\overline\nabla_U^{\tilde{\varfrak g}}\tilde\xi\right)(x_0)=\left.\frac{\nabla^{\tilde{\varfrak g}}\left(\tilde\xi\circ x\right)(t)}{dt}\right|_{t=0},\qquad\tilde\xi\in\Gamma(\pi_{\tilde{\varfrak g},X}),~U\in\mathfrak X(X),~x_0\in X,
\end{equation}
where $x:(-\epsilon,\epsilon)\to X$ is a curve such that $x'(0)=U(x_0)$.

\begin{lemma}\label{lemma:overlinenablag}
For each $\tilde\xi=[s,\xi]_{\mathcal G}\in\Gamma(\pi_{\tilde{\varfrak g},X})$, we have
\begin{equation*}
\overline\nabla^{\tilde{\varfrak g}}\tilde\xi=[s,\nabla^{\varfrak g}\xi+\ad_{s^*\omega}(\xi)]_{\mathcal G}.
\end{equation*}
\end{lemma}

\begin{proof}
Given $U\in\mathfrak X(X)$ and $x_0\in X$, let $x:(-\epsilon,\epsilon)\to X$ be such that $x'(0)=U(x_0)$. Let $\gamma=\pi_{Y,Y/\mathcal G}\circ s\circ x:(-\epsilon,\epsilon)\to Y/\mathcal G$ and write
\begin{equation}\label{eq:g1}
(s\circ x)(t)=\left({^\omega\big|\big|}_{\gamma(0)}^{\gamma(t)}s(x_0)\right)\cdot g(t)
\end{equation}
for some $g:(-\epsilon,\epsilon)\to\mathcal G$. Observe that $\pi_{\mathcal G,X}\circ g=x$ and $g(0)=1_{x_0}$. From Proposition \ref{prop:nablatildevarfarkg} and \eqref{eq:compatibilityofnablag}, we get
\begin{align*}
\left(\overline\nabla_U^{\tilde{\varfrak g}}\tilde\xi\right)(x_0) & =\left.\frac{\nabla^{\tilde{\varfrak g}}\left(\tilde\xi\circ x\right)(t)}{dt}\right|_{t=0}\\
& =\left.\frac{d}{dt}\right|_{t=0}{^{\tilde{\varfrak g}}\big|\big|}_{\gamma(t)}^{\gamma(0)}\left[(s\circ x)(t),(\xi\circ x)(t)\right]_{\mathcal G}\\
& =\left.\frac{d}{dt}\right|_{t=0}{^{\tilde{\varfrak g}}\big|\big|}_{\gamma(t)}^{\gamma(0)}\left[{^\omega\big|\big|}_{\gamma(0)}^{\gamma(t)}s(x_0),\Ad_{g(t)}\circ(\xi\circ x)(t)\right]_{\mathcal G}\\
& =\left.\frac{d}{dt}\right|_{t=0}\left[{^\omega\big|\big|}_{\gamma(t)}^{\gamma(0)}\left({^\omega\big|\big|}_{\gamma(0)}^{\gamma(t)}s(x_0)\right),{^{\varfrak g}\big|\big|}_{x(t)}^{x(0)}\left(\Ad_{g(t)}\circ(\xi\circ x)(t)\right)\right]_{\mathcal G}\\
& =\left.\frac{d}{dt}\right|_{t=0}\left[s(x_0),\Ad_{{^\nu||}_{x(t)}^{x(0)}g(t)}\circ{^{\varfrak g}\big|\big|}_{x(t)}^{x(0)}(\xi\circ x)(t)\right]_{\mathcal G}\\
& =\left[s(x_0),\left(\nabla^{\varfrak g}_U\xi\right)(x_0)+\ad_{\zeta}(\xi(x_0))\right]_{\mathcal G},
\end{align*}
where we denote $\zeta=\left.\frac{d}{dt}\right|_{t=0}{^\nu||}_{x(t)}^{x(0)}g(t) \in\varfrak g_{x_0}$. By taking the covariant derivative in \eqref{eq:g1}, we obtain
\begin{align*}
\left.\frac{D^\omega}{Dt}\right|_{t=0}(s\circ x)(t) & =\left.\frac{d}{dt}\right|_{t=0}{^\omega\big|\big|}_{\gamma(t)}^{\gamma(0)}(s\circ x)(t)\\
& =\left.\frac{d}{dt}\right|_{t=0}{^\omega\big|\big|}_{\gamma(t)}^{\gamma(0)}\left(\left({^\omega\big|\big|}_{\gamma(0)}^{\gamma(t)}s(x_0)\right)\cdot g(t)\right)\\
& =\left.\frac{d}{dt}\right|_{t=0}\left(s(x_0)\cdot{^\nu\big|\big|}_{x(t)}^{x(0)}g(t)\right)\\
& =\zeta_{s(x_0)}^*,
\end{align*}
where Proposition \ref{prop:compatibilityofomega} has been used. We conclude by applying $\omega$ at both sides of the previous equation:
\begin{align*}
(s^*\omega)_{x_0}(U(x_0)) & =\omega_{s(x_0)}\left((ds)_{x_0}(U(x_0))\right)\\
& =\omega_{s(x_0)}\left(\left.\frac{d}{dt}\right|_{t=0}(s\circ x)(t)\right)\\
& =\omega_{s(x_0)}\left(\left.\frac{D^\omega}{Dt}\right|_{t=0}(s\circ x)(t)\right)\\
& =\omega_{s(x_0)}\left(\zeta_{s(x_0)}^*\right)\\
& =\zeta,
\end{align*}
where $(i)$ of Definition \ref{def:connectionform} has been used.
\end{proof}

By linearity, the following map is well defined
\begin{equation}\label{eq:tildevarphi}
\tilde\varphi_*:\tilde{\varfrak g}\to T^*X\otimes\tilde{\varfrak g},\quad [y,\xi_x]_{\mathcal G}\mapsto\left[y,\varphi_* \xi_x\right]_{\mathcal G}.
\end{equation}
Moreover, given $\alpha\otimes\tilde\zeta\in T^*X\otimes\tilde{\varfrak g}$, we define the vector bundle morphism
\begin{equation}\label{eq:adadrepresentation}
\ad_{\alpha\otimes\tilde\zeta}:\tilde{\varfrak g}\to T^*X\otimes\tilde{\varfrak g},\quad\tilde\xi\mapsto\alpha\otimes\ad_{\tilde\zeta}(\tilde\xi),
\end{equation}
which is well defined provided $\pi_{\tilde{\varfrak g},Y/\mathcal G}(\tilde\xi)=\pi_{\tilde{\varfrak g},Y/\mathcal G}(\tilde\zeta)$, where $\ad:\tilde{\varfrak g}\times_{Y/\mathcal G}\tilde{\varfrak g}\to\tilde{\varfrak g}$ is the (fibered) adjoint representation on the (generalized) adjoint bundle. Note that we keep the same symbol for both maps for simplicity. By denoting $\tilde s=[s,s^*\omega]_{\mathcal G}\in\Gamma(\pi_{T^*X\otimes\tilde{\varfrak g},X})$ and $\tilde\xi=[s,\xi]_{\mathcal G}\in\Gamma(\pi_{\tilde{\varfrak g},X})$, it is clear that $[s,\ad_{s^*\omega}(\xi)]_{\mathcal G}=\ad_{\tilde s}(\tilde\xi)$. Thence, the following result is now straightforward.

\begin{corollary}\label{corollary:variacionesverticales}
Let $\delta s$ be a variation of a section $s\in\Gamma(\pi_{Y,X})$ and consider the induced variation of the reduced section $\overline s=\llbracket s,s^*\omega\rrbracket_{\mathfrak{H},\varphi}$. Then for each $x\in X$ we have
\begin{equation*}
\delta ^\nabla \overline s (x)=\left[\overline\nabla^{\tilde{\varfrak{g}}}\tilde \xi-\ad_{\tilde s}(\tilde\xi)-\tilde\varphi_*\tilde\xi+\tilde\Omega(\delta\sigma_s,d\sigma_s)\right]_{\tilde{\mathfrak{H}}},
\end{equation*}
where $\tilde\xi(x)=[s(x),\omega_{s(x)}(\delta s(x))]_{\mathcal G}$ and $\tilde s(x)=[s(x),(s^*\omega)_x]_{\mathcal G}$ for each $x\in X$.
\end{corollary}

%%%%%%%%%%
\subsection{Variations on \texorpdfstring{$J^1(Y/\mathcal G)$}{the first jet of the quotient space}}

The induced variations $\delta\sigma_s$ of $\sigma_s=\pi_{Y,Y/\mathcal G}\circ s$ are just the projection of the variations of $s$, that is,
$$\delta \sigma _s = d\pi _{Y,Y/\mathcal G}(\delta s).$$
In particular, these reduced variations are free, with no particular constraints (in contrast with the constraints for $\delta \overline s$ analyzed in the previous section). For later convenience, we now analyze the vertical part of the 1-jet lift of $\delta \sigma _s$ to $J^1(Y/\mathcal{G})$ with respect to a suitable connection. 

Let $\nabla^{Y/\mathcal G}$ be a linear connection on the tangent bundle $T(Y/\mathcal G)\to Y/\mathcal G$ and consider the operator $\overline\nabla^{Y/\mathcal G}:\Gamma(\pi_{T(Y/\mathcal G),X})\rightarrow\Gamma(\pi_{T^*X\otimes T(Y/\mathcal G),X})$ defined as
\begin{equation}\label{eq:overlinenablaEG}
	\left(\overline\nabla^{Y/\mathcal G}_U\alpha\right)(x)=\left.\frac{\nabla ^{Y/\mathcal G}(\alpha\circ \gamma)(t)}{dt}\right|_{t=0},\qquad\alpha\in\Gamma(\pi_{T(Y/\mathcal G),X}),~U\in\mathfrak X(X),~x\in X,
\end{equation}
where $\gamma:(-\epsilon,\epsilon)\rightarrow X$ is such that $\gamma'(0)=U(x)$. The following Lemma is an adaptation of \cite[Corollary 3.4]{ElGaHoRa2011}.

\begin{lemma}\label{lemma:variacionj1sigma}
	Let $\{s_t\}$ be a variation of a section $s\in\Gamma(\pi_{Y,X})$ and consider the induced variation of $\sigma_s\in\Gamma(\pi_{Y/\mathcal G,X})$. Then:
	\begin{equation*}
		\left.\frac{\nabla ^{Y/\mathcal G}d(\sigma_s)_t(U_x)}{dt}\right|_{t=0}=\left(\overline\nabla^{Y/\mathcal G}_{U}\delta\sigma_s\right)(x)+T^{Y/\mathcal G}\left(\delta\sigma_s(x),j^1_x\sigma_s(U_x)\right),\qquad x\in X,~U_x\in T_x X,
	\end{equation*}
	where $U\in\mathfrak X(X)$ is such that $U(x)=U_x$, and $T^{Y/\mathcal G}\in\mathcal T^1_2(Y/\mathcal G)$ is the \emph{torsion tensor} of $\nabla^{Y/\mathcal G}$.
\end{lemma}

The connection $\nabla^{Y/\mathcal G}$ is said to be \emph{projectable} on a linear connection $\nabla^X$ on $\pi_{TX,X}$ if the following diagram is commutative:
\begin{equation*}
	\begin{tikzpicture}
			\matrix (m) [matrix of math nodes,row sep=3em,column sep=3em,minimum width=2em]
			{	T\left(T(Y/\mathcal G)\right) & T(Y/\mathcal G)\\
				T(TX) & TX\\};
			\path[-stealth]
			(m-1-1) edge [] node [left] {$d(d\pi_{Y/\mathcal G,X})$} (m-2-1)
			(m-1-1) edge [] node [above] {$\nu^{Y/\mathcal G}$} (m-1-2)
			(m-1-2) edge [] node [right] {$d\pi_{Y/\mathcal G,X}$} (m-2-2)
			(m-2-1) edge [] node [above] {$\nu^X$} (m-2-2);
	\end{tikzpicture}
\end{equation*}
where $\nu^{Y/\mathcal G}$ and $\nu^X$ are the vertical projections of the connections $\nabla^{Y/\mathcal G}$ and $\nabla^X$, respectively. Consider the vector bundle
$$\pi_{V,Y/\mathcal G}: V=T^*X\otimes T(Y/\mathcal G)\rightarrow Y/\mathcal G.$$
A section $\rho\in\Gamma(\pi_{J^1(Y/\mathcal G),Y/\mathcal G})$ can be regarded as a section $\rho\in\Gamma(\pi_{V,Y/\mathcal G})$ since $J^1(Y/\mathcal G)\subset V$. Likewise, $\rho$ can be regarded as a connection on $\pi_{Y/\mathcal G,X}$. The next Proposition is an adaptation of \cite[Theorem 3.1, Lemma 3.1]{janyska1996} to our case. See also \cite[Equations (3.13), (3.14)]{ElGaHoRa2011}.

\begin{proposition}\label{prop:conexionJ1EG}
If $\nabla^{Y/\mathcal G}$ is projectable onto $\nabla^X$, then it induces an affine connection $\nabla^{J^1(Y/\mathcal G)}$ in $\pi_{J^1(Y/\mathcal G),Y/\mathcal G}$ given by
\begin{equation*}
	\nabla^{J^1(Y/\mathcal G)}_Z\rho=(id\otimes\nu^{\rho})\circ\nabla^V_Z\rho,\qquad Z\in\mathfrak X(Y/\mathcal G),~\rho\in\Gamma(\pi_{J^1(Y/\mathcal G),Y/\mathcal G}),
\end{equation*}
where $\operatorname{id}: T^*X\rightarrow T^*X$ is the identity map, $\nu^\rho$ is the vertical projection associated to $\rho$, and $\nabla^V$ is the linear connection induced on $V=T^*X\otimes T(Y/\mathcal{G})\to Y/\mathcal{G}$ by the tensor product of the connections $\nabla^{Y/\mathcal{G}}$ and $\nabla^X$.
\end{proposition}

\begin{corollary}\label{corollary:variacionesJEG}
For any variation $\delta \sigma _s = d/dt|_{t=0}(\sigma _s)_t$ of a reduced variation $\sigma _s$, the vertical part with respect to the connection $\nabla ^{J^1(Y/\mathcal{G})}$ of its 1-jet lift is given by
\begin{equation*}
\delta^{J^1(Y/\mathcal G)}j^1\sigma_s=\left.\frac{\nabla ^{J^1(Y/\mathcal G)}j^1(\sigma_s)_t}{dt}\right|_{t=0}=\overline\nabla^{Y/\mathcal G}\delta\sigma_s+T^{Y/\mathcal G}\left(\delta\sigma_s,d\sigma_s\right).
\end{equation*}
\end{corollary}

\begin{proof}
Given $x\in X$ and $U_x\in T_x X$, the tensor product connection $\nabla ^V$ satisfies
\begin{eqnarray*}
\frac{\nabla^V j^1_x(\sigma_s)_t}{dt}(U_x)= \frac{\nabla^{Y/\mathcal G}j^1_x(\sigma_s)_t(U_x)}{dt}-(j^1_x\sigma _s)\left(\frac{\nabla^X U_x}{dt}\right)=
\frac{\nabla^{Y/\mathcal G}j^1_x(\sigma_s)_t(U_x)}{dt}.
\end{eqnarray*}
Since the projection of $j^1_x(\sigma_s)_t(U)_x$ by $d\pi _{Y/\mathcal{G},X}$ is constantly $U_x$, we have that $\nabla^{Y/\mathcal G}j^1_x(\sigma_s)_t(U_x)/dt$ is $\pi _{Y/\mathcal{G},X}$ vertical so that
\begin{equation*}
\frac{\nabla^{Y/\mathcal G}j^1_x(\sigma_s)_t}{dt}= (id\otimes \nu ^{j^1\sigma _s})\circ\frac{\nabla^{V}j^1_x(\sigma_s)_t}{dt}
= \frac{\nabla^{V}j^1_x(\sigma_s)_t}{dt},
\end{equation*}
and the proof is complete by Lemma \ref{lemma:variacionj1sigma}.
\end{proof}

If the connection $\nabla ^X$ is torsionless (and that will be our choice from now on), the formula above simply reads
\[
\delta^{J^1(Y/\mathcal G)}j^1\sigma_s=\overline\nabla^{Y/\mathcal G}\delta\sigma_s.
\]

%%%%%%%%%%
\subsection{Reduced equations}

Let $\mathfrak{H}^\circ\subset TX\otimes \varfrak{g}^*$ be the annihilator of $\mathfrak{H}\subset T^*X\otimes \varfrak{g}$. Then, $\tilde{\mathfrak{H}}^\circ=(Y\times_X \mathfrak{H}^\circ)/\mathcal{G}  \subset TX\otimes \tilde{\varfrak{g}}^*$ is the annihilator of $\tilde{\mathfrak{H}}\subset T^*X\otimes\tilde{\varfrak{g}}$. This space is canonically isomorphic to the dual vector bundle
\begin{equation*}
\tilde{\mathfrak{H}}^\circ \simeq \left((T^*X\otimes\tilde{\varfrak g})/\tilde{\mathfrak{H}}\right)^*,
\end{equation*}
where the dual pairing is given by $\left\langle [y,\zeta_x]_{\mathcal G},\llbracket y,\xi_x\rrbracket_{\tilde{\mathfrak{H}}}\right\rangle=\left\langle\zeta_x,\xi_x\right\rangle$ for each $[y,\zeta_x]_{\mathcal G}\in \tilde{\mathfrak{H}}^\circ$ and $\llbracket y,\xi_x\rrbracket_{\tilde{\mathfrak{H}}}\in(T^*X\otimes\tilde{\varfrak g})/\tilde{\mathfrak{H}}$.

On the other hand, $J^1(Y/\mathcal G)\times_{Y/\mathcal G}(T^*X\otimes\varfrak{g})_{\varphi,\mathfrak{H}}\to Y/\mathcal G$ is an affine bundle, since both $J^1(Y/\mathcal G)$ and $(T^*X\otimes\varfrak{g})_{\varphi,\mathfrak{H}}$ are affine bundles over $Y/\mathcal G$. Let $\nabla^\times$ be the affine connection on this bundle induced by the affine connections $\nabla^{J^1(Y/\mathcal G)}$ and $\nabla$ introduced above.

\begin{definition}
Let $s\in\Gamma(\pi_{Y,X})$ and consider the reduced section $\overline s=\llbracket s,s^*\omega\rrbracket_{\mathfrak{H},\varphi}$, as well as $\sigma_s=\pi_{Y,Y/\mathcal G}\circ s$. The \emph{partial derivatives} of the reduced Lagrangian
\begin{equation*}
l: J^1(Y/\mathcal G)\times_{Y/\mathcal G}(T^*X\otimes\varfrak{g})_{\varphi,\mathfrak{H}}\to\mathbb R
\end{equation*}
are the sections
\begin{equation*}
\frac{\delta l}{\delta\sigma_s}\in\Gamma\left(\pi_{T^*(Y/\mathcal G),X}\right),\quad\frac{\delta l}{\delta j^1\sigma_s}\in\Gamma\left(\pi_{TX\otimes V^*(Y/\mathcal G),X}\right),\quad\frac{\delta l}{\delta\overline s}\in\Gamma\left(\pi_{\mathfrak{\tilde H}^\circ,X}\right)
\end{equation*}
defined as
\begin{align*}
\displaystyle\left\langle\frac{\delta l}{\delta\sigma_s}(x),U_x\right\rangle & =\left.\frac{d}{dt}\right|_{t=0}l\left({\gamma_U(t)}^h(t)\right),\quad && \displaystyle\forall U_x\in T_{\sigma_s(x)}(Y/\mathcal G),\vspace{0.2cm}\\
\displaystyle\left\langle\frac{\delta l}{\delta j^1\sigma_s}(x),V_x\right\rangle & =\left.\frac{d}{dt}\right|_{t=0}l\left(j_x^1\sigma_s+t\,V_x,\overline s(x)\right),\quad && \forall V_x\in T_x^*X\otimes V_{\sigma_s(x)}(Y/\mathcal G),\vspace{0.2cm}\\
\displaystyle\left\langle\frac{\delta l}{\delta\overline s}(x),W_x\right\rangle & =\left.\frac{d}{dt}\right|_{t=0}l\left(j_x^1\sigma_s,\overline s(x)+t\,W_x\right),\quad && \forall W_x\in%\left((T^*X\otimes\varfrak{g})_{0,\mathfrak{H}}\right)_{\sigma_s(x)}
=\left((T^*X\otimes\tilde{\varfrak g})/\tilde{\mathfrak{H}}\right)_{\sigma_s(x)},
\end{align*}
for each $x\in X$, where $\gamma _U (t)^h$ is the horizontal lift with respect to the connection $\nabla ^\times$ of a curve $\gamma_U:(-\epsilon,\epsilon)\rightarrow Y/\mathcal G$ such that $\gamma'(0)=U_x$. As usual, $\langle\cdot,\cdot\rangle$ denotes the corresponding dual pairings. 
\end{definition}

We remark that, whereas the two latter derivatives are (intrinsic) fiber derivatives, the partial derivative $\delta l/\delta\sigma_s$ depends on the choice of the connections. Anyway, all of them are sections projecting onto $\sigma_s$.

Since $\delta l/\delta\overline s$ lies in the annihilator of $\tilde{\mathfrak{H}}$, then for each $\llbracket s,\xi\rrbracket_{\tilde{\mathfrak{H}}}\in\Gamma\left(\pi_{(T^*X\otimes\tilde{\varfrak g})/\tilde{\mathfrak{H}},X}\right)$ the dual pairing satisfies
\begin{equation}\label{eq:emparejadodualparcial}
\left\langle\frac{\delta l}{\delta\overline s},\llbracket s,\xi\rrbracket_{\tilde{\mathfrak{H}}}\right\rangle %=\left\langle \zeta_{s}(x),\xi_x\right\rangle
=\left\langle\frac{\delta l}{\delta\overline s},\left[s,\xi\right]_{\mathcal G}\right\rangle
\end{equation}
Indeed, let $g\in\Gamma(\pi_{\mathcal G,X})$ and $\eta\in\Gamma(\pi_{\mathfrak{H},X})$. Then $\left[s\cdot g,\Ad_{g^{-1}}\circ(\xi+\eta)\right]_{\mathcal G}=\left[s,\xi\right]_{\mathcal G}+\left[s,\eta\right]_{\mathcal G}$ and $\left\langle\delta l/\delta\overline s,\left[s,\eta\right]_{\mathcal G}\right\rangle=0$, since $\left[s,\eta\right]_{\mathcal G}\in\tilde{\mathfrak{H}}$.

\begin{definition}\label{def:divergence}
The \emph{divergence} of the operator $\overline\nabla^{\tilde{\varfrak g}}$ defined in \eqref{eq:overlinenablag} is minus the adjoint of $\overline\nabla^{\tilde{\varfrak g}}$, i.e., the operator $\operatorname{div}^{\tilde{\varfrak g}}:\Gamma\left(\pi_{TX\otimes\tilde{\varfrak g}^*,X}\right)\rightarrow\Gamma\left(\pi_{\tilde{\varfrak g}^*,X}\right)$ given by:
\begin{equation*}
\int_X\left\langle \zeta,\overline\nabla^{\tilde{\varfrak g}}\xi\right\rangle v=-\int_X\left\langle\operatorname{div}^{\tilde{\varfrak g}}\zeta,\xi\right\rangle v
\end{equation*}
for every $\zeta\in\Gamma\left(\pi_{TX\otimes\tilde{\varfrak g}^*,X}\right)$ and $\xi\in\Gamma\left(\pi_{\tilde{\varfrak g},X}\right)$.

Analogously, the \emph{divergence of $\overline\nabla^{Y/\mathcal G}$} is minus the adjoint of \eqref{eq:overlinenablaEG} restricted to vertical sections (the restriction is allowed since the linear connection $\nabla^{Y/\mathcal G}$ is projectable), that is, $\operatorname{div}^{Y/\mathcal G}:\Gamma\left(\pi_{TX\otimes V^*(Y/\mathcal G),X}\right)\rightarrow\Gamma\left(\pi_{V^*(Y/\mathcal G),X}\right)$.

In the same vein, the dual operator of $\tilde\varphi_*:\tilde{\varfrak g}\to T^*X\otimes\tilde{\varfrak g}$ given in \eqref{eq:tildevarphi} is denoted by $\tilde\varphi_*^\dagger:TX\otimes\tilde{\varfrak g}^*\to\tilde{\varfrak g}^*$.
\end{definition}

The (pointwise) \emph{coadjoint representation} of the morphism $\ad_{\tilde s}:\tilde{\varfrak g}\to T^*X\otimes\tilde{\varfrak g}$ introduced in \eqref{eq:adadrepresentation}, i.e., minus its dual morphism, is denoted by $\ad_{\tilde s}^*:TX\otimes\tilde{\varfrak g}^*\to\tilde{\varfrak g}^*$.

\begin{lemma}\label{lemma:restrictedcoadjoint}
In the previous conditions, the coadjoint representation can be restricted to $\tilde{\mathfrak H}^\circ$ yielding the following vector bundle morphism:
\begin{equation*}
\ad_{\overline s}^*:\tilde{\mathfrak H}^\circ\to\tilde{\varfrak g}^*,\quad\tilde\zeta\mapsto\ad_{\tilde s}^*(\tilde\zeta).
\end{equation*}
\end{lemma}

\begin{proof}
Let $x\in X$. As a consequence of the $\Ad$-invariance of $\mathfrak H$, we have that $\ad_{\tilde\eta}(\tilde\xi)\in\tilde{\mathfrak H}$ for each $\tilde\eta\in\tilde{\mathfrak H}_{\sigma_s(x)}$ and $\tilde\xi\in\tilde{\varfrak g}_{\sigma_s(x)}$. Indeed, $\ad_{\tilde\eta}(\tilde\xi)=-[\tilde\xi,\tilde\eta]=\left.(d/dt)\right|_{t=0}\Ad_{\exp(t\,\tilde\xi)}\circ\tilde\eta\in\tilde{\mathfrak H}$. Hence, the map $\ad_{\overline s}^*:\tilde{\mathfrak H}^\circ\to\tilde{\varfrak g}^*$ is well defined. Indeed, for each $\tilde\zeta\in\tilde{\mathfrak H}^\circ$ we have $\langle\ad_{\tilde s(x)+\tilde\eta}^*(\tilde\zeta),\tilde\xi\rangle=\langle\tilde\zeta,\ad_{\tilde s(x)+\tilde\eta}(\tilde\xi)\rangle=\langle\tilde\zeta,\ad_{\tilde s(x)}(\tilde\xi)\rangle=\langle\ad_{\tilde s(x)}^*(\tilde\zeta),\tilde\xi\rangle$, where Proposition \ref{prop:vectormodelaffine} and Lemma \ref{lemma:igualdad} have been taken into account.
\end{proof}

\begin{theorem}[Reduced field equations]\label{theorem:reducedequationsmathfrakH}
Let $\pi_{Y,X}$ be a fiber bundle over a compact manifold $X$, $\pi_{\mathcal G,X}$ be a Lie group bundle and $\pi_{\varfrak g,X}$ be its Lie algebra bundle. Suppose that $\pi_{\mathcal G,X}$ acts fiberwisely, freely and properly on the right on $\pi_{Y,X}$.
Let $\omega\in\Omega^1(Y, \varfrak g)$ be a generalized principal connection on $\pi_{Y,Y/\mathcal G}$ associated to a Lie group connection $\nu$ on $\pi_{\mathcal G,X}$. Let $H\subset J^1\mathcal{G}$ be a Lie group subbundle projecting surjectively onto $\mathcal{G}$ and such that, with the identification $J^1\mathcal{G}\simeq \mathcal{G} \times _X T^*X\otimes \varfrak{g}$ given by $\nu$, it is an affine subbundle $\varphi + \mathfrak{H}$ as in Proposition \ref{ggg}.

For a $H$-invariant Lagrangian density $\mathfrak L=Lv: J^1 Y\rightarrow\bigwedge^n T^*X$ we consider the corresponding reduced Lagrangian $l: J^1(Y/\mathcal G)\times_{Y/\mathcal G}(T^*X\otimes\varfrak{g})_{\varphi,\mathfrak{H}}\rightarrow\mathbb R$. Then, for any section $s\in\Gamma(\pi_{Y,X})$ and its reduced section $\overline s=\llbracket s,s^*\omega\rrbracket_{\mathfrak{H},\varphi}$ and $\sigma = \pi_{Y,Y/\mathcal G}\circ s$, the following assertions are equivalent:

\begin{enumerate}[(i)]
    \item The variational principle $\displaystyle\delta\int_X\mathfrak L(j^1 s)=0$ holds for arbitrary (vertical) variations of $s$.

    \item The section $s$ satisfies the Euler--Lagrange equations for $L$, i.e. $\mathcal{EL}(L)\left(j^2 s\right)=0$.

    \item The variational principle $\displaystyle\delta\int_X l(j^1\sigma_s,\overline s)v=0$ holds for variations of the form:
    \begin{equation*}
    \delta^\nabla\overline s=\left[\overline\nabla^{\tilde{\varfrak{g}}}\tilde \xi-\ad_{\tilde s}\left(\tilde\xi\right)-\tilde\varphi_*\tilde\xi+\tilde\Omega(\delta\sigma_s,d\sigma_s)\right]_{\tilde{\mathfrak{H}}},
    \end{equation*}
    where $\tilde\xi\in\Gamma(\pi_{\tilde{\varfrak g},X})$ is an arbitrary section and $\delta\sigma_s$ is an arbitrary variation of $\sigma_s$.

    \item The reduced section $\overline s$ satisfies the \emph{reduced field equations}:
    \begin{equation*}
    \left\{
    \begin{array}{l}
    \displaystyle\frac{\delta l}{\delta\sigma_s}-\operatorname{div}^{Y/\mathcal G}\left(\frac{\delta l}{\delta j^1\sigma_s}\right)=\left\langle\frac{\delta l}{\delta\overline s},\iota_{d\sigma_s}\tilde\Omega\right\rangle,\vspace{0.2cm}\\
    \displaystyle\left(\operatorname{div}^{\tilde{\varfrak g}}-\ad_{\overline s}^*+\tilde\varphi_*^\dagger\right)\left(\frac{\delta l}{\delta\overline s}\right)=0.
    \end{array}
    \right.
    \end{equation*}
\end{enumerate}
The expressions above take into account the connections described in \S \ref{sec:inducedconnection}.
\end{theorem}

\begin{proof}
The equivalence of $(i)$ and $(ii)$ is a well-known fact as stated above. The equivalence of $(i)$ and $(iii)$ is a straightforward consequence of equation \eqref{eq:equivalenciavariacionesparticular} and Corollary \ref{corollary:variacionesverticales}. To complete the proof, we show the equivalence of $(iii)$ and $(iv)$.

Let $\delta s$ be a $\pi_{Y,X}$-vertical variation of $s$ and consider induced variations of $\sigma_s$ and $\overline s$. Then:
\begin{align*}
\displaystyle\left.\frac{d}{dt}\right|_{t=0}\int_{X} l\left(j^1(\sigma_s)_t,\overline s_t\right)v & =\displaystyle\int_{X}dl\left(\delta\left(j^1\sigma_s,\overline s\right)\right)v\vspace{0.15cm}\\
& =\displaystyle\int_{X}dl\left(\delta\left(j^1\sigma_s,\overline s\right)^h\right)v+\int_{X}dl\left(\delta\left(j^1\sigma_s,\overline s\right)^v\right)v\vspace{0.15cm}\\
& =\displaystyle\int_{X}\left\langle \frac{\delta l}{\delta \sigma_s},\delta \sigma _s\right\rangle v + \int_{X}\left\langle \frac{\delta l}{\delta j^1 \sigma _s},\delta^{J^1(Y/\mathcal G)} j^1\sigma _s\right\rangle v + \int _X \left\langle \frac{\delta l}{\delta \overline{s}},\delta^\nabla \overline{s} \right\rangle v,
\end{align*}
where the vertical parts $\delta^{J^1(Y/\mathcal G)} j^1 \sigma _s$ and $\delta^\nabla \overline{s}$ are defined with the connections $\nabla ^{J^1(Y/\mathcal{G})}$ and $\nabla$ respectively. Making use of Corollaries \ref{corollary:variacionesverticales} and \ref{corollary:variacionesJEG} (with vanishing torsion) we get

\begin{align*}
\displaystyle\left.\frac{d}{dt}\right|_{t=0}\int_{X} l\left(j^1(\sigma_s)_t,\overline s_t\right)v & =\displaystyle\int_{X}\left\langle\frac{\delta l}{\delta\sigma_s},\delta\sigma_s\right\rangle v
+\int_{X}\left\langle\frac{\delta l}{\delta j^1\sigma_s},\overline\nabla^{Y/\mathcal G}\delta\sigma_s\right\rangle v\vspace{0.1cm}\\
& \hspace{5mm}+\int_{X} \left\langle\frac{\delta l}{\delta\overline s},\left[\overline\nabla^{\tilde{\varfrak{g}}}\tilde \xi-\ad_{\tilde s}\left(\tilde\xi\right)-\tilde\varphi_*\tilde\xi+\tilde\Omega(\delta\sigma_s,d\sigma_s)\right]_{\tilde{\mathfrak{H}}}\right\rangle v.
\end{align*}
Thanks to \eqref{eq:emparejadodualparcial}, we may write:
\begin{equation*}
\left\langle\frac{\delta l}{\delta\overline s},\left[\overline\nabla^{\tilde{\varfrak{g}}}\tilde \xi-\ad_{\tilde s}\left(\tilde\xi\right)-\tilde\varphi_*\tilde\xi+\tilde\Omega(\delta\sigma_s,d\sigma_s)\right]_{\tilde{\mathfrak{H}}}\right\rangle=\left\langle\frac{\delta l}{\delta\overline s},\overline\nabla^{\tilde{\varfrak{g}}}\tilde \xi-\ad_{\tilde s}\left(\tilde\xi\right)-\tilde\varphi_*\tilde\xi+\tilde\Omega(\delta\sigma_s,d\sigma_s)\right\rangle.
\end{equation*}
Analogously, thanks to Lemma \ref{lemma:restrictedcoadjoint}, we have
\begin{equation*}
\left\langle\frac{\delta l}{\delta\overline s},\ad_{\tilde s}\left(\tilde\xi\right)\right\rangle=\left\langle\ad_{\overline s}\left(\frac{\delta l}{\delta\overline s}\right),\tilde\xi\right\rangle.
\end{equation*}
Using this and the divergence operators defined above we get:
\begin{align*}
\left.\frac{d}{dt}\right|_{t=0}\int_{X} l\left(j^1(\sigma_s)_t,\overline s_t\right)v  & =\int_{X}\left\langle\frac{\delta l}{\delta\sigma_s},\delta\sigma_s\right\rangle v+\int_{X}\left\langle-\operatorname{div}^{Y/\mathcal G}\left(\frac{\delta l}{\delta j^1\sigma_s}\right),\delta\sigma_s\right\rangle v\vspace{0.1cm}\\
& \hspace{5mm}+\int_{X}\left\langle\left(-\operatorname{div}^{\tilde{\varfrak g}}+\ad_{\overline s}^*-\tilde\varphi_*^\dagger\right)\left(\frac{\delta l}{\delta\overline s}\right),\tilde\xi\right\rangle v\\
& \hspace{5mm}+\int_{X}\left\langle-\left\langle\frac{\delta l}{\delta\overline s},\iota_{d\sigma_s}\tilde\Omega\right\rangle,\delta\sigma_s\right\rangle v.
\end{align*}
As a result, the variational principle of \emph{(iii)} reads
\begin{align*}
& \int_{X}\left\langle\frac{\delta l}{\delta\sigma_s}-\operatorname{div}^{Y/\mathcal G}\left(\frac{\delta l}{\delta j^1\sigma_s}\right)-\left\langle\frac{\delta l}{\delta\overline s},\iota_{d\sigma_s}\tilde\Omega\right\rangle,\delta\sigma_s\right\rangle v\\
& -\int_{X} \left\langle\left(\operatorname{div}^{\tilde{\varfrak g}}-\ad_{\overline s}^*+\tilde\varphi_*^\dagger\right)\left(\frac{\delta l}{\delta\overline s}\right),\tilde\xi\right\rangle v=0
\end{align*}
for every section $\tilde\xi\in\Gamma\left(\pi_{\tilde{\varfrak g},X}\right)$ and every variation $\delta\sigma_s$ of $\sigma_s$.
\end{proof}

\begin{remark}
The first equation holds on $\pi_{\sigma_s^*V^*(Y/\mathcal G),X}$, and the second one on $\pi_{\sigma_s^*\tilde{\varfrak g}^*,X}$. On the other hand, if the connection $\nabla^{Y/\mathcal G}$ had non-vanishing torsion, then the reduced equations would read
\begin{equation*}
    \left\{
    \begin{array}{l}
    \displaystyle\frac{\delta l}{\delta\sigma_s}-\operatorname{div}^{Y/\mathcal G}\left(\frac{\delta l}{\delta j^1\sigma_s}\right)+\left\langle\frac{\delta l}{\delta j^1\sigma_s},\iota_{d\sigma_s}T^{Y/\mathcal G}\right\rangle=\left\langle\frac{\delta l}{\delta\overline s},\iota_{d\sigma_s}\tilde\Omega\right\rangle,\vspace{0.2cm}\\
    \displaystyle\left(\operatorname{div}^{\tilde{\varfrak g}}-\ad_{\overline s}^*+\tilde\varphi_*^\dagger\right)\left(\frac{\delta l}{\delta\overline s}\right)=0.
    \end{array}
    \right.
\end{equation*}
\end{remark}

%%%%%%%%%%%%%%%%%%%%
\section{Reconstruction}\label{sec:reconstruction}

Let $\overline s$ be a section of the reduced bundle $(T^*X\otimes\varfrak{g})_{\varphi,\mathfrak{H}}\to X$ and let  
\begin{equation*}
\sigma=\pi_{(T^*X\otimes\varfrak{g})_{\varphi,\mathfrak{H}},Y/\mathcal G}\circ\overline s
\end{equation*}
be the induced section of $Y/\mathcal{G}\to X$. We consider the subset
\[
Y^\sigma=\pi_{Y,Y/\mathcal G}^{-1}\left(\sigma(X)\right)=\{ y\in Y:\pi _{Y,Y/\mathcal{G}}(y)=\sigma (\pi _{Y,X}(y))\}\subset Y.
\]
The action of $\mathcal{G}$ on $Y$ restricts to $Y^\sigma$ and $Y^\sigma /\mathcal{G}\simeq X$. In fact, we can regard $Y^\sigma$ as a pull-back bundle $Y^\sigma \simeq \sigma ^*Y \to X$ on which the Lie group bundle $\mathcal{G}$ acts transitively along the fibers:
\begin{equation*}
\begin{tikzpicture}
\matrix (m) [matrix of math nodes,row sep=3em,column sep=3em,minimum width=2em]
{Y & Y^\sigma\simeq\sigma^* Y\\
Y/\mathcal G & X\\};
\path[-stealth]
(m-1-1) edge [] node [] {} (m-2-1)
(m-1-2) edge [] node [] {} (m-2-2)
(m-2-1) edge [] node [] {} (m-2-2)
(m-2-2) edge [bend right,dashed] node [above] {$\sigma$} (m-2-1);
\end{tikzpicture}
\end{equation*} 
In particular, the adjoint bundle of $Y^\sigma \to X$ is the pull-back $\sigma ^* \tilde{\varfrak{g}}\to X$. From this point of view, the section $\overline{s}$ can be also considered as a section of the pull-backed bundle
\[
\sigma^*((T^*X\otimes\varfrak{g})_{\varphi,\mathfrak{H}})=\frac{Y^\sigma\times_X(T^*X\otimes\varfrak g)/\mathfrak{H}}{\Phi_\varphi}\to X,
\]

Moreover, the restriction of the connection form $\omega$ to $Y^\sigma$ is a generalized principal connection on $Y^\sigma \to X$ associated to the Lie group connection $\nu$, which we denote by $\sigma ^*\omega$. From \cite[Proposition 3.11]{CaRo2023} we know that, for any section $\tilde{\xi}$ of $T^*X\otimes \sigma ^*\tilde{\varfrak{g}}\to X$, the form $\sigma ^* \omega - \tilde{\xi }$ is a generalized principal connection on $Y^\sigma\to X$.

\begin{proposition}
\label{rec0}
Let $s\in\Gamma(\pi_{Y,X})$. We regard $s$ as a section of $Y^\sigma\to X$, where $\sigma = \pi_{Y,Y/\mathcal{G}}\circ s$. Then the generalized principal connection defined as
\begin{equation*}
\omega^{\overline s}=\sigma ^*\omega - \tilde{s}
\end{equation*}
is flat, where $\tilde{s}=[s,s^*\omega]_\mathcal{G}\in\Gamma(\pi_{T^*X\otimes \sigma ^*\tilde{\varfrak{g}},X})$.
\end{proposition}

\begin{proof}
To begin with, we have that $s(X)\subset Y^\sigma$ is an integral leaf of $\sigma^*\omega-\tilde{s}$. Indeed, it is easy to check that $(ds)_x(T_x X)$ is $\omega^{\overline s}$-horizontal for each $x\in X$. This means that $\operatorname{Curv}\left(\omega^{\overline s}\right)_{s(x)}=0$ for each $x\in X$, where $\operatorname{Curv}\left(\omega^{\overline s}\right)$ is the curvature of $\omega^{\overline s}$.

An analogous computation to \eqref{eq:reducedcurvature} yields
\begin{equation*}
\left[s(x)\cdot g,\operatorname{Curv}\left(\omega^{\overline s}\right)_{s(x)\cdot g}\right]_{\mathcal G}=\left[s(x),\operatorname{Curv}\left(\omega^{\overline s}\right)_{s(x)}\right]_{\mathcal G}=0,\qquad x\in X,~g\in\mathcal G_x,
\end{equation*}
whence $\operatorname{Curv}\left(\omega^{\overline s}\right)_y=0$ for every $y\in Y^\sigma$.
\end{proof}

\begin{theorem}[Reconstruction]\label{theorem:reconstruction}
Let $\overline s\in\Gamma\left(\pi_{(T^*X\otimes\varfrak{g})_{\varphi,\mathfrak{H}},X}\right)$ be a critical section for the variational problem defined in $(iii)$ of Theorem \ref{theorem:reducedequationsmathfrakH}, and $\sigma=\pi_{(T^*X\otimes\varfrak{g})_{\varphi,\mathfrak{H}},Y/\mathcal G}\circ\overline s\in\Gamma(\pi_{Y/\mathcal{G},X})$. Let $\tilde{s}$ be a section of $T^*X\otimes \sigma ^*\tilde{\varfrak{g}}\to X$ such that $\overline s(x)=\llbracket y,\xi_x\rrbracket_{\mathfrak{H},\varphi}$ for each $x\in X$, where $\tilde s(x)=[y,\xi_x]_{\mathcal G}$. If the connection
\[
\omega^{\overline{s}}=\sigma^*\omega-\tilde{s}
\]
is flat and has trivial holonomy, then the integral leaves of the $\omega^{\overline s}$ are critical sections of the variational problem defined by $L$. Furthermore, any critical section of $L$ is obtained in this way.
\end{theorem}

\begin{proof}
The trivial holonomy of $\omega^{\overline s}$ implies that its integral leaves are sections of $\sigma ^*Y\to X$ that project to the reduced section $\overline{s}$. According to Theorem \ref{theorem:reducedequationsmathfrakH}, these sections are critical. 

Conversely, if $s$ is a critical section of $L$, then $\omega^{\overline s}$ is flat and $s$ is an integral leaf of $\omega ^{\overline{s}}$, by Proposition \ref{rec0}.
\end{proof}

\begin{corollary}
If $X$ is simply connected, then any connection has trivial holonomy and we have the following equivalence of equations
\begin{equation*}
    \mathcal{EL}(L)(j^1 s)=0\iff \left\{
    \begin{array}{l}
    \displaystyle\frac{\delta l}{\delta\sigma_s}-\operatorname{div}^{Y/\mathcal G}\left(\frac{\delta l}{\delta j^1\sigma_s}\right)=\left\langle\frac{\delta l}{\delta\overline s},\iota_{d\sigma_s}\tilde\Omega\right\rangle\vspace{0.1cm}\\
    \displaystyle\left(\operatorname{div}^{\tilde{\varfrak g}}-\ad_{\overline s}^*+\tilde\varphi_*^\dagger\right)\left(\frac{\delta l}{\delta\overline s}\right)=0\vspace{0.1cm}\\
    \mathrm{Curv}\left(\omega ^{\overline{s}}\right)=0
    \end{array}
    \right. 
\end{equation*}
\end{corollary}

For non-simply connected manifolds, the equivalence above holds locally only. There are topological obstructions and examples of reduced sections that do not admit global unreduced sections. See \cite[\S 5.2]{Ca2012}.

\begin{remark} Given a reduced section $\overline{s}$, there may be many sections $\tilde{s}$ satisfying the conditions of Theorem \ref{theorem:reconstruction}. Each choice will induce different solution $s$, and the transitions between them are governed by the symmetries of the system, that is, sections of $\mathcal{G}\to X$ with 1-jets lying on $H\to X$.
\end{remark}

%%%%%%%%%%%%%%%%%%%%
\section{Noether's theorem}\label{sec:noethertheorem}

The well known Noether's theorem establishes that infinitesimal symmetries of the Lagrangian density yield preserved quantities for the dynamics of the system. The aim of this section is to show that the vertical part of the reduced equation is equivalent to the Noether's conservation law defined by the action of a Lie group bundle. As before, let $\mathfrak L=Lv$ be an $H$-invariant Lagrangian density for some $H\subset J^1\mathcal G\simeq\mathcal G\times_X(T^*X\otimes\varfrak g)$ as in Proposition \ref{ggg} 

\begin{definition}
An \emph{infinitesimal symmetry} of $\mathfrak L$ is a vector field $U\in\mathfrak X(J^1 Y)$ such that $\pounds_U\mathfrak L=0$ or, equivalently, $\pounds_U\Theta_{\mathfrak L}=0$, where $\pounds$ denotes the Lie derivative and $\Theta_{\mathfrak L}\in\Omega^n(J^1 Y)$ is the (covariant) Cartan form of $\mathfrak L$ (cf., for instance, \cite[\S 5]{Sa1989}).
\end{definition}

If $s\in\Gamma(\pi_{Y,X})$ is a critical section for $\mathfrak L$ and $U\in\mathfrak X(J^1 Y)$ is an infinitesimal symmetry of the Lagrangian density, the Noether's theorem gives (see, for example, \cite{Ga1974}) the following conservation law
\begin{equation*}\label{eq:conservacion}
{\rm d}\left(\left(j^1 s\right)^*\iota_{U}\Theta_{\mathfrak L}\right)=0.
\end{equation*}
In particular, if $\xi\in\Gamma(\pi_{\varfrak g,X})$ is such that its 1-jet extension falls in the Lie algebra bundle $\mathrm{Lie}(H)$ of $H$, it is clear that $(\xi^*)^ {(1)}\in\mathfrak X(J^1 Y)$ is an infinitesimal symmetry of $\mathfrak L$, thanks to the $H$-invariance. Recall that, from Proposition \ref{ggg}, we know that
\begin{equation}\label{eq:liealgebraH}
\mathrm{Lie}(H)=\left\{(\xi,\hat\eta_x)\in\varfrak g\times_X(T^*X\otimes\varfrak g)\mid\hat\eta_x=\varphi_*(\xi)+\eta_x,~\eta_x\in\mathfrak H_x\right\}\subset J^1\varfrak{g}.
\end{equation}

\begin{theorem}
Let $L:J^1Y\to \mathbb{R}$ be an $H$-invariant Lagrangian for some subgroup bundle $H\subset J^1 \mathcal{G}$ as in Proposition \ref{ggg} and let $s$ be a section of $\pi_{Y,X}$. If $\overline{s}=\llbracket s,s^*\omega\rrbracket_{\mathfrak{H},\varphi}$ is the induced section of $\pi_{(T^*X\otimes\varfrak{g})_{\varphi,\mathfrak{H}},X}$ and $l:J^1(Y/\mathcal{G})\times_{Y/\mathcal{G}}(T^*X\otimes\varfrak{g})_{\varphi,\mathfrak{H}}\to\mathbb{R}$ is the reduced Lagrangian, then the Noether conservation law 
$$
{\rm d}((j^1s)^*i_{(\xi^*)^{(1)}}\Theta_{\mathfrak L})=0
$$
holds for any section $\xi\in\Gamma(\pi_{\varfrak{g},X})$ such that $\xi^{(1)}\in\Gamma(\pi_{\operatorname{Lie}(H),X})$ if and only if the vertical reduced equation
$$
\left(\mathrm{div}^{\tilde{\varfrak{g}}}-\ad_{\overline s}^*+\tilde\varphi_*^\dagger\right)\left(\frac{\delta l}{\delta\overline{s}}\right)=0
$$
is satisfied for $\overline{s}$.
\end{theorem}

\begin{proof}
The result being local enables us to work in coordinates. Some conditions on these coordinates will be imposed along the proof.

Let $(x^\mu,y^i,y^\alpha)$ be bundle coordinates for $\pi_{Y,X}$ and consider the corresponding bundle coordinates $(x^\mu,y^i,y^\alpha,v_\mu^i,v_\mu^\alpha)$ for $\pi_{J^1 Y,X}$. Suppose that they are chosen so that $(x^\mu,y^i)$ are bundle coordinates for $\pi_{Y/\mathcal G,X}$ and, thus, $(x^\mu,y^i,v_\mu^i)$ are bundle coordinates for $\pi_{J^1(Y/\mathcal G),X}$. Let $\{B_\alpha: 1\leq\alpha\leq m\}$ be a basis of local sections of $\pi_{\varfrak g,X}$ and $\{B^\alpha: 1\leq\alpha\leq m\}$ be its dual basis. For a fixed $x_0=(x_0^\mu)\in X$, we suppose that, using these coordinates, we have
\begin{equation}\label{eq:mathfrakhlocal}
\mathrm{Lie}(H)_{x_0}=\operatorname{span}\{(dx^\mu)_{x_0}\otimes B_\alpha(x_0): 1\leq\mu\leq r,1\leq\alpha\leq s\}\subset T_{x_0}^*X\otimes\varfrak g_{x_0}
\end{equation}
for some $r$, $1\leq r\leq n$, and $s$, $1\leq s\leq m$. We start by finding the local expression of the infinitesimal symmetries of our Lagrangian density.

We also suppose that $(y^\alpha)$ are normal coordinates of $G$ on a neighbourhood $\mathcal U\subset G$ of the identity element, $1$. This means that there exist $g_{\beta\gamma}^\alpha\in C^\infty(\mathcal U\times\mathcal U)$, $1\leq \alpha,\beta,\gamma\leq m$, such that
\begin{equation*}
y^\alpha\left(\hat g_1\hat g_2\right)=y^\alpha\left(\hat g_1\right)+y^\alpha\left(\hat g_2\right)+g_{\beta\gamma}^\alpha\left(\hat g_1,\hat g_2\right) y^\beta\left(\hat g_1\right) y^\gamma\left(\hat g_2\right),\qquad 1\leq\alpha\leq m,
\end{equation*}
for each $\hat g_1,\hat g_2\in\mathcal U$ such that $\hat g_1\hat g_2\in\mathcal U$. Hence, the infinitesimal generators are given by
\begin{equation*}
\left(B_\beta\right)_y^*=\left(\delta_\beta^\alpha+g_{\gamma\beta}^\alpha(\hat g,1)y^\gamma(\hat g)\right)(\partial_\alpha)_y,\qquad y=(\sigma,\hat g)\in Y/\mathcal G\times\mathcal U,~1\leq\beta\leq m.
\end{equation*}
Denote by $\hat g:\mathbb R^m\rightarrow G$ the inverse of $(y^\alpha): G\rightarrow\mathbb R^m$, and define $f_{\alpha\beta}^\gamma\in C^\infty(\mathbb R^m)$, $1\leq \alpha,\beta,\gamma\leq m$, as
\begin{equation*}
f_{\alpha\beta}^\gamma(y^\alpha)=g_{\alpha\beta}^\gamma(\hat g(y^\alpha),1)
\end{equation*}
Therefore, given $\xi=\xi^\beta B_\beta\in\Gamma(\pi_{\varfrak g,X})$, where $\xi^\beta\in C^\infty(X)$, $1\leq\beta\leq m$, the previous expressions yield
\begin{equation*}
\xi^*(x^\mu,y^i,y^\alpha)=\xi^\beta(x^\mu)\left(\delta_\beta^\alpha+f_{\gamma\beta}^\alpha(y^\alpha)y^\gamma \right)\partial_\alpha.
\end{equation*}
From the formula of 1-jet lift of vector fields (for example, see, \cite{Sa1989}), we get
\begin{equation}\label{eq:xi1}
(\xi^*)^{(1)}= \xi^\beta\left(\delta_\beta^\alpha+f_{\gamma\beta}^\alpha y^\gamma \right)\partial_\alpha+\left(\partial_\mu\xi^\beta\left(\delta_\beta^\alpha+f_{\gamma\beta}^\alpha y^\gamma\right)+v_\mu^\delta\xi^\beta \left(\partial_\delta f_{\gamma\beta}^\alpha y^\gamma+f_{\delta\beta}^\alpha\right)\right)\partial_\alpha^\mu.
\end{equation}

On the other hand, in coordinates the map $\varphi:\mathcal G\to T^*X\otimes\varfrak g$ reads
\begin{equation*}
\varphi(x^\mu,y^\alpha)=\varphi_\mu^\alpha(x^\mu,y^\alpha) dx^\mu\otimes B_\alpha,
\end{equation*}
for some $\varphi_\mu^\alpha\in C^\infty(\mathcal G)$. By denoting $\hat\varphi_{\mu,\beta}^\alpha(x^\mu)=(\partial\varphi_\mu^\alpha/\partial y^\beta)(x^\mu,y^\alpha=0)$, we have
\begin{equation}\label{eq:dvarphilocal}
\varphi_*(\xi)=\xi^\beta\hat\varphi_{\mu,\beta}^\alpha dx^\mu\otimes B_\alpha,\qquad\xi=\xi^\alpha B_\alpha\in\Gamma(\pi_{\varfrak g,X}).
\end{equation}
Similarly, our generalized principal connection, $\omega\in\Omega^1(Y,\varfrak g)$, locally reads
\begin{equation*}
\omega(x^\mu,y^i,y^\alpha)=\left(\omega_\mu^\alpha(x^\mu,y^i)dx^\mu+\omega_i^\alpha(x^\mu,y^i)dy^i+dy^\alpha\right)\otimes B_\alpha,
\end{equation*}
for some functions $\omega_\mu^\alpha,\omega_i^\alpha\in C^\infty(Y/\mathcal G)$, $1\leq\mu\leq n$, $1\leq i\leq k$, $1\leq\alpha\leq m$, being $k\in\mathbb Z^+$ the dimension of the fiber of $\pi_{Y/\mathcal G,X}$. Furthermore, there exist bundle coordinates $(x^\mu,y^i,v_\mu^i;w_\mu^\alpha)$ for the reduced space, where the indices of $w_\mu^\alpha$ go through $\mu\in\{r+1,\dots,n\}$ and $\alpha\in\{s+1,\dots,m\}$, such that identification of Theorem \ref{theorem:descomposicioncociente} is given by
\begin{equation*}
\begin{array}{ccc}
J^1 Y/H & \to & J^1(Y/\mathcal G)\times_{Y/\mathcal G}(T^*X\otimes\varfrak{g})_{\varphi,\mathfrak{H}}\\
\left[x^\mu,y^i,y^\alpha,v_\mu^i,v_\mu^\alpha\right]_H & \mapsto & \left(x^\mu,y^i,v_\mu^i;w_\mu^\alpha\right)
\end{array}
\end{equation*}
and, in the fiber over $x_0$, they satisfy
\begin{equation*}
w_\mu^\alpha=\varphi_\mu^\alpha(x_0^\mu,0)+\omega_\mu^\alpha(x_0^\mu,y^i)+\omega_i^\alpha(x_0^\mu,y^i)v_\mu^i+v_\mu^\alpha,\qquad r+1\leq\mu\leq n,~s+1\leq\alpha\leq m.
\end{equation*}

Likewise, using the definition of reduced Lagrangian, i.e.,
\begin{equation*}
l(x^\mu,y^i,v_\mu^i;w_\mu^\alpha)=L(x^\mu,y^i,y^\alpha,v_\mu^i,v_\mu^\alpha),
\end{equation*}
it is easy to find the local expression of the vertical equation.
Namely, the local expression of the partial derivative of $l$ is
\begin{equation*}
\frac{\delta l}{\delta\overline s}(x^\mu)=\frac{\partial l}{\partial w_\mu^\alpha}\left(x^\mu,s^i(x^\mu),\frac{\partial s^i(x^\mu)}{\partial x^\mu};w_\mu^\alpha(x^\mu)\right)\partial_\mu\otimes B^\alpha,
\end{equation*}
and then, the vertical reduced equation becomes
\begin{equation}\label{eq:vertical}
\left(\operatorname{div}^{\tilde{\varfrak g}}-\ad_{\overline s}^*+\tilde\varphi_*^\dagger\right)\left(\frac{\partial l}{\partial w_\mu^\alpha}\partial_\mu\otimes B^\alpha\right)=0.
\end{equation}
From Lemma \ref{lemma:overlinenablag}, we may locally write $\operatorname{div}^{\tilde{\varfrak g}}=\operatorname{div}^{\varfrak g}+\operatorname{ad}_{\overline s}^*$, where $\operatorname{div}^{\varfrak g}$ is the divergence of $\nabla^{\varfrak g}$. Observe that $\sigma_s^*\left(\tilde{\varfrak g}\right)\simeq\varfrak g$, since we are working locally. By definition, $\mathrm{div}\langle\zeta,\xi\rangle=\left\langle\operatorname{div}^{\varfrak g}\zeta,\xi\right\rangle+\left\langle\zeta,\nabla^{\varfrak g}\xi\right\rangle $, for each $\zeta=\zeta_\alpha^\mu\partial_\mu\otimes B^\alpha\in\Gamma(\pi_{TX\otimes\varfrak g^*,X})$ and $\xi=\xi^\alpha B_\alpha\in\Gamma(\pi_{\varfrak g,X})$. An easy computation using this expression shows that
\begin{equation*}
\operatorname{div}^{\varfrak g}\zeta=\left(\partial_\mu\zeta_\alpha^\mu-\Gamma_{\mu,\alpha}^\beta\zeta_\beta^\mu\right) B^\alpha
\end{equation*}
where $\Gamma_{\mu,\alpha}^\beta\in C^\infty(X)$, $1\leq\mu\leq n$, $1\leq \alpha,\beta\leq m$, are the Christoffel symbols of $\nabla^{\varfrak g}$, i.e.,
\begin{equation*}
\nabla^{\varfrak g}\xi=\left(\partial_\mu\xi^\alpha+\Gamma_{\mu,\beta}^\alpha\xi^\beta\right)dx^\mu\otimes B_\alpha,\qquad\xi=\xi^\alpha B_\alpha\in\Gamma(\pi_{\varfrak g,X}).
\end{equation*}
We choose the bundle coordinates on $\pi_{\varfrak g,X}$ so that $\nabla^{\varfrak g}$ is flat at $x_0=(x_0^\mu)$, i.e., $\Gamma_{\mu,\alpha}^\beta(x_0^\mu)=0$ for each $1\leq\mu\leq n$ and $1\leq \alpha,\beta\leq m$. Analogously, $\left\langle\varphi_*(\xi),\zeta\right\rangle=\left\langle\xi,\varphi_*^\dagger(\zeta)\right\rangle$. Thus, from \eqref{eq:dvarphilocal} we get
\begin{equation*}
\tilde\varphi_*^\dagger(\zeta)=\hat\varphi_{\mu,\alpha}^\beta\zeta_\beta^\mu B^\alpha,\qquad\zeta=\zeta_\alpha^\mu\partial_\mu\otimes B^\alpha\in\Gamma(\pi_{TX\otimes\varfrak g^*,X}).
\end{equation*}
As a result, \eqref{eq:vertical} reads
\begin{equation}\label{localver}
\left(\partial_\mu\left(\frac{\partial l}{\partial w_\mu^\alpha}\right)+\hat\varphi_{\mu,\alpha}^\beta\frac{\partial l}{\partial w_\mu^\beta}\right)(x_0)B^\alpha(x_0)=0,
\end{equation}
where we recall that the summation is for $r+1\leq\mu\leq n$ and $s+1\leq\alpha,\beta\leq m$.

Using the expression for the coordinates $w_\mu^\alpha$ at $x_0$ we have
\begin{equation*}
\begin{array}{cl}
\bullet & \displaystyle\frac{\partial L}{\partial v_\mu^\alpha}(x_0)=\left\{\begin{array}{ll}
0, & 1\leq\mu\leq r,~1\leq\alpha\leq s,\\
\displaystyle\frac{\partial l}{\partial w_\mu^\alpha}(x_0), & r+1\leq\mu\leq n,\quad s+1\leq\alpha\leq m.
\end{array}\right.\vspace{0.2cm}\\
\bullet & \displaystyle\frac{\partial L}{\partial v_\mu^i}(x_0)=\left\{\begin{array}{ll}
\displaystyle\frac{\partial l}{\partial v_\mu^i}(x_0), & 1\leq\mu\leq r,\\
\displaystyle\frac{\partial l}{\partial v_\mu^i}(x_0)+\sum_{\alpha=s+1}^m\frac{\partial l}{\partial w_\mu^\alpha}(x_0)\omega_i^\alpha,\quad & r+1\leq\mu\leq n.
\end{array}\right.
\end{array}
\end{equation*}
In addition, if the volume form is given by $v=d^n x=dx^1\wedge\dots\wedge dx^n$, then the local expression of the Poincaré-Cartan form is
\begin{equation*}
\Theta_{\mathfrak L}=\frac{\partial L}{\partial v_\mu^i}\left(dy^i-v_\mu^i dx^\mu\right)\wedge d^{n-1}x_\mu+\frac{\partial L}{\partial v_\mu^\alpha}\left(dy^\alpha-v_\mu^\alpha dx^\mu\right)\wedge d^{n-1}x_\mu+L\,d^n x,
\end{equation*}
where  $d^{n-1}x_\mu=\iota_{\partial_\mu}d^n x$. By using \eqref{eq:xi1}, we obtain
\begin{align*}
\iota_{(\xi^*)^{(1)}}\Theta_{\mathfrak L}\left(x^\mu,y^i,y^\alpha,v_\mu^i,v_\mu^\alpha\right) & =\displaystyle\frac{\partial L}{\partial v_\mu^\alpha}\left(x^\mu,y^i,y^\alpha,v_\mu^i,v_\mu^\alpha\right)\xi^\beta(x^\mu)\left(\delta_\beta^\alpha+f_{\gamma\beta}^\alpha(y^\alpha)y^\gamma \right) d^{n-1}x_\mu\\
& =\displaystyle\frac{\partial l}{\partial w_\mu^\alpha}\left(x^\mu,y^i,v_\mu^i;w_\mu^\alpha\right)\xi^\beta(x^\mu)\left(\delta_\beta^\alpha+f_{\gamma\beta}^\alpha(y^\alpha)y^\gamma \right) d^{n-1}x_\mu.
\end{align*}

The last condition on the coordinates is the requirement that $s(x_0)=(\sigma_s(x_0),1)\in Y/\mathcal G\times G$, i.e., $s^\alpha(x_0^\mu)=0$, $1\leq\alpha\leq m$. Therefore,
\begin{equation*}
\begin{array}{ccl}
\left(j^1 s\right)^*\left(\iota_{(\xi^*)^{(1)}}\Theta_{\mathfrak L}\right)(x_0^\mu) & = & \displaystyle\frac{\partial l}{\partial w_\mu^\alpha}\left(x_0^\mu,s^i(x_0^\mu),\partial_\mu s^i(x_0^\mu);w_\mu^\alpha(x_0^\mu)\right)\xi^\alpha(x_0^\mu) d^{n-1}x_\mu.
\end{array}
\end{equation*}
Observe that the jet extension of $\xi$ is given by $j^1\xi=(\xi,d\xi)\in\Gamma(\pi_{\varfrak g\times_X(T^*X\otimes\varfrak g),X})$. Subsequently, from \eqref{eq:liealgebraH}, \eqref{eq:mathfrakhlocal} and \eqref{eq:dvarphilocal} we deduce that the condition $\xi^{(1)}\in\Gamma(\pi_{\operatorname{Lie}(H),X})$ at $x_0$ reads
\begin{equation*}
\left(\partial_\mu\xi^\alpha\right)(x_0)=\xi^\beta(x_0)\hat\varphi_{\mu,\beta}^\alpha(x_0),\qquad r+1\leq\mu\leq n,~s+1\leq\alpha\leq m.
\end{equation*}
Therefore, the Noether conservation law reads 
\begin{align*}
{\rm d}\left(\frac{\partial l}{\partial w_\mu^\alpha}\xi^\alpha d^{n-1}x_\mu\right)(x_0) & =\left[\partial_\mu\left(\frac{\partial l}{\partial w_\mu^\alpha}\right)\xi^\alpha+\frac{\partial l}{\partial w_\mu^\alpha}\partial_\mu\xi^\alpha\right](x_0)(d^n x)_{x_0}\\
& =\left[\partial_\mu\left(\frac{\partial l}{\partial w_\mu^\alpha}\right)+\frac{\partial l}{\partial w_\mu^\beta}\hat\varphi_{\mu,\alpha}^\beta\right](x_0)\xi^\alpha(x_0)(d^n x)_{x_0}=0.
\end{align*}
As $\xi(x_0)=\xi^\alpha(x_0)B_\alpha(x_0)\in\varfrak g_{x_0}$ is arbitrary, this is equivalent to \eqref{localver}, and we conclude.
\end{proof}

%%%%%%%%%%%%%%%%%%%%
\section{Examples}\label{sec:examples}

In this final section we discuss several applications of the reduction theory above. The first one consists of recovering the case of rigid symmetries treated in the literature. Then we describe the reduced equations when the system is invariant by the whole jet bundle, $J^1\mathcal G$. The third example is devoted to study Electromagnetism in vacuum, as well as it extension to $k$-forms. We pursue the study the proccess of symmetry breaking by product groups in pure gauge theories. Finally, non-Abelian gauge theories are analized, recovering the Utiyama theorema as well as the Tang--Mills equations.

%%%%%%%%%%
\subsection{Classical case}

The reduction by the action of a Lie group in \cite{ElGaHoRa2011} can be recovered as a particular case of our theory. Namely, let $\pi_{Y,X}$ be a fiber bundle and $G$ be a Lie group acting freely and properly on the right on $Y$, and denote the (standard) action by
\begin{equation*}
\Psi:Y\times G\to Y,\quad(y,\hat g)\mapsto\Psi(y,\hat g)=\Psi_{\hat g}(y)=\Psi_y(\hat g)=y\cdot\hat g.
\end{equation*}
In addition, suppose that $\pi_{Y,X}(y\cdot\hat g)=\pi_{Y,X}(y)$ for every $y\in Y$ and $\hat g\in G$. This action may be regarded as a fibered action by the trivial Lie group bundle,
\begin{equation*}
\mathcal G=X\times G,
\end{equation*}
by setting $\Phi(y,g)=\Psi_{\hat g}(y)=y\cdot\hat g$ for each $(y,g)\in Y\times_X\mathcal G$, with $g=(x,\hat g)$.  It is clear that the fibered quotient and the usual quotient agree, i.e. $Y/\mathcal G \simeq Y/G$.

Recall that $\pi_{Y,Y/G}$ is a principal bundle and the Lie algebra bundle of $\pi_{\mathcal G,X}$ is $\varfrak g=X\times\mathfrak g$, where $\mathfrak g$ is the Lie algebra of $G$. From \cite[Proposition 4.1.]{CaRo2023} we know that a generalized principal connection $\omega\in\Omega^1(Y,\varfrak g)$ on $\pi_{Y/\mathcal G}$ associated to the trivial connection $\nu_0$ on $\pi_{\mathcal G,X}$ is simply a principal connection $A\in\Omega^1(Y,\mathfrak g)$. Thanks to the bijective correspondence between (local) sections of $\pi_{\mathcal G,X}$ and (local) functions $X\rightarrow G$, the identification \eqref{eq:Thetanu} for $\nu_0$ may be written as
\begin{equation*}
J^1 \mathcal G\simeq G\times T^*X\otimes\mathfrak g,\qquad j^1_x\gamma\mapsto\left(\hat\gamma(x),d(R_{\hat\gamma(x)^{-1}}\circ\hat\gamma)_x\right),
\end{equation*}
where $\gamma=(\operatorname{id}_{X},\hat\gamma)$ and $R_g:G\to G$ denotes the right multiplication by $g\in G$. Let $H$ be the Lie group subbundle of $\pi_{J^1 \mathcal G,X}$ corresponding to locally constant functions, i.e., 
\begin{equation*}
H\simeq G\times T^*X\otimes\{0\}\simeq X\times G=\mathcal G.
\end{equation*}
The first jet extension of the fibered action restricted to $\pi_{H,X}$ (recall Equation \eqref{eq:accionjet}) and the jet extension of the Lie group action (cf. \cite[Equation 2.7]{ElGaHoRa2011}) yield the same quotient:
\begin{equation*}\label{eq:clasico1}
J^1 Y/H\simeq J^1 Y/G.
\end{equation*}
Since $\varphi=0$ and $\mathfrak{H}=\{0\}$, the quotient \eqref{eq:cocienteafin} is isomorphic to 
\begin{equation*}\label{eq:clasico2}
(T^*X\otimes\varfrak{g})_{0,\mathfrak{H}} \overset{\sim}{\to} T^*X\otimes\tilde{\mathfrak g},\quad\left\llbracket y,(x,\hat X_x)\right\rrbracket_{\mathfrak{H},0}\mapsto\big[y,\hat X_x\big]_{G},
\end{equation*}
where $\tilde{\mathfrak g}=(Y\times\mathfrak g)/G$ is the adjoint bundle of $\pi_{Y,Y/G}$.

\begin{theorem}\label{theorem:descomposicionclasica}
Let $A\in\Omega^1(Y,\mathfrak g)$ be a principal connection on $\pi_{Y,Y/G}$ and consider the corresponding generalized principal connection $\omega\in\Omega^1(Y,\varfrak g)$ on $\pi_{Y,Y/\mathcal G}$ associated to the canonical connection $\nu_0$ on $\pi_{\mathcal G,X}$. Then the fiber diffeomorphism given in Theorem \ref{theorem:descomposicioncociente} reduces to the identification given in \cite[Equation (2.8)]{ElGaHoRa2011} under the previous identifications, i.e.,
\begin{equation*}
J^1 Y/G\to J^1(Y/G)\times_{Y/G}\left(T^*X\otimes\tilde{\mathfrak g}\right),\quad\left[j^1_x s\right]_{ G}\mapsto\left(j^1_x\sigma_s,\left[s(x),A_{s(x)}\circ(ds)_x\right]_ G\right).
\end{equation*}
\end{theorem}

Let $A\in\Omega^1(Y,\mathfrak g)$ be a principal connection on $\pi_{Y,Y/G}$, $\nabla_0^{\varfrak g}$ be the canonical connection on $\pi_{\varfrak g,X}$ and $\nabla^X$ be a linear connection on $\pi_{TX,X}$, and consider the linear connection $\nabla^\otimes=\nabla^*\otimes\nabla_0^{\varfrak g}$ on $\pi_{T^*X\otimes\varfrak g,X}$, where $\nabla^*$ is the dual of $\nabla^X$. They yield a linear connection $\nabla$ on $\pi_{(T^*X\otimes\varfrak{g})_{0,\mathfrak{H}},Y/\mathcal G}\simeq\pi_{T^*X\otimes\tilde{\mathfrak g},Y/G}$, as we have seen in Proposition \ref{prop:inducedconnection}. By using bases of local sections of these vector bundles, it can be seen that $\nabla$ agrees with the linear connection defined in \cite[Equation (3.19)]{ElGaHoRa2011}. In addition, let $\nabla^{A}$ be the linear connection on $\pi_{\tilde{\mathfrak g},Y/G}$ associated to $A$ (cf. \cite[Equation (2.6)]{ElGaHoRa2011}). It descends to an operator $\overline\nabla^{A}:\Gamma(\pi_{\tilde{\mathfrak g},X})\rightarrow\Gamma(\pi_{T^*X\otimes \tilde{\mathfrak g},X})$ as in \cite[Equation (3.7)]{ElGaHoRa2011}, whose \emph{divergence} is denoted by:
\begin{equation*}
\operatorname{div}^{A}:\Gamma(\pi_{TX\otimes \tilde{\mathfrak g}^*,X})\to\Gamma(\pi_{\tilde{\mathfrak g}^*,X}).
\end{equation*}
Lastly, let $\nabla^{Y/\mathcal G}$ be a torsion free linear connection on $\pi_{T(Y/\mathcal G),Y/\mathcal G}$ projectable onto $\nabla^X$ and consider the induced affine connection $\nabla^{J^1(Y/\mathcal G)}$ on $\pi_{J^1(Y/\mathcal G),Y/\mathcal G}$.

Let $L: J^1 Y\rightarrow\mathbb R$ be a $G$-invariant Lagrangian and $l: J^1(Y/G)\times_{Y/G}(T^*X\otimes\tilde{\mathfrak g})\rightarrow\mathbb R$ be the reduced Lagrangian. As in the general theory, for the sake of simplicity we suppose that $X$ is compact. Let $s\in\Gamma\left(\pi_{Y,X}\right)$ and consider the reduced section $\overline s=[s,s^*A]_{G}\in\Gamma\left(\pi_{T^*X\otimes\tilde{\mathfrak g},X}\right)$. Observe that $\mathfrak{H}^\perp=TX\otimes\mathfrak g^*$ and, hence, $Y\times_{\mathcal G}\mathfrak{H}^\perp\simeq TX\otimes\tilde{\mathfrak g}^*$. Given a reduced section $\overline s\in\Gamma(\pi_{T^*X\otimes \tilde{\mathfrak g},X})$, the coadjoint representation of $\mathfrak g^*$ induces a map: 
\begin{equation*}
\ad_{\overline s}^*:\Gamma(\pi_{TX\otimes \tilde{\mathfrak g}^*,X})\to\Gamma(\pi_{\tilde{\mathfrak g}^*,X}),
\end{equation*}
which is well-defined for every $\overline\zeta\in\Gamma(\pi_{TX\otimes \tilde{\mathfrak g}^*,X})$ such that:
\begin{equation*}
\pi_{TX\otimes \tilde{\mathfrak g}^*,Y/G}\circ\overline\zeta=\pi_{T^*X\otimes \tilde{\mathfrak g},Y/G}\circ\overline s.
\end{equation*}

\begin{theorem}
Let $s\in\Gamma\left(\pi_{Y,X}\right)$ and $\overline s=[s,s^*A]_{G}\in\Gamma\left(\pi_{T^*X\otimes\tilde{\mathfrak g},X}\right)$ be the reduced section. Then the reduced equations for $\overline s$ given in Theorem \ref{theorem:reducedequationsmathfrakH} are equivalent to the Lagrange--Poincaré field equations given in \cite[Theorem 3.5]{ElGaHoRa2011}, i.e.,
\begin{equation*}
\left\{\begin{array}{l}
\displaystyle\frac{\delta l}{\delta \sigma_s}-\operatorname{div}^{Y/\mathcal G}\left(\frac{\delta l}{\delta j^1\sigma_s}\right)=\left\langle\frac{\delta l}{\delta\overline s},\iota_{d\sigma_s}\tilde F^A\right\rangle,\vspace{0.1cm}\\
\displaystyle\operatorname{div}^{A}\left(\frac{\delta l}{\delta \overline s}\right)-\operatorname{ad}_{\overline s}^*\left(\frac{\delta l}{\delta\overline s}\right)=0,
\end{array}\right.
\end{equation*}
where $\tilde F^A\in\Omega^2(Y/G,\tilde{\mathfrak g})$ is the reduced curvature of $A$. 
\end{theorem}

%%%%%%%%%%
\subsection{Full jet symmetry}
Let $\pi_{\mathcal G,X}$ be a Lie group bundle endowed with a Lie group connection $\nu$, and suppose that it acts (on the right) on a fiber bundle $\pi_{Y,X}$ freely and properly. In this example we consider a Lagrangian $L: J^1 Y\rightarrow\mathbb R$ which is invariant by the whole jet bundle, that is,
\begin{equation*}
H=J^1\mathcal G\overset{\nu}{\simeq}\mathcal G\times_X(T^*X\otimes\varfrak g).
\end{equation*}
It is clear that $(T^*X\otimes\varfrak{g})_{\varphi,\mathfrak{H}}=X$. Consequently, the identification of Theorem \ref{theorem:descomposicioncociente} can be performed without fixing a generalized principal connection:
\begin{equation}\label{eq:quotientfulljets}
\left(J^1 Y\right)/J^1 \mathcal G\to J^1(Y/\mathcal G),\quad\left[j_x^1s\right]_{J^1\mathcal G}\mapsto j_x^1\sigma_s.
\end{equation}

In the same way, since the reduced section $\overline s$ vanishes, the second equation of \emph{(iv)} in Theorem \ref{theorem:reducedequationsmathfrakH} does not appear. Therefore, the reduced equations are the usual \emph{Euler--Lagrange equations} for the reduced Lagrangian $l: J^1(Y/\mathcal G)\rightarrow\mathbb R$:
\begin{equation*}
\frac{\delta l}{\delta\sigma_s}-\operatorname{div}^{Y/\mathcal G}\left(\frac{\delta l}{\delta j^1\sigma_s}\right)=0.
\end{equation*}
Of course, in order to write them we have fixed a linear connection $\nabla^{Y/\mathcal G}$ on $\pi_{T(Y/\mathcal G),Y/\mathcal G}$ projectable onto a linear connection $\nabla^X$ on $\pi_{TX,X}$.

%%%%%%%%%%
\subsection{Electromagnetism in vacuum}\label{sec:em}

To describe Electromagnetism in vacuum as an Abelian geometric Yang--Mills theory, let $(X,g)$ be a 4-dimensional, compact, oriented pseudo-Riemannian manifold with volume form $v_g\in\Omega^4(X)$, and let $\pi_{P,X}:P\to X$ be a principal $\operatorname{U}(1)$-bundle. The configuration bundle of this theory is the bundle of connections of $\pi_{P,X}$, i.e.,
$$
C(P)=\left.J^1 P\right/\operatorname{U}(1)\to X.
$$
Recall that it is an affine bundle modelled on $T^*X\to X$.

\begin{definition}
The Maxwell Lagrangian density for Electromagnetism in vacuum is $\mathfrak L=L\,v_g$ with:
\begin{equation*}
L\left(j^1 A\right)=g\left(\tilde F^A,\tilde F^A\right),\qquad A\in\Gamma(\pi_{C(P),X}),
\end{equation*}
where $\tilde F^A\in\Omega^2(X)$ is the \emph{reduced curvature} of the principal connection $A$.
\end{definition}

Gauge transformations are defined by maps $g:X\to \operatorname{U}(1)$, that is, sections of the trivial bundle $X\times \operatorname{U}(1) \to X$. The Maxwell Lagrangian is invariant with respect to the transformation
\[
A_x\mapsto A_x + (dg)_xg(x)^{-1},\qquad x\in X,
\]
which can be understood as a symmetry by the fibered action $C(P)\times J^1(X,\operatorname{U}(1)) \to C(P)$ defined above. Note that the value $g(x)\in \operatorname{U}(1)$ does not play an essential role and that  $(dg)_x\,g(x)^{-1}\in T_x^* X\otimes\mathfrak u(1)\simeq T_x^* X$. Subsequently, this fibered action induces another one,
\begin{equation}\label{eq:accioncotagenteconexiones}
C(P)\times_X T^*X\to C(P),\quad\left(A_x,\alpha_x\right)\mapsto A_x+\alpha_x,
\end{equation}
which can be straightforwardly extended to the first jets, $J^1 C(P)\times_X J^1(T^*X)\rightarrow J^1 C(P)$. By virtue of this, the Lie group bundle of symmetries for Electromagnetism is the cotangent bundle of $X$ with the additive structure,
$$
\mathcal{G}=T^*X\to X.
$$

Second derivatives of gauge transformations correspond to the Lie group subbundle $J^2(X,\operatorname{U}(1))\subset J^1\left(J^1(X,\operatorname{U}(1))\right)$. One can readily obtain the corresponding subbundle of jet  symmetries as
\begin{equation*}
H=\left\{j_x^1\alpha\in J^1(T^*X): ({\rm d}\alpha)_x=0\right\}\subset J^1(T^*X).
\end{equation*}
Note that $C(P)/T^*X=X$. In order to study the quotient $J^1C(P)/H$ we need to consider a generalized principal connection on $\pi_{C(P),X}$. The following lemma can be easily proven using local coordinates.

\begin{lemma}
Let $\nu$ be a linear connection on $\pi_{T^*X,X}$, then it is Lie group bundle connection and any affine connection $\omega$ on $\pi_{C(P),X}$ modelled on $\nu$ is a generalized principal connection associated to $\nu$. Furthermore, if the connection $\nu$ is torsionless, then $\hat\nu(T^*X)\subset H$, where $\hat\nu\in\Gamma \left(\pi_{J^1(T^*X),T^*X}\right)$ 
is the section induced by $\nu$.
\end{lemma}

The identification \eqref{eq:Thetanu} reads $J^1(T^*X)\overset{\nu}{\simeq}T^*X\times_X(T^*X\otimes T^*X)$ and for a torsionless connection $\nu$ we have
\begin{equation}\label{eq:Hnusimetrico}
\textstyle H\overset{\nu}{\simeq}T^*X\times_X\bigvee^2 T^*X.
\end{equation}
Subsequently, in Proposition \ref{ggg} we may choose $\varphi=0$ and $\mathfrak H=\bigvee^2 T^*X$.

\begin{proposition}\label{prop:identificacionElectromagnetismo}
Let $\omega$ be the affine connection on $\pi_{C(P),X}$ induced by a torsionless linear connection $\nu$ on $\pi_{T^*X,X}$. Then the isomorphism of Theorem \ref{theorem:descomposicioncociente} is the curvature mapping (up to a minus sign), that is,
\begin{equation*}
J^1 C(P)/H\to\textstyle\bigwedge^2 T^*X,\quad\left[j_x^1 A\right]_H\mapsto\operatorname{Skew}\left(A^*\omega\right)_x=-\tilde F_x^A,
\end{equation*}
where $\operatorname{Skew}: T^* X\otimes T^* X\rightarrow\bigwedge^2 T^* X$ is the \emph{skew-symmetrization}.
\end{proposition}

\begin{proof}
Firstly, note that $(T^*X\otimes T^*X)/\bigvee^2 T^*X\simeq \bigwedge^2 T^*X$ via the skew-symmetrization. Likewise, as $\varphi=0$ and the Lie group is Abelian, the action \eqref{eq:affineaction} reads
\begin{equation*}
\left(C(P)\times_X\textstyle\bigvee^2 T^*X\right)\times_X T^*X\to C(P)\times_X\textstyle\bigvee^2 T^*X,\quad\left(\left(A_x,\xi_x\right),\eta_x\right)\mapsto\left(A_x+\eta_x,\xi_x\right).
\end{equation*}
Subsequently, $\ad_{\mathfrak H,0}(C(P))\simeq\bigwedge^2 T^*X$, whence the isomorphism of Theorem \ref{theorem:descomposicioncociente} becomes $J^1 C(P)/H\simeq\bigwedge^2 T^*X$, since $C(P)/T^*X\simeq X$. To conclude, by using local coordinates it can be checked that $\operatorname{Skew}(A^*\omega)=-\tilde F^A$ for each $A\in\Gamma(\pi_{C(P),X})$.
\end{proof}

Therefore, for each $ \tilde F^A\in\Gamma\left(\pi_{\bigwedge^2 T^*X,X}\right)$, the reduced Lagrangian is given by
\begin{equation*}
\textstyle l\left(\tilde F^A\right)=g\left(\tilde F^A,\tilde F^A\right),
\end{equation*}
for which the partial derivative is\footnote{
Observe that $\left(\bigwedge^2 T^*X\right)^*=\bigwedge^2 TX=\left(\bigvee^2 T^*X\right)^\perp=\left((T^*X\otimes T^*X)\left/\bigvee^2 T^*X\right.\right)^*$
}
\begin{equation*}
\frac{\delta l}{\delta \tilde F^A}=2\,\iota_{\tilde F^A}g\in\Gamma\left(\pi_{\bigwedge^2 TX,X}\right).
\end{equation*}

Let $\nabla^*:\Gamma(\pi_{T^*X,X})\rightarrow\Gamma(\pi_{T^*X\otimes T^*X,X})$ be the covariant derivative on $\pi_{T^*X,X}$ corresponding to the linear connection $\nu$ and $\operatorname{div}^*:\Gamma(\pi_{TX\otimes TX,X})\rightarrow\Gamma(\pi_{TX,X})$ be its divergence. From Theorem \ref{theorem:reducedequationsmathfrakH} we know that if $A\in\Gamma(\pi_{C(P),X})$ is a solution of the Euler--Lagrange equations for $L$, then the reduced section $\tilde F^A\in\Gamma\left(\pi_{\bigwedge^2 T^*X,X}\right)$ satisfies the following reduced equation,
\begin{equation}\label{eq:reducedeqsElectromagnetism}
\operatorname{div}^*\left(\iota_{\tilde F^A}g\right)=0.
\end{equation}

Recall that the \emph{Hodge star operator} $\star:\Omega^k(X)\rightarrow\Omega^{4-k}(X)$ is defined implicitly as:
\begin{equation}\label{eq:hodge}
\alpha\wedge\star\beta=g(\alpha,\beta)\,v_g,\qquad \alpha,\beta\in\Omega^k(X),
\end{equation}
and it satisfies $\star\star=(-1)^{k(4-k)}\epsilon(g)$ on $\Omega^k(X)$, being $\epsilon(g)$ the parity of the signature of $g$.

\begin{theorem}[Maxwell equations]
In the above conditions, the reduced equation \eqref{eq:reducedeqsElectromagnetism} is equivalent to the \emph{Maxwell equation in vacuum}, that is:
\begin{equation*}
{\rm d}^\star \tilde F^A=0,
\end{equation*}
where ${\rm d}^\star=\star\circ{\rm d}\circ\star:\Omega^k(X)\rightarrow\Omega^{k-1}(X)$ denotes the codifferential.
\end{theorem}

The equivalence is proved in local charts, making use of the isomorphisms $\sharp: T^*X\rightarrow TX$ and $\flat: TX\rightarrow T^*X$ implicitly defined as $g\left(\alpha^\sharp,U\right)=g\left(\alpha, U_\flat\right)=\alpha(U),\qquad (\alpha,U)\in T^* X\times_X T X$.

\begin{theorem}[Reconstruction]
Let $\mathcal U\subset X$ be a simply connected domain and let $F\in\Gamma\left(\mathcal U,\pi_{\bigwedge^2 T^*X,X}\right)$ be a solution of the Maxwell's equations in vacuum. Then there exists a solution $A\in\Gamma\left(\mathcal U,\pi_{C(P),X}\right)$ of the Euler--Lagrange equations for $L$ such that $F=\tilde F^A$ if and only if the following \emph{compatibility condition} holds,
\begin{equation*}
{\rm d}F=0.
\end{equation*}
\end{theorem}

In short, we have the following local equivalence for sections $A\in\Gamma(\pi_{C(P),X})$:
\begin{equation*}
\mathcal{EL}(L)\left(j^1 A\right)=0\qquad\Longleftrightarrow\qquad\left\{\begin{array}{l}
{\rm d}^\star \tilde F^A=0,\\
{\rm d} \tilde F^A = 0.
\end{array}\right.
\end{equation*}

%%%%%
\subsubsection{\texorpdfstring{$k\,$}{k}-form Electromagnetism}

When the principal bundle $P$ is trivial $P=X\times \mathrm{U}(1)$, the bundle of connections is $T^*X$ and the Maxwell formulation of the previous section is defined on 1-forms. This can be generalized to $k$-form Electromagnetism, $k\in \mathbb{N}$ (see for example \cite{navarrosancho2012}). In this case, both the configuration and Lie group bundle are the same so that the fiberwise actions is
\begin{equation*}
\textstyle\bigwedge^{k}T^*X\times_X\textstyle\bigwedge^{k}T^*X\to\textstyle\bigwedge^{k}T^*X,\quad(A_x,\alpha_x)\mapsto A_x+\alpha_x.
\end{equation*}
A generalized principal connection $\omega\in\Omega^1\left(\bigwedge^{k}T^*X,\bigwedge^{k}T^*X\right)$ on $\pi_{\bigwedge^{k}T^*X,X}$ is just a linear connection on that vector bundle, and it is associated to itself. The corresponding isomorphism \eqref{eq:Thetanu} is
\begin{equation*}
J^1\left(\textstyle\bigwedge^{k}T^*X\right)\to\textstyle\bigwedge^{k}T^*X\oplus\left(T^*X\otimes\textstyle\bigwedge^{k}T^*X\right),\quad j_x^1\alpha\mapsto\left(\alpha(x),(\alpha^*\omega)_x\right).
\end{equation*}
Analogous to classical Electromagnetism, we pick the Lie group subbundle of closed forms and we restrict the previous isomorphism to it, i.e.,
\begin{equation*}
\textstyle H=\left\{j_x^1\alpha\in J^1\left(\bigwedge^{k}T^*X\right): ({\rm d}\alpha)_x=0\right\}\overset{\omega}{\simeq}\bigwedge^{k}T^*X\oplus\left(T^*X\vee \bigwedge^{k}T^*X\right).
\end{equation*}

By using local coordinates, it can be shown that the identification of Theorem \ref{theorem:descomposicioncociente} reads (compare to Proposition \ref{prop:identificacionElectromagnetismo})
\begin{equation*}
\left.J^1\left(\textstyle\bigwedge^{k}T^*X\right)\right/H\to\textstyle\bigwedge^{k+1}T^*X,\quad\left[j_x^1 A\right]_H\mapsto\operatorname{Skew}(A^*\omega)_x=-({\rm d}A)_x.
\end{equation*}

The Yang--Mills Lagrangian $L: J^1\left(\bigwedge^{k}T^*X\right)\rightarrow\mathbb R$ is defined as
\begin{equation*}
L\left(j^1 A\right)=g({\rm d}A,{\rm d}A),\qquad A\in\Gamma\left(\pi_{\bigwedge^{k+1}T^*X,X}\right).
\end{equation*}
It is $H$-invariant, so we may consider the reduced Lagrangian $l:\bigwedge^{k+1}T^*X\rightarrow\mathbb R$. Namely, it is given by
\begin{equation*}
l(C)=g(C,C),\qquad C\in\Gamma\left(\pi_{\bigwedge^{k+1}T^*X,X}\right).
\end{equation*}
Fixed $C\in\Gamma\left(\pi_{\bigwedge^{k+1}T^*X,X}\right)$, the partial derivative is $\delta l/\delta C=2\,\iota_C g\in\Gamma\left(\pi_{\bigwedge^{k+1}TX,X}\right)$ and, hence, the reduced equation is
\begin{equation*}
\operatorname{div}^*\left(\iota_C g\right)=0.
\end{equation*}
In the previous expression, $\operatorname{div}^*:\Gamma\left(\pi_{TX\otimes\bigwedge^{k}TX}\right)\rightarrow\Gamma\left(\pi_{\bigwedge^{k}TX,X}\right)$ is the divergence of the linear connection $\omega$. Similarly to classical Electromagnetism, this reduced equation is equivalent to the Maxwell equations,
\begin{equation*}
{\rm d}^\star C=0.
\end{equation*}

%%%%%%%%%%
\subsection{Symmetry breaking by product groups}

In this example, we consider a gauge theory whose structure group is a direct product and we suppose that the gauge symmetry is broken to the subgroup given by one of the factors. More specifically, let $\pi_{P,X}:P\to X$ be a principal $G$-bundle with 
\begin{equation*}
G=N\times \operatorname{U}(1),    
\end{equation*}
being $N$ a semisimple Lie group, and denote by $\mathfrak g=\mathfrak{n}\oplus\mathbb R$ the corresponding Lie algebra. Consider the Lie subgroup $G_0=\{1\}\times \operatorname{U}(1)\simeq \operatorname{U}(1)$. Since action of $G$ on $G_0$ by conjugation is trivial, we have $\operatorname{Ad}(P)=(P\times G)/G\simeq(P\times H)/G\times \operatorname{U}(1)$. Hence, we may write:
\begin{equation*}
J^1\operatorname{Ad}(P)\simeq J^1\left(\frac{P\times N}{G}\right)\times_X J^1(X,\operatorname{U}(1)).
\end{equation*}

As a usual gauge theory, the configuration bundle is the bundle of connections of $\pi_{P,X}$, that is, $\pi_{C(P),X}$, but we suppose that the symmetry is broken to $G_0$, i.e., we only consider gauge transformations coming from elements of this subgroup. In other words, we restrict the right fibered action $C(P)\times_X J^1\operatorname{Ad}(P)\rightarrow C(P)$ to the Lie group subbundle $J^1(X,\operatorname{U}(1))\simeq\{1\}\times_X J^1(X,\operatorname{U}(1))\subset J^1 \operatorname{Ad}(P)$. A quick computation shows that it is given by:
\begin{equation*}
C(P)\times_X J^1(X,\operatorname{U}(1))\to C(P),\quad\left(A_x,j_x^1g\right)\mapsto A_x+(dg)_x g(x)^{-1}.
\end{equation*}
Analogous to Electromagnetism in vacuum, $T^*X$ may be taken as the Lie group bundle of symmetries,
\begin{equation}\label{eq:accionruptura}
C(P)\times_X T^*X\to C(P),\quad\left(A_x,\alpha_x\right)\mapsto A_x+\alpha_x.
\end{equation}
A trivialization of $\pi_{P,X}$ enables us to prove the following result.

\begin{lemma}
Let $P_0=P/G_0$, which is a principal $N$-bundle over $X$, and consider the corresponding bundle of connections, $C(P_0)\to X$. Then there exists a bundle isomorphism given by
\begin{equation*}
C(P)/T^*X\to C(P_0),\quad\left[A_x\right]_{T^*X}\mapsto(A_0)_x.
\end{equation*}
\end{lemma}

On the other hand, from the jet extension of the fibered action \eqref{eq:accionruptura}, we are only interested in elements coming from $J^2(X,\operatorname{U}(1))\subset J^1\left(J^1(X,\operatorname{U}(1))\right)$, which correspond to the Lie group subbundle
\begin{equation*}
H=\left\{j_x^1\alpha\in J^1(T^*X):({\rm d}\alpha)_x=0\right\}\subset J^1(T^*X).
\end{equation*}
Let $\nu$ be a linear connection on $\pi_{T^*X,X}$ such that $\hat\nu(T^*X)\subset H$. Note that equation \eqref{eq:Hnusimetrico} is also valid for this case, so we may choose $\varphi=0$ and, hence, $\mathfrak{H}=\bigvee^2 T^*X$. Similarly, we have
\begin{equation*}
\ad_{\mathfrak H,0}(C(P))\simeq C(P_0)\times_X\textstyle\bigwedge^2 T^*X.
\end{equation*}
A slight modification of the proof of Proposition \ref{prop:identificacionElectromagnetismo} leads to the following result.

\begin{proposition}
Let $\omega\in\Omega^1(C(P),T^*X)$ be a generalized principal connection on $\pi_{C(P),C(P_0)}$ associated to $\nu$. Then the identification of Theorem \ref{theorem:descomposicioncociente} reads
\begin{equation*}
J^1 C(P)/H\to J^1C(P_0)\times_X C(P_0)\times_X\textstyle\bigwedge^2 T^*X,\quad\left[j_x^1 A\right]_H\mapsto\left(j_x^1A_0,A_0(x),-\tilde F_x^A\right).
\end{equation*}
\end{proposition}

Thanks to this identification, for each section $A\in\Gamma(\pi_{C(P),X})$ we define the reduced section as 
\begin{equation*}
\overline A=\left(A_0,\tilde F^A\right)\in\Gamma\left(\pi_{C(P_0)\times_X\bigwedge^2 T^*X,X}\right).
\end{equation*}

Let $L: J^1C(P)\rightarrow\mathbb R$ be an $H$-invariant Lagrangian density and consider the reduced Lagrangian, $l: J^1C(P_0)\times_X C(P_0)\times_X\bigwedge^2 T^*X\rightarrow\mathbb R$. Let $\nabla^0$ be a torsion free linear connection on $T\left(C(P_0)\right)$ projectable onto a linear connection $\nabla^X$ on $TX$. We know that it induces an affine connection $\nabla^{(1)}$ on $\pi_{J^1C(P_0),C(P_0)}$. In addition, we assume that $\nu$ is the dual connection of $\nabla^X$. These connections induce an affine connection on the reduced space, as described in Proposition \ref{prop:inducedconnection}.
The partial derivatives of the reduced Lagrangian are\footnote{
Recall that $\left(\bigwedge^2 T^*X\right)^*=\bigwedge^2 TX=\left((T^*X\otimes T^*X)\left/\bigvee^2 T^*X\right.\right)^*$.
}
\begin{equation*}
\frac{\delta l}{\delta A_0}\in\Gamma\left(\pi_{T^*\left(C(P_0)\right),X}\right),\quad\frac{\delta l}{\delta j^1 A_0}\in\Gamma\left(\pi_{TX\otimes V^*\left(C(P_0)\right),X}\right),\quad\frac{\delta l}{\delta \tilde F^A}\in\Gamma\left(\pi_{\bigwedge^2 TX,X}\right).
\end{equation*}
Let $\nabla^*:\Gamma(\pi_{T^*X,X})\rightarrow\Gamma(\pi_{T^*X\otimes T^*X,X})$ be the covariant derivative associated to $\nu$ and denote by $\operatorname{div}^*:\Gamma(\pi_{TX\otimes TX,X})\rightarrow\Gamma(\pi_{TX,X})$ its divergence. Likewise, let $\operatorname{div}^0:\Gamma(\pi_{TX\otimes V^*\left(C(P_0)\right),X})\rightarrow\Gamma(\pi_{V^*\left(C(P_0)\right),X})$ be the divergence of the operator $\overline\nabla^0$ defined from $\nabla^0$ in (\ref{eq:overlinenablaEG}). The reduced equations are straightforwardly obtained from \emph{(iv)} of Theorem \ref{theorem:reducedequationsmathfrakH}.

\begin{theorem}
Let $\omega\in\Omega^1(C(P),T^*X)$ be a generalized principal connection on $\pi_{C(P),C(P_0)}$ associated to $\nu$, and $A\in\Gamma(\pi_{C(P),X})$ be a solution of the Euler--Lagrange equations for $L$. Then the reduced section $\overline A=\left(A_0,\tilde F^A\right)\in\Gamma(\pi_{C(P_0)\times_X\bigwedge^2 T^*X,X})$ satisfies the following reduced equations:
\begin{equation*}
\left\{
\begin{array}{l}
\displaystyle\frac{\delta l}{\delta A_0}-\operatorname{div}^0\left(\frac{\delta l}{\delta j^1 A_0}\right)=\left\langle\frac{\delta l}{\delta \tilde F^A},\iota_{dA_0}\tilde\Omega\right\rangle,\vspace{0.1cm}\\
\displaystyle\operatorname{div}^*\left(\frac{\delta l}{\delta \tilde F^A}\right)=0.
\end{array}
\right.
\end{equation*}
where $\tilde\Omega\in\Omega^2(C(P_0),T^*X)$ is the reduced curvature of $\omega$.
\end{theorem}

\subsection{Non-Abelian gauge theories}

Electromagnetism (section \ref{sec:em} above) is an instance of an Abelian gauge theory. We now extend these result to non-Abelian gauge theories. In particular, we show that the Utiyama theorem (see \cite{Ut1956} for the original version, and \cite{Ga1977} for the geometric version) and the Yang--Mills equations may be obtained through the gauge reduction process described in this article.

\subsubsection{Geometric formulation}

Let $\pi_{P,X}:P\to X$ be a (standard) principal bundle with a semisimple structure group  $G$. The configuration bundle of gauge theories is the bundle of connections:
\begin{equation*}
\pi_{C(P),X}:C(P)=\left(J^1 P\right)/G\to X,    
\end{equation*}
whose sections define principal connections on $\pi_{P,X}$, and which is an affine bundle modelled on $\pi_{T^*X\otimes\mathfrak g,X}:T^*X\otimes\tilde{\mathfrak g}\to X$, the bundle of covectors taking values in the adjoint bundle $\tilde{\mathfrak g}=(P\times\mathfrak g)/G$, where $\mathfrak g$ is the Lie algebra of $G$ and the action of $G$ on $\mathfrak g$ is given by the adjoint representation.

The main instance in this framework is the \emph{Yang--Mills theory}. In this case, given a pseudo-Riemannian metric, $g$, on $X$ and the Killing metric, $K:\mathfrak g\times\mathfrak g\to\mathbb R$, on the Lie algebra, we first define a fibered inner product:
\begin{equation*}
\begin{array}{rccl}
\left\langle\cdot,\cdot\right\rangle_g: & \left(\bigwedge^2 T^*X\otimes\tilde{\mathfrak g}\right)\times_X\left(\bigwedge^2 T^*X\otimes\tilde{\mathfrak g}\right)& \to & \mathbb R\\
& \left(\alpha_1\otimes[p,\xi_1]_G,\alpha_2\otimes[p,\xi_2]_G\right) & \mapsto & g(\alpha_1,\alpha_2)\,K(\xi_1,\xi_2).
\end{array}
\end{equation*}
Then the Yang--Mills Lagrangian density is $\mathfrak L=L\,v_g$, where $v_g\in\Omega^n(X)$ is the pseudo-Riemannian volume form and
\begin{equation}\label{eq:ymdensity}
L:J^1 C(P)\to\mathbb R,\quad j_x^1 A\mapsto\left\langle\tilde F_x^A,\tilde F_x^A\right\rangle_g, 
\end{equation}
where $\tilde F^A\in\Omega^2(X,\tilde{\mathfrak g})$ is the \emph{reduced curvature} of the principal connection $A$, i.e., the curvature regarded as a 2-form on the base manifold with values in the adjoint bundle. 

The Yang--Mills Lagrangian density is a particular case of a gauge invariant Lagrangian. From the perspective of fibered actions, gauge transformations can be regarded as sections of the associated bundle $\pi_{\tilde G,X}:\tilde{G}=(P\times G)/G\to X$, where the action of $G$ on itself is by conjugation, which is a Lie group bundle with the fiberwise group structure inherited from $G$. this Lie group bundle acts on the left on $\pi_{P,X}$. In order to have a fiber action on $C(P)$, we need first derivatives of the sections of $\pi_{\tilde G,X}$, i.e., we consider its first jet bundle. More precisely, we define
\begin{equation}\label{eq:fiberedactionymj}
J^1 \tilde{G}\times_X C(P)\to C(P),\quad\left(j_x^1\gamma,[j_x^1 s]_G\right)\mapsto \left[j_x^1(\gamma\cdot s)\right]_G.
\end{equation}
Following the notation in this article, $\mathcal{G}=J^1\tilde{G}$ is the Lie group bundle of the reduction. Furthermore, the gauge invariance of a Lagrangian is understood with respect to the Lie group subbundle $H=J^2\tilde{G} \subset J^1\left(J^1 \tilde{G}\right)=J^1\mathcal G$.

Unfortunately, since the fibered action \eqref{eq:fiberedactionymj} is not free, the corresponding quotients fail to be manifolds and our results cannot be applied as they are. However, this singular situation can be fixed by extending the configuration bundle of the theory; namely, we take $\pi_{J^1 P,X}$ instead of $\pi_{C(P),X}$ as the configuration bundle. The action of $J^1\tilde{G}$ is now free and, on fact, transitive, i.e., $\left.J^1 P\right/J^1 \tilde{G}\simeq X$.

Observe that any Lagrangian $L:J^1C(P)\to \mathbb{R}$ may be lifted to a new (unreduced) Lagrangian $\hat{L}: J^1\left(J^1P\right)\to\mathbb{R}$ by setting $\hat L\left(j_x^1\hat s\right)=L\left(j^1_x[\hat{s}]_G\right)$ for each $j^1_x\hat{s}\in J^1\left(J^1P\right)$. As a straightforward consequence of \eqref{eq:quotientfulljets}, a function $\hat{L}\in C^\infty\left(J^1\left(J^1P\right)\right)$ comes from a function $L\in C^\infty\left(J^1 C(P)\right)$ if and only if it is invariant with respect to the following (right) fibered action,
\begin{equation}\label{eq:actionyangmillsj}
J^1\left(J^1 P\right)\times_X J^1(X,G)\to J^1\left(J^1 P\right),\quad\left(j^1_x\hat{s},j^1_x\tilde h\right)\mapsto j^1_x\left(\hat{s}\cdot\tilde h\right),
\end{equation}
where $J^1(X,G)$ is the jet of functions from $X$ to $G$, which coincides with the jet of the trivial bundle $X\times G\to X$.

\subsubsection{Reduced configuration space}

Since all the results concerning reduction are local, for brevity we will confine ourselves to trivializing charts. More precisely, we may assume that $P=X\times G$ with $X=\mathbb{R}^n$, whence $\tilde{G}=X\times G$ and $\tilde{\mathfrak g}=X\times\mathfrak g$. By means of the right trivialization of the tangent bundle of $G$, we can identify $J^1 P=J^1 \tilde{G}=G\ltimes(T^*X\otimes\mathfrak g)$ and $J^1\tilde{\mathfrak g}=\mathfrak g\times(T^*X\otimes\mathfrak g)$. The expression of the fibered group product with this identification is (cf. \cite[Theorem 4.2]{CaRo2023b}):
\begin{equation*}
\left(g,\xi_x\right)\cdot\left(h,A_x\right)=\left(gh,\xi_x+\Ad_g\circ A_x\right),\qquad(g,\xi_x),(h,A_x)\in J^1\tilde{G}.
\end{equation*}
This expression also holds for the fibered action of $\pi_{J^1\tilde G,X}$ on $\pi_{J^1 P,X}$. In addition, with the aid of any linear connection $\nabla^X$ on the tangent bundle of $X$, we have an identification (cf. \cite[Theorem 3.1]{CaRo2023b})
\begin{equation*}
J^1\left(J^1 P\right)=J^1\left(J^1 \tilde{G}\right)\simeq G\times(T^*X\otimes\mathfrak g)\times_X(T^*X\otimes\mathfrak g)\times_X(T^*X\otimes T^*X\otimes\mathfrak g).
\end{equation*}
In this trivialization, the extension of the dual connection of $\nabla^X$ to $\mathfrak g$-valued forms on $X$ is used, i.e., $\tilde\nabla:\Gamma(T^*X\otimes\mathfrak g)\to\Gamma(T^*X\otimes T^*X\otimes\mathfrak g)$. In this situation, the second jet bundle takes the form (cf. \cite[Corollary 5.5]{CaRo2023b}) 
\begin{equation*}
\label{Hj}
H= J^2 \tilde{G}\simeq \left\{\left(g,\xi_x;\eta_x,\phi_x\right)\in J^1\left(J^1 \tilde{G}\right):\xi_x=\eta_x,~\operatorname{Skew}(\phi_x)=-\frac{1}{2}[\xi_x,\xi_x]\right\},
\end{equation*}
where $\operatorname{Skew}:T^*X\otimes T^*X\to\bigwedge^2 T^*X$ is the skew-symmetrization map. As a Lie group connection on $\pi_{J^1\tilde G,X}$ and principal bundle connection on $\pi_{J^1 P,X}$ we choose the jet fields $\hat\nu:J^1\tilde G\to J^1(J^1\tilde G)$, $ \hat{\nu}(g,\xi_x)=(g,\xi_x;0_x,0_x)$ and $\hat\omega:J^1 P\to J^1\left(J^1 P\right)$, $\hat\omega (g,\xi_x)=(g,\xi_x;0_x,0_x)$ (see \cite[Lemma 4.1]{CaRo2023} and \cite[Proposition 4.3]{CaRo2023}, respectively). With them, we have  $\mathfrak{H}=\bigvee ^2T^*X\otimes \tilde{\mathfrak{g}}$ and we can take 
\begin{equation*}
\varphi:J^1\tilde G\to T^*X\times_X J^1\tilde{\mathfrak g},\qquad \varphi(g,\xi_x)=\left(\xi_x,-\frac{1}{2}[\xi_x,\xi_x]\right),
\end{equation*}
for $H$ in Proposition \ref{ggg}. Under these identifications, it can be checked that Theorem \ref{theorem:descomposicioncociente} reads as follows.

\begin{proposition}[Reduced space]\label{lemma:identificationRHS}
For gauge theories, the identification \eqref{ISO} is given by
\begin{equation}\label{ISOgauj}
\left.\left(J^1\left(J^1P\right)\right)\right/J^2\tilde{G}\overset{\sim}{\to}(T^*X\otimes\mathfrak{g})\oplus\left(\textstyle\bigwedge ^2T^*X\otimes\mathfrak{g}\right),\quad\left[j_x^1\hat s\right]_{J^2\tilde G}\mapsto\left(A_x,-\tilde F_x^A\right),
\end{equation}
where $A_x=[\hat s(x)]_{G}\in C(P)$ is the (pointwise) principal connection and $\tilde F_x^A\in\bigwedge^2 T^*X\otimes\mathfrak g$ denotes its \emph{reduced curvature}.
\end{proposition}

As a result, given an extended section $\hat s\in\Gamma(\pi_{J^1 P,X})$ such that $A=\pi_{J^1 P,C(P)}\circ s\in\Gamma(\pi_{C(P),X})$, then the corresponding reduced section is given by
\begin{equation*}
\overline s=(A,-\tilde F^A)\in\Gamma\left((T^*X\otimes\mathfrak g)\oplus\left(\textstyle\bigwedge^2 T^*X\otimes\mathfrak g\right)\right).
\end{equation*}

\subsubsection{Utiyama's theorem and Yang--Mills equations}

From Proposition \ref{lemma:identificationRHS}, the reduced Lagrangian $l$ depends on two arguments: $(A_x,-F_x)$. However, we now have to take into account that the unreduced Lagrangian $L$ was in fact defined in $J^1(C(P))$ and not in $J^1(J^1P)$. In other words, we have to impose that $l$ is invariant by the action of $J^1(X,G)$. After transferring this (right) action by \eqref{ISOgauj}, we obtain
\begin{equation}
\label{rightinvariance}
\left(A_x,\tilde F_x\right)\cdot(h,\mu_x)=\left(\Ad_{h^{-1}}\circ(A_x+\mu_x),\Ad_{h^{-1}}\circ \tilde F_x\right),
\end{equation}
for each $\left(A_x,\tilde F_x\right)\in(T^*X\otimes\mathfrak g)\times_X\left(\bigwedge^2 T^*X\otimes\mathfrak g\right)$ and $(h,\mu_x)\in G\ltimes(T^*X\otimes\mathfrak g)$. If $l$ must be invariant invariant with respect to this action, then $l$ must not depend on $A_x$ and the dependence on $\tilde{F}_x$ is adjoint invariant. Utiyama's theorem is now straightforward.

\begin{corollary}[Utiyama's theorem]\label{cor:utiyama}
A Lagrangian $L:J^1 C(P)\to\mathbb R$ is gauge invariant if and only if $L=l\circ\tilde F$, where
\begin{equation*}
\tilde F:J^1 C(P)\to\textstyle\bigwedge^2 T^*X\otimes\tilde{\mathfrak g},\quad j_x^1 A\mapsto\tilde F_x^A,
\end{equation*}
is the \emph{reduced curvature map}, and $l:\textstyle\bigwedge^2 T^*X\otimes\tilde{\mathfrak g}\to\mathbb R$ is an $\Ad$-invariant function.
\end{corollary}

This represents an alternative proof of the Utiyama's theorem that relies on the gauge reduction theory developed in this article. Roughly speaking, this new approach takes advantage of the gauge invariance of the Lagrangian density to transfer it to the corresponding fibered quotient and, by means of the identification \eqref{ISOgauj}, this quotient space is seen to be nothing but the curvature bundle of the theory. Of course, the $\Ad$-invariance of the reduced Lagrangian is also taken into account from our perspective.

Finally, we explore the reduced equations. It can be checked easily that the dual operator of the map \eqref{eq:tildevarphi} is given by
\begin{equation*}
\varphi_*^\dagger:TX\otimes\left(J^1\tilde{\mathfrak g}\right)^*\to\left(J^1\tilde{\mathfrak g}\right)^*,\quad(\alpha_x,\beta_x)\mapsto(0,\alpha_x).
\end{equation*}
On the other hand, the coadjoint representation given in Lemma \ref{lemma:restrictedcoadjoint} reads
\begin{equation*}
\ad_{\overline s}^*:{\mathfrak H}^\circ\to\left(J^1\tilde{\mathfrak g}\right)^*,\quad(\varpi_x,\Pi_x)\mapsto\left(\ad_{A_x}^*\circ\varpi_x-\ad_{\tilde F_x^A}^*\circ\Pi_x,~\ad_{A_x}^*\circ\Pi_x\right),
\end{equation*}
where ${\mathfrak H}^\circ=(TX\otimes\mathfrak g^*)\times_X\left(\bigwedge^2 TX\otimes\mathfrak g^*\right)$. Then one can check that the reduced equations for a section $(A,-\tilde{F})$ of $(T^*X\otimes \mathfrak{g} )\oplus (\bigwedge ^2 T^*X\otimes \mathfrak{g})$ are
\begin{eqnarray}\label{eq:redeqym1}
\ad_{\tilde F^A}^*\left(\frac{\delta l}{\delta\tilde F}\right)=0,\\\label{eq:redeqym2}
\left(\tilde{\operatorname{div}}-\ad_A^*\right)\left(\frac{\delta l}{\delta\tilde F}\right)=0,
\end{eqnarray}
where we have taken into account that $\partial l /\partial A=0$ from \eqref{rightinvariance}, and $\tilde{\operatorname{div}}:\Gamma(TX\otimes TX\otimes\mathfrak g^*)\to\Gamma(TX\otimes\mathfrak g^*)$ denotes the divergence of $\tilde\nabla$. But equation \eqref{eq:redeqym1} is already satisfied by the adjoint invariance of $l$ described by the other condition provided by \eqref{rightinvariance}. We have thus the single equation \eqref{eq:redeqym2}. Note that, even though $l$ does not depend on $A$, this variable is still a variational field of the reduced problem and it appears in the equations.

In particular, for the Yang--Mills Lagrangian $l =\parallel F\parallel^2$ we recover equation $\star({\rm d}\star F+[A,\star F])=0$, which together with the reconstruction equation ${\rm d}+[A,F]=0$ are the Yang--Mills equations.

%%%%%%%%%%%%%%%%%%%%%%%%%%%%%%%%%%%%%%%%%%%%%%%%%
\paragraph{Acknowledgments.}

MCL and ARA have been partially supported by Ministerio de Ciencia e Innovación (Spain), under grants PGC2018-098321-B-I00 and PID2021-126124NB-I00. ARA has been supported by Ministerio de Universidades (Spain) under an FPU grant.

%%%%%%%%%%%%%%%%%%%%%%%%%%%%%%%%%%%%%%%%%%%%%%%%%%
\bibliographystyle{plain}
\bibliography{biblio_reduction.bib}

\end{document}